\documentclass{amsart}
\usepackage{amsmath,amscd,amssymb,combelow,color,enumitem,graphicx,float,comment,mathtools,mathrsfs,appendix,hyperref}
\usepackage[all]{xy}
\usepackage{tikz}
\usetikzlibrary{calc,arrows.meta, positioning}
\usepackage{refcount}

\newcommand{\primeeqtag}[1]{\tag{\getrefnumber{#1}'}}

\newtheorem{theorem}{Theorem}[section]

\newtheorem{proposition}[theorem]{Proposition}
\newtheorem{lemma}[theorem]{Lemma}
\newtheorem{corollary}[theorem]{Corollary}
\newtheorem{conjecture}[theorem]{Conjecture}
\newtheorem{definition}[theorem]{Definition}
\newtheorem{claim}[theorem]{Claim}
\newtheorem{example}[theorem]{Example}
\newtheorem{remark}[theorem]{Remark}

\def\loccitt{\emph{loc. cit.}}

\def\fg{{\mathfrak{g}}}

\def\fm{{\mathfrak{m}}}

\def\fS{{\mathfrak{S}}}

\def\BC{{\mathbb{C}}}

\def\BK{{\mathbb{K}}}

\def\BN{{\mathbb{N}}}

\def\BZ{{\mathbb{Z}}}

\def\CO{{\mathcal{O}}}

\def\CS{{\mathcal{S}}}

\def\CV{{\mathcal{V}}}

\def\ph{\varphi}

\def\vs{\varsigma}

\def\and{\textrm{ }\&\textrm{ }}

\def\eSym{\emph{Sym}}

\def\oI{\bar{I}}

\def\oi{\bar{i}}
\def\oj{\bar{j}}

\def\oCS{\bar{\CS}}
\def\oCV{\bar{\CV}}

\def\oUU{\bar{\UU}}
\def\oUUpm{\bar{\UU}^\pm}
\def\oUUp{\bar{\UU}^+}
\def\oUUm{\bar{\UU}^-}

\def\oalpha{\bar{\alpha}}
\def\ozeta{\bar{\zeta}}

\def\CC{{{\mathcal{C}}}}

\def\nn{{{\BN}}^I}
\def\zz{{{\BZ}}^I}

\def\twU{U_q(L\fg^\sigma)}

\def\UU{\mathbf{U}}
\def\UUo{\mathbf{U}^0}
\def\UUp{\mathbf{U}^+}
\def\UUpm{\mathbf{U}^\pm}
\def\UUm{\mathbf{U}^-}
\def\UUg{\mathbf{U}^\geq}
\def\UUl{\mathbf{U}^\leq}

\def\tUU{{\widetilde{\mathbf{U}}}}
\def\tUUp{{\widetilde{\mathbf{U}}^+}}
\def\tUUpm{{\widetilde{\mathbf{U}}^\pm}}

\def\tUUm{{\widetilde{\mathbf{U}}^-}}

\def\bs{{\boldsymbol{\vs}}}

\def\b0{{\boldsymbol{0}}}

\def\oQ{{\bar{Q}}}

\def\bm{\boldsymbol{m}}
\def\bn{\boldsymbol{n}}

\def\bx{\boldsymbol{x}}

\def\tUpsilon{\widetilde{\Upsilon}}

\def\fold{\rho}

\def\opsi{\bar{\psi}}
\def\bpsi{\boldsymbol{\psi}}
\def\obpsi{\bar{\bpsi}}
\def\bord{\textbf{ord }}

\def\vac{|\varnothing \rangle}

\def\on{\bar{n}}
\def\obn{\bar{\bn}}

\def\onn{\BN^{\oI}}
\def\ozz{\BZ^{\oI}}

\def\oR{\bar{R}}
\def\ofS{\bar{\fS}}

\def\ox{\bar{x}}
\def\obx{\bar{\bx}}

\def\oUpsilon{\bar{\Upsilon}}

\def\fS{{\mathfrak{S}}}

\def\sh{\text{sh}}
\def\esh{\emph{sh}}

\begin{document}
	
	\title[FOLDING SHUFFLE ALGEBRAS AND TWISTED $q$-CHARACTERS]{\Large{\textbf{Folding shuffle algebras and twisted $q$-characters}}}
	
	\author[Andrei Negu\cb t and Keyu Wang]{Andrei Negu\cb t and Keyu Wang}
	
	\address{École Polytechnique Fédérale de Lausanne (EPFL), Lausanne, Switzerland \newline \text{ } \ \ Simion Stoilow Institute of Mathematics (IMAR), Bucharest, Romania} 
	\email{andrei.negut@gmail.com}
	
	\address{Faculty of Mathematics, University of Vienna, Vienna, Austria}
	\email{keyu.wang@univie.ac.at}
	
	\def\uppercasenonmath#1{}
	
	\begin{abstract} 
		
		Using our new notion of folding shuffle algebras, we prove a conjecture of Hernandez on the equality between certain $q$-characters of quantum untwisted affine algebra modules and their twisted counterparts. We generalize this result to the setting of arbitrary quivers with automorphisms, in particular by defining and describing twisted quantum toroidal algebras. 
		
	\end{abstract}
	\maketitle
	
	\section{Introduction}
	
	\subsection{Motivation}
	\label{sub:motivation}
	
	Let $\fg$ be a simple finite-dimensional Lie algebra of type ADE, with set of simple roots $I$. The untwisted quantum affine algebra $U_q(L\fg)$ has been widely studied over the past several decades. In particular, the Drinfeld new realization of this algebra has spurred significant development in its representation theory, for instance with the introduction of Drinfeld polynomials and $q$-characters.
	
	\medskip
	
	\noindent By contrast, the twisted quantum affine algebra $U_q(L\fg^\sigma)$ associated to an automorphism $\sigma$ of the Dynkin diagram of $\fg$, is technically more difficult in the Drinfeld new realization. Indeed, while $L \fg^\sigma$ is a Lie subalgebra of $L\fg$, no such simple relation exists between the $q$-versions of these Lie algebras, and thus the relationship between their categories of modules remains mysterious. An interesting and long-standing question has been to transfer results from the untwisted to the twisted setup. 
	
	\medskip
	
	\noindent One important direction of study concerns the precise connection between the Grothendieck rings of categories of modules. More precisely, the following conjectural relation was proposed in \cite{H KR}. Let $\CC_{\BZ}$ and $\CC^\sigma_{\BZ}$ denote the categories of finite-dimensional modules of $U_q(L\fg)$ and $U_q(L\fg^\sigma)$, respectively. Using the theory of $q$-characters developed in \cite{FR}, Hernandez constructed in \cite{H KR} a ring isomorphism 
	\begin{equation}
		\label{eqn:iso intro}
		\pi: K_0\big(\CC_{\BZ}\big) \xrightarrow{\sim} K_0\big(\CC^{\sigma}_{\BZ}\big)
	\end{equation}
	Explicitly, if we denote by $V_{i,a}$ the fundamental modules of either $U_q(L\fg)$ or $U_q(L\fg^\sigma)$, then we obtain the following identification (that goes back to the work of \cite{CP}) 
	\[
	K_0\big(\CC_{\BZ}\big) \simeq \BC\big[[V_{i,a}]~|~i \in I, a \in \BZ \big]
	\]
	and analogously for $\CC^\sigma_{\BZ}$.
	In this presentation, the map $\pi$ simply sends the classes of fundamental modules to classes of fundamental modules. However, it is unclear how the classes of more general simple modules behave under $\pi$. Hernandez conjectured (see \cite{H KR} and \cite[Conjecture~2.20]{W QQ}) that $\pi$ maps simple classes to simple classes. More precisely, the class of the simple module $L(\bpsi)$ classified by a loop weight
	$$
	\bpsi \in \left(\BC[[z^{-1}]]^\times\right)^{I}
	$$
	should be mapped to the class of the simple module $L(\fold(\bpsi))$, where $\fold(\bpsi)$ is described explicitly in \eqref{eqn:fold intro 1}-\eqref{eqn:fold intro 2}. The aforementioned conjecture is equivalent to the fact that the folding map $\fold$ sends $\chi_q(L(\bpsi))$ to $\chi_q(L(\fold(\bpsi)))$ (see Conjecture~\ref{conj:main}). This conjecture was proved in \cite{H KR} for Kirillov--Reshetikhin modules, and more recently in \cite{FQ} for any simple module in the category $\CC_{\BZ}$ for $\fg$ of type $A$ or $D$.
	
	\medskip
	
	\noindent In the present paper, we provide a uniform proof of Hernandez' conjecture. Perhaps more importantly, we generalize and systematize it by:
	
	\medskip
	
	\begin{itemize}[leftmargin=*]
		
		\item extending the result to more general simple modules, that includes those in the category $\CO_{\BZ}$ defined in \cite{H Fusion} (which is denoted by $\widehat{\CO}_{\BZ}$ in \cite{MY}), as well as the category  $\CO^{\sh}_{\BZ}$ of shifted quantum affine algebra modules defined in \cite{H Shifted};
		
		\medskip
		
		\item extending the result to general quantum loop algebras, including those associated with arbitrary symmetrizable Kac--Moody Lie algebras $\fg$, as well as other algebras that have recently garnered much attention (such as $K$--theoretic Hall algebras of quivers and BPS algebras associated to toric Calabi--Yau threefolds).
		
	\end{itemize}  
	
	\subsection{\texorpdfstring{$q$}{q}-characters}
	\label{sub:q char intro}
	
	Fix a parameter $q \in \BC^* \backslash \sqrt[\BN]{1}$ and a finite-dimensional complex simple (henceforth called ``finite type") Lie algebra $\fg$ endowed with an automorphism $\sigma$ of its Dynkin diagram. Let $m = \text{ord }\sigma$. We will denote by $I$ the set of simple roots of $\fg$, and by $\oI$ the set of $\sigma$ orbits of $I$. Drinfeld constructed a realization
	\begin{equation}
		\label{eqn:twisted intro}
		U_q(L\fg^\sigma) \cong \BC \Big \langle e_{i,d},f_{i,d},\ph^\pm_{i,d'} \Big \rangle_{i \in I, d \in \BZ, d' \geq 0} \Big / \Big(\text{relations}\Big)
	\end{equation}
	of the $q$-deformed universal enveloping algebra of $\widehat{\fg}^{\sigma}$ (see \cite{D, Dr}), which we will recall in Subsection~\ref{sub:affine}. We call \eqref{eqn:twisted intro} a twisted quantum affine algebra, and refer to the case $\sigma = \text{Id}$ as the untwisted quantum affine algebra $U_q(L\fg)$. The representation theory of the algebras above was studied in \cite{CP, CP tw}. In a nutshell, finite-dimensional simple modules of the untwisted $U_q(L\fg)$ are indexed by their highest loop weight
	\begin{equation}
		\label{eqn:loop weight intro}
		\bpsi = (\psi_i(z))_{i \in I} \in \left(\BC[[z^{-1}]]^\times\right)^{I}
	\end{equation}
	Conversely, for any $\bpsi$ as above, a simple module $L(\bpsi)$ for the Borel subalgebra of the quantum affine algebra was constructed in \cite{HJ}, and a simple module for the so-called shifted quantum affine algebra was constructed in \cite{H Shifted} (we review these notions in Section~\ref{sec:modules}). Following \cite{FR}, the $q$-character of any module $V$ of either the Borel or the shifted version of $U_q(L\fg)$ is defined as 
	\begin{equation}
		\label{eqn:q-character intro}
		\chi_q(V) = \sum_{\bpsi} [\bpsi] \dim_{\BC} \left(V_{\bpsi}\right) 
	\end{equation}
	where $V_{\bpsi}$ is the generalized eigenspace of the commuting operators $\{\ph^+_{i,d'}\}_{i\in I, d' \geq 0}$ corresponding to eigenvalues encoded by the loop weight $\bpsi$, see \eqref{eqn:q char}. The notions above also make sense for twisted quantum affine algebras, in which case we will denote them by adding bars on top of the relevant symbols. There is a folding map
	\begin{equation}
		\label{eqn:fold intro 1}
		\fold :  \left(\BC[[z^{-1}]]^\times\right)^{I} \rightarrow  \left(\BC[[z^{-1}]]^\times\right)^{\oI}
	\end{equation}
	defined as $\fold(\bpsi) = \obpsi$ with $\bpsi = (\psi_i(z))_{i \in I}$ and $\obpsi = (\opsi_{\oi}(z))_{\oi \in \oI}$ related by
	\begin{equation}
		\label{eqn:fold intro 2}
		\opsi_{\oi}(z) = \prod_{k=1}^m \psi_{\sigma^k(i_0)}(z\omega^k)
	\end{equation}
	where we choose a distinguished preimage $i_0 \in I$ of any $\oi \in \oI$, and $\omega$ denotes a primitive $m$-th root of unity. We now present the key relation between $q$-characters of the untwisted and twisted algebras, conjectured in \cite{H KR}, see also \cite[Conjecture 2.20]{W QQ} for the general statement. 
	
	\medskip
	
	\begin{conjecture}
		\label{conj:main}
		
		Let $\fg$ be any finite type Lie algebra with an automorphism $\sigma$. If 
		$$
		\bpsi = Y_{i_1,x_1} \dots Y_{i_n,x_n}
		$$
		for various $i_1,\dots,i_n \in I$ and $x_1,\dots,x_n \in q^{\BZ}$ (here use use the standard notation for fundamental loop weights: $Y_{i,x}$ denotes the $I$-tuple of power series equal to 
		\[
		\frac {zq_i - xq_i^{-1}}{z-x}
		\]
		on the $i$-th spot and 1 everywhere else, where $q_i = q^{d_i}$ with $d_i \in \{1,2,3\}$), then
		\[
		\fold(\chi_q(L(\bpsi)) = \chi_q(L(\fold(\bpsi)))
		\]
		where $L(\bpsi)$ and $L(\rho(\bpsi))$ are simple modules of $U_q(L\fg)$ and $U_q(L\fg^\sigma)$, respectively.
		
	\end{conjecture}
	
	\medskip 
	
	\subsection{Shuffle algebras}
	\label{sub:shuffle intro}
	
	Our main tool for proving Conjecture \ref{conj:main}, as well as the natural vehicle of its generalization, is the shuffle algebra interpretation of quantum loop algebras. In more detail, consider any quiver $Q$ with parameters in any ground field $\BK$. We denote the vertex set of the quiver by $I$ and assume that for every arrow
	\[
	\alpha : i \rightarrow j, \text{ we are given a parameter } q_{\alpha} \in \BK^*
	\]
	We may use these parameters to construct the zeta functions
	\begin{equation}
		\label{eqn:zeta intro}
		\left\{ \zeta_{ij} (x) \sim \frac {\prod_{\alpha : i \rightarrow j}(1-xq_\alpha)}{(1-x)^{\delta_{ij}}} \right\}_{i,j \in I}
	\end{equation}
	where $\sim$ means equality up to a Laurent monomial in $x$ (which will not be important to us). Given the data \eqref{eqn:zeta intro}, we recall in Subsections~\ref{sub:zeta}-\ref{sub:upsilon} the algebras
	\[
	\Big(\text{half quantum loop algebras }\UUpm \Big) \cong \Big(\text{shuffle algebras }\CS^\pm \Big)
	\]
	following \cite{N Arbitrary}. We show in Subsections~\ref{sub:pairing}-\ref{sub:double} how to construct the double
	\[
	\UU = \UUp \otimes \BK[\ph^\pm_{i,d}]_{i \in I, d \geq 0}\otimes \UUm
	\]
	assuming that the zeta functions are balanced, in the sense of \eqref{eqn:limit}. In Subsection \ref{sub:simple modules}, we recall how to define the shifted versions $\UU^{\sh}$ of the above double algebras. As shown in \cite{NT} following \cite{E,FO}, the untwisted quantum loop algebra $U_q(L\fg)$ with symmetrized Cartan matrix $\{d_{ij}\}_{i,j \in I}$ is isomorphic to the algebra $\UU$ for the data
	\begin{equation}
		\label{eqn:zeta intro particular}
		\left\{ \zeta_{ij}(x) \sim \frac {1-xq^{d_{ij}}}{(1-x)^{\delta_{ij}}} \right\}_{i,j \in I}
	\end{equation}
	modulo the additional relation $\ph_{i,0}^+\ph_{i,0}^-=1$ for all $i \in I$.
	
	\subsection{The theorem}
	\label{sub:general}
	
	We establish our main result (the proof of Conjecture \ref{conj:main}) in two related settings, which we respectively call ``folding" and ``covering" quivers.
	
	\subsubsection{Folding quivers}
	\label{sub:folding intro}
	
	\noindent Assume that the quiver $Q$ in Subsection \ref{sub:shuffle intro} has an automorphism $\sigma$ of order $m$, which preserves the arrow parameters in the sense that $q_{\alpha} = q_{\sigma(\alpha)}$ for every arrow $\alpha$. We call \eqref{eqn:zeta intro} the \emph{unfolded zeta functions}, and we will consider in Subsection~\ref{sub:quivers automorphism} the following \emph{folded zeta functions}:
	\begin{equation}
		\label{eqn:folded zeta intro}
		\left \{ \ozeta_{\oi \oj}(x) \sim  \frac {\prod_{k = 1}^{\frac m{m_{\oi}}}\prod_{\alpha : \sigma^k(i_0) \rightarrow j_0}(1-x^{m_{\oi}} \omega^{km_{\oi}} q^{m_{\oi}}_\alpha)}{(1-x^{m_{\oi}})^{\delta_{\oi \oj}}} \right\}_{ \oi,\oj \in \oI}
	\end{equation}
	where $\oI = I/\sigma$ and $i_0,j_0 \in I$ denote distinguished preimages of $\oi,\oj \in \oI$. In the equation above, $\omega$ denotes a primitive $m$-th root of unity, and we define the numbers
	\[
	m_{\oi} = \frac m{\min\{a \geq 1|\sigma^a(i_0) = i_0\}}
	\]
	The following is our main result (see Theorem~\ref{thm:main folding} for the full details).
	
	\medskip 
	
	\begin{theorem}
		\label{thm:main folding intro}
		
		For any quiver $Q$ with an automorphism $\sigma$ as above, consider the shifted quantum loop algebra modules $\UU^{\esh} \curvearrowright L(\bpsi)$ and $\oUU^{\esh} \curvearrowright L(\fold(\bpsi))$ associated to the folded and unfolded zeta functions \eqref{eqn:zeta intro} and \eqref{eqn:folded zeta intro}, respectively. If $\bpsi$ is an acceptable (in the sense of Definition~\ref{def:acceptable folding}) highest loop weight, then
		\[
		\fold(\chi_q(L(\bpsi)) = \chi_q(L(\fold(\bpsi)))
		\]
		
	\end{theorem}
	
	\medskip 
	
	\subsubsection{Coverings of quivers}
	\label{sub:covering intro}
	
	The following setting generalizes the situation in the previous Subsubsection when the automorphism acts freely on the vertex set of the quiver (see Remark~\ref{rem:intersection}). In general, we define a covering of quivers to be a map
	\begin{equation}
		\label{eqn:covering quivers intro}
		\fold : Q \rightarrow \oQ 
	\end{equation}
	whose underlying function of vertex sets $\fold : I \rightarrow \oI$ is surjective and has the property that above every arrow of $\oQ$ with any target $\oj$, there exists a unique arrow in $Q$ with target any given preimage $j \in \fold^{-1}(\oj)$:
	\[
	\forall \oalpha : \oi \rightarrow \oj \text{ and } j \in \fold^{-1}(\oj), \ \exists! i \in \fold^{-1}(\oi) \text{ and } \alpha : i \rightarrow j \text{ s.t. } \fold(\alpha) = \oalpha
	\]
	Any choice of parameters $\{q_{\oalpha} \in \BK^*\}_{\oalpha \text{ arrow of }\oQ}$ uniquely induces a choice of parameters $q_\alpha \in \BK^*$ for every arrow of $Q$ via the map $\fold$, see \eqref{eqn:matching parameters}. Thus, we refer to the functions \eqref{eqn:zeta intro} as the \emph{upper zeta functions} and define the \emph{lower zeta functions} as
	\[
	\left\{\ozeta_{\oi \oj}(x) \sim \frac {\prod_{\oalpha : \oi \rightarrow \oj} (1-xq_{\oalpha})}{(1-x)^{\delta_{\oi \oj}}}\right\}_{\oi,\oj \in \oI}
	\]
	The following is an analogue of Theorem~\ref{thm:main folding intro} (see Theorem~\ref{thm:main covering} for the full details).
	
	\medskip 
	
	\begin{theorem}
		\label{thm:main covering intro}
		
		For any covering of quivers \eqref{eqn:covering quivers intro}, consider the shifted quantum loop algebra modules $\UU^{\esh} \curvearrowright L(\bpsi)$ and $\oUU^{\esh} \curvearrowright L(\fold(\bpsi))$ associated to the upper and lower zeta functions, respectively. The folding map $\rho$ on loop weights is defined in \eqref{eqn:loop weight covering}. If $\bpsi$ is an acceptable (in the sense of Definition~\ref{def:acceptable covering}) highest loop weight, then
		\[
		\fold(\chi_q(L(\bpsi)) = \chi_q(L(\fold(\bpsi)))
		\]
		
	\end{theorem}
	
	\medskip 
	
	\subsection{Main example: quantum affinizations of Kac--Moody Lie algebras}
	\label{sub:strongly intro}
	
	Let $\fg$ be the Kac--Moody Lie algebra associated to any generalized Cartan matrix 
	\[
	C = \left(\frac {2d_{ij}}{d_{ii}}  \in \BZ\right)_{i,j \in I}
	\]
	We call $\fg$ \emph{strongly symmetrizable} if 
	\[
	d_{ij} \in \left\{0,-\max\left(\frac {d_{ii}}2, \frac {d_{jj}}2 \right) \right\}
	\]
	for all $i \neq j$. Consider any strongly symmetrizable $\fg$ equipped with an automorphism $\sigma : I \rightarrow I$ of order $m$, by which we mean that $d_{ij} = d_{\sigma(i)\sigma(j)}$ for all $i,j \in I$. Given these data, we define the corresponding twisted quantum loop algebra as:
	\[
	U_q(L\fg^{\sigma}) = \BC \Big \langle e_{i,d}, f_{i,d}, \ph^\pm_{i,d'} \Big \rangle_{i \in I, d \in \BZ, d' \geq 0} \Big/ \Big(\text{relations \eqref{eqn:rel quantum 0 twisted intro}-\eqref{eqn:rel quantum 9 twisted intro}} \Big) 
	\]
	where for $e_i(z) = \sum_{d \in \BZ} \frac {e_{i,d}}{z^d}$, $f_i(z) = \sum_{d \in \BZ} \frac {f_{i,d}}{z^d}$, $\ph^\pm_i(z) = \sum_{d = 0}^{\infty} \frac {\ph^\pm_{i,d}}{z^{\pm d}}$, we set
	\begin{equation}
		\label{eqn:rel quantum 0 twisted intro}
		e_{\bigcirc} = f_{\bigcirc} = 0 \ \footnote{We refer to $e_\bigcirc, f_\bigcirc$ as the ``clunky relations", which were introduced in \cite{N Reduced}, cf. \eqref{eqn:clunky}.} \text{as }\bigcirc \text{ runs over the wheels \eqref{eqn:twisted wheel strongly}} 
	\end{equation} 
	\begin{equation}
		\label{eqn:rel quantum 1 twisted intro}
		e_{\sigma(i),d} = e_{i,d} \omega^d, f_{\sigma(i),d} = f_{i,d} \omega^d, \ph^\pm_{\sigma(i),d} = \ph^\pm_{i,d} \omega^{\pm d}
	\end{equation} 
	\begin{equation}
		\label{eqn:rel quantum 2 twisted intro}
		e_i(x) e_j(y) \prod_{k=1}^m \left( x\omega^k - y q^{d_{\sigma^k(i)j}} \right) = e_j(y) e_i(x) \prod_{k=1}^m \left( x\omega^k q^{d_{\sigma^k(i)j}} - y \right)
	\end{equation}
	\begin{equation}
		\label{eqn:rel quantum 3 twisted intro}
		f_j(y) f_i(x) \prod_{k=1}^m \left( x\omega^k - y q^{d_{\sigma^k(i)j}} \right) = f_i(x) f_j(y) \prod_{k=1}^m \left( x\omega^k q^{d_{\sigma^k(i)j}} - y \right)
	\end{equation}
	\begin{equation}
		\label{eqn:rel quantum 4 twisted intro}
		\left[\ph^+_{i}(x), \ph^+_{j}(y)\right] = \left[\ph^+_{i}(x), \ph^-_{j}(y)\right] = \left[\ph^-_{i}(x), \ph^-_{j}(y)\right] = 0, \ \ph^+_{i,0}\ph^-_{i,0} = 1
	\end{equation} 
	\begin{align}
		&\ph^\pm_i(x) e_j(y) = e_j(y) \ph^\pm_i(x) \prod_{k=1}^m \frac {x\omega^k q^{d_{\sigma^k(i)j}}-y}{x\omega^k-y q^{d_{\sigma^k(i)j}}} \label{eqn:rel quantum 5 twisted intro} \\
		&f_j(y)\ph^\pm_i(x)  = \ph^\pm_i(x)  f_j(y) \prod_{k=1}^m \frac {x\omega^k q^{d_{\sigma^k(i)j}}-y}{x\omega^k-y q^{d_{\sigma^k(i)j}}} \label{eqn:rel quantum 7 twisted intro} 
	\end{align}
	\begin{equation}
		\label{eqn:rel quantum 9 twisted intro}
		\Big[e_{i}(x), f_{j}(y)\Big] = \left( \sum_{k=1}^m \frac {\delta_{\sigma^k(i)j}\delta(\frac{x\omega^k}{y})}{m_i} \right) \frac {\ph^+_i(x) - \ph^-_j(y)}{q_i-q_i^{-1}}
	\end{equation} 
	for all $i,j \in I$, where the numbers $m_i$ are defined in \eqref{eqn:m i}. Since strongly symmetrizable Kac--Moody Lie algebras include all finite type and affine Lie algebras (except for affine $A_1$), the formulas above both pertain to twisted quantum affine algebras and provide a definition of twisted quantum toroidal algebras. 
	
	\medskip
	
	\noindent More generally, one can define $U_q(L\fg^\sigma)$ for any symmetrizable Kac--Moody $\fg$ (by applying the general construction of Subsection \ref{sub:shuffle intro} to the zeta functions \eqref{eqn:zeta strongly twisted}), but explicit analogues of relations \eqref{eqn:rel quantum 0 twisted intro} are not known outside the strongly symmetrizable case. For $\fg$ of finite type, we prove the equivalence between the aforementioned relations and the usual Drinfeld--Serre relations in Proposition~\ref{prop:affine}. For affine $\fg$, we will develop Drinfeld--Serre-like versions of relations \eqref{eqn:rel quantum 0 twisted intro} in Proposition~\ref{prop:toroidal}.
	
	\medskip
	
	\noindent The following result generalizes Conjecture~\ref{conj:main} to the setting of symmetrizable Kac--Moody Lie algebras (it will be derived from Theorem \ref{thm:main folding intro} in Corollary~\ref{cor:main}).
	
	\medskip
	
	\begin{theorem}
		\label{thm:main}
		
		Let $\fg$ be any symmetrizable Kac--Moody Lie algebra with an automorphism $\sigma$. If $\bpsi$ is an acceptable highest loop weight, then we have
		\[
		\fold(\chi_q(L(\bpsi)) = \chi_q(L(\fold(\bpsi)))
		\]
		where $U_q(L\fg)^{\esh} \curvearrowright L(\bpsi)$ and $U_q(L\fg^{\sigma})^{\esh} \curvearrowright L(\fold(\bpsi))$ are the simple modules (of the appropriately shifted quantum algebras) defined in Subsections~\ref{sub:simple modules} and \ref{sub:simple folding}, respectively. The word ``acceptable" is explained in Definition~\ref{def:acceptable folding}; we remark that it includes all simple modules in the category $\CO^{\esh}_{\BZ}$, as explained in Corollary 4.8.
		
	\end{theorem}
	
	\medskip
	
	\subsection{Applications and further directions}
	\label{sub:applications}
	
	The main consequence of Theorem \ref{thm:main} is that many results established for Grothendieck rings of untwisted quantum affine algebras can be transferred directly to the twisted case, including the monoidal categorification of cluster algebras via $K_0(\CO^{\sh}_{\BZ})$ \cite{GHL} and the extended $QQ$-systems \cite{FH Extended}. Our constructions and results also apply to the following problems:
	
	\medskip
	
	\begin{itemize}[leftmargin=*]
		
		\item relating Langlands dual quantum affine algebras \cite{FH Langlands,KO}, since the Langlands dual of $U_q(L\fg)$ for non simply-laced $\fg$ will be a twisted quantum affine algebra;
		
		\medskip
		
		\item solving quantum integrable models associated with twisted affine algebras, as their twisted $q$-characters are closely related to the transfer matrices of quantum integrable models, see for example \cite{FHR};
		
		\medskip
		
		\item shedding light on the geometric representation theory of foldings/coverings of quivers, following in the footsteps of \cite{Nak};
		
		\medskip
		
		\item developing a functorial relation between the category $\CO^{\sh}_{\BZ}$ and its twisted counterpart, which induces the isomorphism of Grothendieck rings \eqref{eqn:iso intro}. Note that for a certain subcategory $\CC^0_{\BZ}$ of $\CC_{\BZ}$ in type $A$, such a relation is studied in \cite{KKKO,KKO} using the quantum affine Schur-Weyl duality.
		
	\end{itemize}
	
	\medskip 
	
	\subsection{Plan of the paper}
	\label{sub:plan intro}
	
	In Section~\ref{sec:shuffle}, we 
	
	\medskip 
	
	\begin{itemize}[leftmargin=*]
		
		\item recall the definition and basic structure of quantum loop algebras $\UU$ and shuffle algebras $\CS^\pm$ associated to any quiver with parameters;
		
		\medskip 
		
		\item apply these notions to the case of symmetrizable Kac--Moody Lie algebras.
		
	\end{itemize}
	
	\medskip 
	
	\noindent In Section~\ref{sec:folding}, we  
	
	\medskip 
	
	\begin{itemize}[leftmargin=*]
		
		\item define the notion of folding shuffle algebras in the generality of Subsection~\ref{sub:folding intro};
		
		\medskip
		
		\item in the Kac--Moody case, state Proposition~\ref{prop:affine} relating our definition with that of twisted quantum affine algebras, and Proposition~\ref{prop:toroidal} on its toroidal analogue;
		
		\medskip  
		
		\item define the notion of coverings of quivers in the generality of Subsection~\ref{sub:covering intro}.
		
	\end{itemize}
	
	\noindent In Section~\ref{sec:modules}, we
	
	\medskip

	\begin{itemize}[leftmargin=*]
		
		\item recall the shuffle algebra incarnation of simple modules of (un)twisted quantum loop algebras and their $q$-characters from \cite{HN, N Cat}, following \cite{CP, FR, HJ};
		
		\medskip
		
		\item prove the variant Theorem~\ref{thm:main covering intro} first, and then adapt its proof to obtain Theorem~\ref{thm:main folding intro}. As we will explain in Corollary~\ref{cor:main}, this implies Theorem~\ref{thm:main}.
		
	\end{itemize}
	
	\medskip
	
	\noindent In the Appendix, we prove Propositions~\ref{prop:affine} and \ref{prop:toroidal}, thus relating our definition of twisted quantum affine/toroidal algebras with the ones more commonly found in the literature. This has the advantage of making the (rather mysterious at first) relations \eqref{eqn:rel quantum 0 twisted intro} completely explicit in affine/toroidal types.
	
	\medskip 
	
	\subsection{Acknowledgements} The authors would like to thank David Hernandez for his great help in our understanding (twisted and untwisted) quantum affine algebras, and Alexander Tsymbaliuk for many very useful conversations on shuffle algebras. We gratefully acknowledge the support of the Swiss National Science Foundation grant 10005316. The second-named author is partially supported by the Austrian Science Fund (FWF): PAT 9039323, Grant DOI 10.55776/PAT9039323.
	
	\bigskip
	
	\section{Shuffle algebras} 
	\label{sec:shuffle}
	
	\medskip
	
	\noindent We recall the basic theory of shuffle algebras associated to arbitrary quivers with parameters, which was developed in \cite{N Arbitrary} based on the original construction of \cite{E, FO}. 
	
	\medskip 
	
	\begin{itemize}[leftmargin=*]
		
		\item In Subsections~\ref{sub:zeta}-\ref{sub:double}, we present the dual theories of quantum loop algebras and shuffle algebras associated to quivers with parameters.
		
		\medskip 
		
		\item In Subsections~\ref{sub:Kac-Moody}-\ref{sub:square roots}, we apply the general theory above to symmetrizable Kac-Moody Lie algebras $\fg$ (generalizing the case of finite type $\fg$ from \cite{Dr}).
		
	\end{itemize}
	
	\medskip 
	
	\subsection{Quivers and zeta functions} 
	\label{sub:zeta}
	
	The set $\BN$ will contain 0 in the present paper, and we will work over an arbitrary field $\BK$ of characteristic 0. Consider a quiver $Q$ with vertex set $I$. Assume that every quiver arrow
	\[
	\alpha : i \rightarrow j
	\]
	is endowed with a parameter $q_{\alpha} \in \BK^*$. This allows us to construct rational functions
	\begin{equation}
		\label{eqn:def zeta}
		\zeta_{ij}(x) \sim (1-x)^{-\delta_{ij}} \prod_{\alpha: i \rightarrow j} (1-xq_\alpha)
	\end{equation}
	for all $i,j \in I$. The notation $\sim$ means ``equality up to a Laurent monomial in $x$", whose particular value will be unimportant to us for the time being. 
	
	\medskip
	
	\begin{definition}
		\label{def:pre-quantum}
		
		The (positive half of the) \emph{pre-quantum loop algebra} associated to the data $\{\zeta_{ij}(x)\}_{i,j \in I}$ is 
		\[
		\tUUp = \BK \Big \langle e_{i,d} \Big \rangle_{i \in I, d \in \BZ} \Big/ \Big(\text{relation \eqref{eqn:rel quad}}\Big)
		\]
		where we impose the following relation for all $i,j \in I$
		\begin{equation}
			\label{eqn:rel quad}
			e_i(x) e_j(y) \zeta_{ji} \left(\frac yx\right) = e_j(y) e_i(x) \zeta_{ij} \left( \frac xy \right)
		\end{equation}
		Above and henceforth, we use the notation:
		\[
		e_i(x) = \sum_{d \in \BZ} \frac {e_{i,d}}{x^d}
		\]
		for all $i \in I$, and formula \eqref{eqn:rel quad} is interpreted as an infinite collection of relations obtained by equating the coefficients of all $\{x^a y^b\}_{a,b\in \BZ}$ in the left and right-hand sides (if $i = j$, one clears the denominators $x-y$ from \eqref{eqn:rel quad} before equating coefficients). 
		
	\end{definition}
	
	\medskip
	
	\subsection{Big shuffle algebras}
	\label{sub:def shuf}
	
	Consider an infinite collection of variables $z_{i1},z_{i2},\dots$ for all $i \in I$. For any $\bn = (n_i)_{i \in I} \in \nn$, we will write $\bn! = \prod_{i \in I} n_i!$ and
	\[
	\fS_{\bn} = \prod_{i \in I} \fS_{n_i}
	\]
	where $\fS_n$ denotes the symmetric group on $n$ letters. A function $F(z_{i1},\dots,z_{in_i})_{i \in I}$ is called \emph{color-symmetric} if it is fixed by the natural $\fS_{\bn}$ action on its variables. The following construction generalizes the trigonometric shuffle algebras of \cite{E, FO} in a straightforward manner.
	
	\medskip
	
	\begin{definition}
		\label{def:big shuf}
		
		The \emph{big shuffle algebra} associated to the data $\{\zeta_{ij}(x)\}_{i,j \in I}$ is
		\[
		\CV^- = \bigoplus_{\bn \in \BN^I} \CV_{-\bn}, \quad \text{where} \quad \CV_{-\bn} = \BK[z_{i1},z_{i1}^{- 1},\dots,z_{in_i}, z_{in_i}^{-1}]_{i \in I}^{\fS_{\bn}}
		\]
		endowed with the multiplication
		\begin{equation}
			\label{eqn:shuf prod}
			\begin{split}
				&R(z_{i1},\dots,z_{in_i})_{i \in I} * R'(z_{i1},\dots,z_{in_i'})_{i \in I} = \\
				\eSym_{\fS_{\bn+\bn'}}& \left[ \frac {R(z_{i1},\dots,z_{in_i})_{i \in I} R'(z_{i,n_i+1},\dots,z_{i,n_i+n_i'})_{i \in I}}{\bn! \bn'!} \mathop{\prod^{i,j \in I}_{1\leq a\leq n_i}}_{n_j < b \leq n_j+n_j'} \zeta_{ji} \left(\frac {z_{jb}}{z_{ia}} \right) \right]
			\end{split}
		\end{equation}
		where $\eSym_{\fS_{\bn+\bn'}}$ denotes summation over the orbits of the group $\fS_{\bn+\bn'}$ \footnote{Although the zeta functions might seem to contribute simple poles at $z_{ia}-z_{ib}$ for $a \neq b$ to the right-hand side of \eqref{eqn:shuf prod}, these poles disappear when taking the symmetrization (the poles in question can only have even order in any symmetric rational function).}. 
		
	\end{definition} 
	
	\medskip 
	
	\noindent We let $\CV^+ = \CV^{-,\text{op}}$, which means that one must replace in \eqref{eqn:shuf prod}
	\[
	\zeta_{ji} \left(\frac {z_{jb}}{z_{ia}} \right) \quad \text{by} \quad \zeta_{ij} \left(\frac {z_{ia}}{z_{jb}} \right) 
	\]
	to get the multiplication in $\CV^+$. The grading of $\CV^+$ will be denoted by $\bigoplus_{\bn \in \nn} \CV_{\bn}$.
	
	\medskip

	\subsection{Small shuffle algebras} 
	\label{sub:upsilon}
	
	The opposite algebra to that of Definition~\ref{def:pre-quantum} is
	\[
	\tUUm = \BK \Big \langle f_{i,d} \Big \rangle_{i \in I, d \in \BZ} \Big / \Big(\text{relation \eqref{eqn:opposite}}\Big)
	\]
	with the following relation for all $i,j \in I$:
	\begin{equation}
		\label{eqn:opposite}
		f_j(y) f_i(x) \zeta_{ji}\left(\frac yx \right) = f_i(x)f_j(y) \zeta_{ij} \left( \frac xy \right) 
	\end{equation}
	for all $i,j \in I$, where $f_i(x) = \sum_{d \in \BZ} \frac {f_{i,d}}{x^d}$. There exist algebra homomorphisms
	\begin{equation}
		\label{eqn:tilde upsilon}
		\tUpsilon^\pm : \tUUpm \rightarrow \CV^\pm, \qquad e_{i,d} \mapsto z_{i1}^d \in \CV_{\bs^i}, \quad f_{i,d} \mapsto z_{i1}^d \in \CV_{-\bs^i}
	\end{equation}
	for all $i \in I$, $d \in \BZ$, where $\bs^i =(0,\dots,0,1,0,\dots,0)$ with the $1$ on the $i$-th position. Indeed, proving the well-defined-ness of the homomorphisms \eqref{eqn:tilde upsilon} entails checking that relations \eqref{eqn:rel quad} and \eqref{eqn:opposite} are respected by the shuffle product \eqref{eqn:shuf prod} and its opposite, which is straightforward. The maps \eqref{eqn:tilde upsilon} are neither injective nor surjective in general, and an important role will be played by the \emph{small shuffle algebras}
	\begin{equation}
		\label{eqn:small}
		\CS^\pm = \text{Im }\widetilde{\Upsilon}^\pm
	\end{equation}
	If we define the (half) \emph{quantum loop algebras} as
	\[
	\UUpm = \tUUpm \Big/ \Big(\text{Ker }\widetilde{\Upsilon}^\pm\Big)
	\]
	then the first isomorphism theorem implies that $\tUpsilon^\pm$ descends to isomorphisms
	\begin{equation}
		\label{eqn:upsilon}
		\Upsilon^\pm : \UUpm \rightarrow \CS^\pm
	\end{equation}
	
	\medskip
	
	\subsection{The pairing}
	\label{sub:pairing}
	
	Let us consider the following notation for all rational functions $G(z_1,\dots,z_n)$ homogeneous of degree 0. If $Dz_a = \frac {dz_a}{2\pi i z_a}$, then we will write
	\begin{equation}
		\label{eqn:contour integral}
		\int_{|z_1| \gg \cdots \gg |z_n|} G(z_1,\dots,z_n) \prod_{a=1}^n Dz_a
	\end{equation}
	for the constant term in the expansion of $G$ as a power series in 
	\[
	\frac {z_2}{z_1}, \dots, \frac {z_n}{z_{n-1}}
	\]
	The notation in \eqref{eqn:contour integral} is motivated by the fact that if $\BK = \BC$, one could compute this constant term as a contour integral (with the contours being concentric circles, situated very far from each other compared to the absolute values of the coefficients of $G$). We define $\int_{|z_1| \ll \cdots \ll |z_n|} G(z_1,\dots,z_n) \prod_{a=1}^n Dz_a$ analogously.
	
	\medskip
	
	\begin{definition}
		\label{def:pair}
		
		There exist bilinear pairings
		\begin{align}
			&\tUUp \otimes \CV^- \xrightarrow{\langle \cdot, \cdot \rangle} \BK \label{eqn:pair} \\
			&\CV^+ \otimes \tUUm \xrightarrow{\langle \cdot, \cdot \rangle} \BK \label{eqn:pair opposite}
		\end{align}
		given for all homogeneous $E \in \CV_{\bn}$, $F \in \CV_{-\bn}$ and $i_1,\dots,i_n \in I$, $d_1,\dots,d_n \in \BZ$ by
		\begin{align}
			&\Big \langle e_{i_1,d_1} \cdots e_{i_n,d_n}, F \Big \rangle = \int_{|z_1| \gg \cdots \gg |z_n|} \frac {z_1^{d_1}\cdots z_n^{d_n} F(z_1,\dots,z_n)}{\prod_{1\leq a < b \leq n} \zeta_{i_bi_a} \left(\frac {z_b}{z_a} \right)} \prod_{a=1}^n Dz_a \label{eqn:pair formula} \\
			&\Big \langle E, f_{i_1,d_1} \cdots  f_{i_n,d_n} \Big \rangle = \int_{|z_1| \ll \cdots \ll |z_n|} \frac {z_1^{d_1}\cdots z_n^{d_n} E(z_1,\dots,z_n)}{\prod_{1\leq a < b \leq n} \zeta_{i_ai_b} \left(\frac {z_a}{z_b} \right)} \prod_{a=1}^n Dz_a 
			\label{eqn:pair formula opposite}
		\end{align}
		If $\bs^{i_1}+\cdots +\bs^{i_n} \neq \bn$ or $E,F$ fail to have degree $-d_1-\dots-d_n$, then the pairings above are defined to be $0$. 
		
	\end{definition}
	
	\medskip
	
	\noindent In the right-hand sides of \eqref{eqn:pair formula} and \eqref{eqn:pair formula opposite}, we implicitly identify
	\begin{equation}
		\label{eqn:relabeling}
		z_a \quad \text{with} \quad z_{i_a\bullet_a}, \quad \forall a \in \{1,\dots, n\}
	\end{equation}
	where $\bullet_1,\dots,\bullet_n$ are the minimal positive integers such that $\bullet_a<\bullet_b$ if $a<b$ and $i_a = i_b$. We call \eqref{eqn:relabeling} a \emph{relabeling} of the variables of $(E\text{ or }F)(z_{i1},\dots,z_{in_i})_{i \in I}$ in accordance with $i_1,\dots,i_n$. The following result was proved in \cite[Theorem 2.5]{N Arbitrary}.
	
	\medskip 
	
	\begin{theorem}
		\label{thm:pairings}
		
		The kernels of $\tUpsilon^\pm$ pair trivially with $\CS^\mp$ under the pairings \eqref{eqn:pair} and \eqref{eqn:pair opposite}, and so descend to pairings
		\begin{align*}
			&\UUp \otimes \CS^- \xrightarrow{\langle \cdot, \cdot \rangle} \BK \\
			&\CS^+ \otimes \UUm \xrightarrow{\langle \cdot, \cdot \rangle} \BK 
		\end{align*}
		which are non-degenerate in both arguments. The pairings above induce two pairings
		\begin{equation}
			\label{eqn:ultimate pairing}
			\UUp \otimes \UUm \xrightarrow{\langle \cdot, \cdot \rangle} \BK
		\end{equation}
		(under the isomorphism \eqref{eqn:upsilon}) which coincide. 
		
	\end{theorem}
	
	\medskip
	
	\subsection{Double algebras}
	\label{sub:double}
	
	We call the collection of zeta functions \eqref{eqn:def zeta} \emph{balanced} if
	\begin{equation}
		\label{eqn:limit}
		\lim_{x \rightarrow \infty}
		\frac {\zeta_{ij}(x)}{\zeta_{ji} (x^{-1})} < \infty 
	\end{equation}
	for all $i,j \in I$. Whenever this condition holds, we consider the extended algebras
	\begin{align}
		&\UUg = \frac {\UUp[\ph^+_{i,0}, \ph^+_{i,1}, \ph^+_{i,2}, \dots]_{i \in I}}{\left( \ph^+_i(x) e_j(y) = e_j(y) \ph^+_i(x) \frac {\zeta_{ij} \left(\frac xy \right)}{\zeta_{ji} \left(\frac yx \right)} \right)} \label{eqn:uug} \\
		&\UUl = \frac {\UUm[\ph^-_{i,0}, \ph^-_{i,1}, \ph^-_{i,2}, \dots]_{i \in I}}{\left( \ph^-_i(x) f_j(y) = f_j(y) \ph^-_i(x) \frac {\zeta_{ji} \left(\frac yx \right)}{\zeta_{ij} \left(\frac xy \right)} \right)} \label{eqn:uul}
	\end{align}
	where
	\[
	\ph^\pm_i(x) = \sum_{d=0}^{\infty} \frac {\ph^\pm_{i,d}}{x^{\pm d}} 
	\]
	and the quotient relations in \eqref{eqn:uug} (respectively \eqref{eqn:uul}) are interpreted by expanding them as power series in negative (respectively positive) powers of $\frac xy$. Following Drinfeld, we can make $\UUg$ and $\UUl$ into topological bialgebras via the coproduct
	\begin{align*}
		& \Delta(\ph^\pm_i(x)) = \ph^\pm_i(x) \otimes \ph^\pm_i(x) \\
		& \Delta(e_i(x)) = \ph^+_i(x) \otimes e_i(x) + e_i(x) \otimes 1 \\
		& \Delta(f_i(x)) = 1 \otimes f_i(x) + f_i(x) \otimes \ph^-_i(x)
	\end{align*}
	It is straightforward to check that the pairing \eqref{eqn:ultimate pairing} extends to a bialgebra pairing
	\[
	\UUg \otimes \UUl \xrightarrow{\langle \cdot, \cdot \rangle} \BK 
	\]
	via 
	\[
	\Big \langle \ph^+_i(x), \ph^-_i(y) \Big \rangle = \frac {\zeta_{ij} \Big(\frac xy \Big)}{\zeta_{ji} \Big(\frac yx \Big)} \quad \text{expanded as } |x| \gg |y|
	\]
	and properties $\langle a, bb'  \rangle = \langle \Delta(a), b \otimes b' \rangle$ and $\langle aa', b \rangle =\langle a' \otimes a, \Delta(b)  \rangle$ for all $a,a' \in \UUg$ and $b,b' \in \UUl$. There are also antipode maps on $\UUg$ and $\UUl$ satisfying the usual properties in a topological Hopf algebra.
	
	\medskip
	
	\begin{definition}
		\label{def:quantum}
		
		For any balanced collection of zeta functions \eqref{eqn:def zeta}, the \textbf{quantum loop algebra} is defined as
		\begin{equation}
			\label{eqn:quantum}
			\UU = \UUg \otimes \UUl
		\end{equation}
		with the multiplication governed by the Drinfeld double relation
		\begin{equation}
			\label{eqn:dd}
			(a \otimes b) (a' \otimes b') = \Big \langle S(a'_{1}), b_{1} \Big \rangle aa'_{2} \otimes b_{2}  b'\Big \langle a'_{3}, b_{3} \Big \rangle
		\end{equation}
		for any $a,a' \in \UUg$, $b,b' \in \UUl$ with their coproducts written in Sweedler notation
		\[
		\Delta^{(2)}(a') = a'_{1} \otimes a'_{2} \otimes a'_{3} \qquad \text{and} \qquad \Delta^{(2)}(b) =  b_{1} \otimes b_{2} \otimes b_{3}
		\]
		
	\end{definition}
	
	\medskip
	
	\noindent Using relation \eqref{eqn:dd}, it is straightforward to deduce commutation relations between $e_i(x), \ph^+_i(x)$ and $f_j(y),\ph^-_j(y)$. Thus, the full set of relations in $\UU$ is
	\begin{equation}
		\label{eqn:rel quantum 1}
		\Big(\text{any element of Ker }\widetilde{\Upsilon}^\pm \Big) = 0
	\end{equation} 
	\begin{equation}
		\label{eqn:rel quantum 2}
		e_i(x) e_j(y) \zeta_{ji} \left(\frac yx\right) = e_j(y) e_i(x) \zeta_{ij} \left( \frac xy \right)
	\end{equation}
	\begin{equation}
		\label{eqn:rel quantum 3}
		f_j(y) f_i(x) \zeta_{ji} \left(\frac yx\right) = f_i(x) f_j(y) \zeta_{ij} \left( \frac xy \right)
	\end{equation}
	\begin{equation}
		\label{eqn:rel quantum 4}
		\left[\ph^+_{i}(x), \ph^+_{j}(y)\right] = \left[\ph^+_{i}(x), \ph^-_{j}(y)\right] = \left[\ph^-_{i}(x), \ph^-_{j}(y)\right] = 0 
	\end{equation} 
	\begin{align}
		&\ph^\pm_i(x) e_j(y) = e_j(y) \ph^\pm_i(x) \frac {\zeta_{ij} \Big(\frac xy \Big)}{\zeta_{ji} \Big(\frac yx \Big)} \label{eqn:rel quantum 5} \\
		&f_j(y)\ph^\pm_i(x)  = \ph^\pm_i(x)  f_j(y) \frac {\zeta_{ij} \Big(\frac xy \Big)}{\zeta_{ji} \Big(\frac yx \Big)} \label{eqn:rel quantum 7} 
	\end{align}
	\begin{equation}
		\label{eqn:rel quantum 9}
		\Big[e_{i}(x), f_{j}(y)\Big] = \delta_{ij} \delta \left(\frac xy \right) \Big( \ph^-_j(y) - \ph^+_i(x)\Big)
	\end{equation} 
	for all $i,j \in I$, where $\delta(x) = \sum_{d \in \BZ} x^d$ is a formal series, and $\delta_{ij}$ is the Kronecker delta. Note that in relations \eqref{eqn:rel quantum 2}-\eqref{eqn:rel quantum 3}, if $i=j$ one first clears the denominator $x-y$ before equating the left and right-hand sides. By construction, we have
	\[
	\UU \cong \UUp \otimes \UUo \otimes \UUm
	\]
	where the three factors denote the algebras generated by $e_{i,d}$, $\ph^\pm_{i,d}$, $f_{i,d}$ (respectively) modulo the relations among \eqref{eqn:rel quantum 1}-\eqref{eqn:rel quantum 9} that pertain to the respective generators.
	
	\medskip 
	
	\subsection{Kac-Moody Lie algebras} 
	\label{sub:Kac-Moody}
	
	Arguably the most important example (as well as the initial motivation) of the quantum loop algebra $\UU$ of \eqref{eqn:quantum} is that of quantum affine algebras. We will actually recall the construction in the wider generality of a symmetrizable Kac-Moody Lie algebra $\fg$, which is determined by a set $I$ and a generalized symmetrizable Cartan matrix 
	\[
	\left(c_{ij} = \frac {2d_{ij}}{d_{ii}} \in \BZ \right)_{i,j\in I}
	\]
	Above, $d_{ij} = d_{ji} \in \BZ_{\leq 0}$ for all $i \neq j$, while $d_{ii} \in 2 \BZ_{>0}$. We will write
	\begin{equation}
		\label{eqn:two quantum loop}
		\widetilde{U}_q(L\fg) \twoheadrightarrow U_q(L\fg) 
	\end{equation}
	for the algebras $\tUU \twoheadrightarrow \UU$ of Subsections~\ref{sub:zeta}-\ref{sub:double} associated to the zeta functions 
	\begin{equation}
		\label{eqn:zeta strongly}
		\zeta_{ij}(x) \sim \frac {1-xq^{d_{ij}}}{(1-x)^{\delta_{ij}}}
	\end{equation}
	where $\sim$ means that the two sides of \eqref{eqn:zeta strongly} differ by a monomial chosen such that
	\begin{equation}
		\label{eqn:zeta strongly exchange}
		\frac {\zeta_{ij}(x)}{\zeta_{ji}(x^{-1})} = \frac {xq^{d_{ij}}-1}{x-q^{d_{ij}}}
	\end{equation}
	(an asymmetric choice for said monomial is $(-x)^{-\delta_{i>j}}$ for an arbitrary total order on the set $I$, and a more symmetric choice will be given in Subsection~\ref{sub:square roots}). As customary, we also impose the relation $\ph_{i,0}^+\ph_{i,0}^- = 1, \forall i\in I$ when defining $U_q(L\fg)$.  
	
	\medskip 
	
	\begin{remark}
		\label{rem:quiver} 
		
		In order to see that \eqref{eqn:zeta strongly} is a particular case of \eqref{eqn:def zeta}, note that the quiver $Q$ from Subsection~\ref{sub:zeta} is not the Dynkin diagram of $\fg$, even though its vertex set is the same as $I$. Instead, $Q$ has a single arrow from every $i \in I$ to every $j \in I$ (including loops that correspond to $i=j$) with parameter $q^{d_{ij}}$. In Subsection~\ref{sub:twisted quantum loop}, we will show how to remove those arrows with $d_{ij}=0$ from the quiver $Q$, all the while obtaining isomorphic quantum loop and shuffle algebras. The resulting quiver $Q$ will be called the tripled quiver associated to the Dynkin diagram of $\fg$, due to the presence of the arrows going both ways and the loops at all the vertices. 
		
	\end{remark}
	
	\medskip
	
	\noindent The notation in \eqref{eqn:two quantum loop} is motivated by the case of a simple finite-dimensional Lie algebra $\fg$, in which case the Drinfeld-Beck isomorphism (\cite{B, Dr}) shows that $U_q(L\fg)$ is isomorphic to the quantum affine algebra $U_q(\widehat{\fg})|_{c=1}$ with trivial central charge.
	
	\medskip 
	
	\begin{remark}
		\label{rem:symmetric}
		
		When $d_{ii} = 2, \forall i \in I$, we call the Kac-Moody Lie algebra $\fg$ simply-laced. In this case, the quantum loop algebra $U_q(L\fg)$ was fully described in \cite{N Symmetric}.
		
	\end{remark}
	
	\medskip 
	
	\begin{definition}
		\label{def:strongly}
		
		A Kac-Moody Lie algebra is called \emph{strongly symmetrizable} if 
		\[
		d_{ij} \in \left\{0,-\max\left(\frac {d_{ii}}2, \frac {d_{jj}}2 \right) \right\}
		\]
		for all $i \neq j$. All finite type and affine Lie algebras are strongly symmetrizable, except for affine $A_1$ (which instead fits into the framework of Remark~\ref{rem:symmetric}).
		
	\end{definition}
	
	\medskip 
	
	\noindent For any symmetrizable Kac-Moody Lie algebra $\fg$, we have
	\[
	\widetilde{U}_q(L\fg) = \BC \Big \langle e_{i,d}, f_{i,d}, \ph^\pm_{i,d'} \Big \rangle_{i \in I, d \in \BZ, d' \geq 0} \Big/ \Big(\text{relations \eqref{eqn:rel quantum 1 strongly}-\eqref{eqn:rel quantum 8 strongly}} \Big) 
	\]
	\begin{equation}
		\label{eqn:rel quantum 1 strongly}
		e_i(x) e_j(y) (x - yq^{d_{ij}}) = e_j(y) e_i(x) (xq^{d_{ij}} - y)
	\end{equation}
	\begin{equation}
		\label{eqn:rel quantum 2 strongly}
		f_j(y) f_i(x)  (x - yq^{d_{ij}}) = f_i(x) f_j(y) (xq^{d_{ij}} - y)
	\end{equation}
	\begin{equation}
		\label{eqn:rel quantum 3 strongly}
		\left[\ph^+_{i}(x), \ph^+_{j}(y)\right] = \left[\ph^+_{i}(x), \ph^-_{j}(y)\right] = \left[\ph^-_{i}(x), \ph^-_{j}(y)\right] = 0, \ \ph^+_{i,0}\ph^-_{i,0} = 1
	\end{equation} 
	\begin{align}
		&\ph^\pm_i(x) e_j(y) = e_j(y) \ph^\pm_i(x) \frac {xq^{d_{ij}}-y}{x-yq^{d_{ij}}} \label{eqn:rel quantum 4 strongly} \\
		&f_j(y)\ph^\pm_i(x)  = \ph^\pm_i(x)  f_j(y) \frac {xq^{d_{ij}}-y}{x-yq^{d_{ij}}} \label{eqn:rel quantum 6 strongly} 
	\end{align}
	\begin{equation}
		\label{eqn:rel quantum 8 strongly}
		\Big[e_{i}(x), f_{j}(y)\Big] = \delta_{ij} \delta \left(\frac xy \right) \frac {\ph^+_i(x) - \ph^-_j(y)}{q_i-q_i^{-1}}
	\end{equation} 
	for all $i,j \in I$ \footnote{One goes from \eqref{eqn:rel quantum 9} to the more standard \eqref{eqn:rel quantum 8 strongly} by rescaling the pairings \eqref{eqn:pair formula}, \eqref{eqn:pair formula opposite} by 
		\[
		(q_{i_1}^{-1}-q_{i_1})^{-1} \cdots (q_{i_n}^{-1}-q_{i_n})^{-1},
		\]
		which does not change anything in the overall theory. Here we write $q_i=q^{\frac {d_{ii}}2}$, as is conventional.}. In order to obtain the quotient $U_q(L\fg)$ of $\widetilde{U}_q(L\fg)$, we must impose additional relations, namely \eqref{eqn:rel quantum 1} for the zeta functions \eqref{eqn:zeta strongly}. In the following Subsection, we will show how to interpret such relations from the point of view of shuffle algebras, and recover the well-known Drinfeld-Serre relations when $\fg$ is strongly symmetrizable.
	
	\medskip 
	
	\subsection{Feigin-Odesskii wheel conditions and relations} 
	\label{sub:wheel conditions}
	
	It was shown in \cite{N New} that for a strongly symmetrizable Kac-Moody Lie algebra $\fg$, the corresponding shuffle algebra (defined with respect to the zeta functions \eqref{eqn:zeta strongly}) is completely described by
	\[
	\CS^\mp = \Big\{\text{Laurent polynomials satisfying the wheel conditions \eqref{eqn:wheel strongly}}\Big\} \subset \CV^\mp
	\]
	where a color-symmetric Laurent polynomial $F(z_{i1},\dots,z_{in_i})_{i \in I}$ is said to satisfy the \emph{wheel conditions} (initially discovered by \cite{FO}) if for all $i \neq j$ in $I$ we have
	\begin{equation}
		\label{eqn:wheel strongly}
		F\Big|_{(z_{i1},z_{i2}, \dots, z_{i,-c_{ij}}, z_{i,1-c_{ij}}) \mapsto (z_{j1} q^{d_{ij}}, z_{j1} q^{d_{ij}+d_{ii}},  \dots, z_{j1} q^{-d_{ij}-d_{ii}}, z_{j1} q^{-d_{ij}})} =  0
	\end{equation}
	The terminology arises from the following picture
	
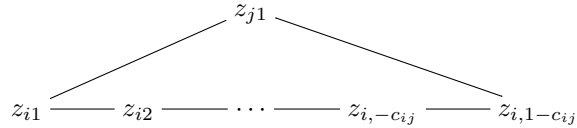
\[
	\xymatrix{& & z_{j1} \ar@{-}[rrd] & & \\
		z_{i1} \ar@{-}[rru] &  z_{i2} \ar@{-}[l] & \cdots \ar@{-}[l] &  z_{i,-c_{ij}} \ar@{-}[l] & z_{i,1-c_{ij}} \ar@{-}[l]}
	\]
	which we interpret as follows: imagine that the symbols $z_{ia}$ are placed upon a real line at successive distance $d_{ii}$ apart. When we set
	\[
	z_{ia} \mapsto q^{\text{horizontal coordinate denoted by ``$z_{ia}$" above}}
	\]
	and similarly for $z_{j1}$, the specialization of $F$ is 0.
	
	\medskip 
	
	\begin{example} 
		\label{ex:easy wheels}
		
		If $d_{ij} = 0$, then the wheel conditions state that
		\[
		F \Big |_{z_{i1} \mapsto z_{j1}} = 0 
		\]
		which implies that any $F \in \CS^\pm$ is divisible by $z_{ia}-z_{jb}$ for all $a,b$. If $d_{ij} = -1$, then the wheel conditions correspond to a triangle, so any $F \in \CS^\pm$ satisfies
		\[
		F \Big |_{z_{i1} \mapsto z_{j1} q^{-1}, z_{i2} \mapsto z_{j1} q} = 0 
		\]
	\end{example} 
	
	\medskip 
	
	\noindent In \cite{N Reduced}, the first-named author considered the following interplay between wheel conditions and relations in quantum loop algebras: suppose we wish to consider Laurent polynomials $F$ which vanish when a certain subset $z_1,\dots,z_k$ of their variables are specialized according to the so-called \emph{wheel}
	\begin{equation}
		\label{eqn:general wheel}
		\bigcirc = \Big\{ z_k = z_{k-1} q_k, \dots, z_2 = z_1 q_2, z_1 = z_k q_1 \Big\}
	\end{equation}
	where each $q_a$ is the parameter of an arrow $i_{a-1} \rightarrow i_a$, such that $q_1\cdots q_k = 1$ but $q_{a+1}\dots q_b \neq 1$ whenever $i_a = i_b$. Above, we make the convention that $z_a$ actually denotes the variable $z_{i_a\bullet_a}$ of $F$, for various $i_a \in I$ and $\bullet_a \geq 1$ (the latter of which will not be important to us), see \eqref{eqn:relabeling}. Recall the following formal series from \cite{N Reduced}:
	\begin{equation}
		\label{eqn:clunky}
		\begin{split}
			e_{\bigcirc} = \sum_{\ell = 1}^k P_{\ell}(x_1,\dots,x_k) e_{i_\ell}(x_{\ell}) \cdots e_{i_1}(x_1) e_{i_k}(x_k) \cdots e_{i_{\ell+1}}(x_{\ell+1}) \\
			f_{\bigcirc} = \sum_{\ell = 1}^k P_{\ell}(x_1,\dots,x_k) f_{i_{\ell+1}}(x_{\ell+1})  \cdots  f_{i_k}(x_k) f_{i_1}(x_1) \cdots f_{i_\ell}(x_{\ell})
		\end{split}
	\end{equation}
	where if we write $\widetilde{\zeta}_{ij}(x) = \zeta_{ij}(x) (1-x)^{\delta_{ij}}$, then we set
	\[
	P_{\ell}(x_1,\dots,x_k) = \frac {\frac {x_1 q_2 \cdots q_{\ell}}{x_{\ell}} \prod^{1 \leq a < b \leq \ell \text{ or } \ell < a < b \leq k}_{\text{or } 1 \leq b \leq \ell < a \leq k} \widetilde{\zeta}_{i_ai_b}\left(\frac {x_a}{x_b} \right) \left(-\frac {x_b}{x_a}\right)^{\delta_{b<a} \delta_{i_bi_a}}}{\left(1- \frac {x_k q_{1}}{x_{1}}\right)^{\delta_{\ell \neq k}} \prod^{1 \leq a < \ell \text{ or}}_{\ell < a < k}\left(1- \frac {x_a q_{a+1}}{x_{a+1}}\right)}
	\]
	The reason for the above choice of $P_{\ell}$'s is that they are Laurent polynomials and that we have the following formula for expansions of rational functions (see \cite[Proposition 3.5]{N Reduced})
	\begin{equation*}
		\begin{split}
			\sum_{\ell=1}^k \text{ev}_{|x_{\ell}| \gg \cdots \gg |x_1| \gg |x_k| \gg \cdots \gg |x_{\ell+1}|} \left[ \frac {P_{\ell}(x_1,\dots,x_k)}{\prod^{1 \leq a < b \leq \ell \text{ or } \ell < a < b \leq k}_{\text{or } 1 \leq b \leq \ell < a \leq k} \zeta_{i_ai_b}\left(\frac {x_a}{x_b} \right)} \right] \\ = \delta \left(\frac {z_k}{z_{k-1}q_k} \right) \cdots  \delta \left(\frac {z_2}{z_1q_2} \right) \prod_{1 \leq a < b \leq k} \left(1-\frac {x_a}{x_b}\right)^{\delta_{i_ai_b}}
		\end{split}
	\end{equation*}
	As shown in \loccitt, the formula above ensures that
	\begin{equation}
		\label{eqn:compare 1}
		\Big \langle e_{\bigcirc}, F \Big \rangle = 0 \quad \Leftrightarrow \quad \Big(F \text{ vanishes at the wheel \eqref{eqn:general wheel}} \Big)
	\end{equation}
	Thus, in all cases when the small shuffle algebra $\CS^- \subset \CV^-$ consists of shuffle elements which vanish at several wheels $\bigcirc_1,\dots,\bigcirc_N$, it was shown in \cite{N Arbitrary, NSS} that
	\[
	\text{Ker } \tUpsilon^+  = \Big(e_{\bigcirc_1},\dots,e_{\bigcirc_N}\Big) 
	\]
	This discussion provides a generators-and-relations presentation for quantum loop algebras associated to any strongly symmetrizable Kac-Moody Lie algebra
	\begin{equation}
		\label{eqn:km 1}
		U_q(L\fg) = \BC \Big \langle e_{i,d}, f_{i,d}, \ph^\pm_{i,d'} \Big \rangle^{i \in I}_{d \in \BZ, d' \geq 0} \Big/ \Big(\text{relations \eqref{eqn:rel quantum 1 strongly}-\eqref{eqn:rel quantum 8 strongly} and } e_{\bigcirc} = f_{\bigcirc} = 0\Big) 
	\end{equation}
	as $\bigcirc$ goes over all wheels \eqref{eqn:wheel strongly}. The reason why the above presentation looks different from the usual one is that there is nothing canonical about formulas \eqref{eqn:clunky}. One could replace them by any other formulas which satisfy property \eqref{eqn:compare 1}, and one would obtain the same quotient \eqref{eqn:km 1}. For instance, let us consider the well-known Drinfeld-Serre relations (let $\text{Sym}_{\mathfrak{S}_{1-c_{ij}}}$ denote symmetrization over the $x$ variables)
	\begin{align*}
		S_{ij}^+ = \sum_{k=0}^{1-c_{ij}} (-1)^k \binom{1-c_{ij}}{k}_{q_i} \text{Sym}_{\mathfrak{S}_{1-c_{ij}}} \Big[ e_i(x_1) \cdots e_i(x_k) e_j(y) e_i(x_{k+1}) \cdots e_i(x_{1-c_{ij}}) \Big]  \\
		S_{ij}^- = \sum_{k=0}^{1-c_{ij}} (-1)^k \binom{1-c_{ij}}{k}_{q_i} \text{Sym}_{\mathfrak{S}_{1-c_{ij}}} \Big[ f_i(x_1) \cdots f_i(x_k) f_j(y) f_i(x_{k+1}) \cdots f_i(x_{1-c_{ij}}) \Big]
	\end{align*}
	for all $i \neq j$, where $q_i = q^{\frac {d_{ii}}2}$. It was shown in \cite{NT} that 
	\[
	\Big \langle S_{ij}^+, F \Big \rangle = 0 \quad \Leftrightarrow \quad \Big(F \text{ satisfies \eqref{eqn:wheel strongly}} \Big)
	\]
	which implies that
	\begin{equation}
		\label{eqn:compare 3}
		\Big \langle S_{ij}^+, F \Big \rangle = 0 \quad \Leftrightarrow \quad \Big \langle e_{\bigcirc \text{ of \eqref{eqn:wheel strongly}}}, F \Big \rangle = 0
	\end{equation}
	As explained in general terms in \cite{N Arbitrary}, formula \eqref{eqn:compare 3} implies that the ideal $\text{Ker }\tUpsilon^+$ is generated by either the elements $e_{\bigcirc \text{ of \eqref{eqn:wheel strongly}}}$ or by the elements $S_{ij}^+$ for all $i \neq j$. We conclude that for any strongly symmetrizable Kac-Moody Lie algebra, the presentation \eqref{eqn:km 1} of the quantum loop group is equivalent to the usual one
	\begin{equation}
		\label{eqn:km 2}
		U_q(L\fg) = \BC \Big \langle e_{i,d}, f_{i,d}, \ph^\pm_{i,d'} \Big \rangle^{i \in I}_{d \in \BZ, d' \geq 0} \Big/ \Big(\text{relations \eqref{eqn:rel quantum 1 strongly}-\eqref{eqn:rel quantum 8 strongly} and } S^\pm_{ij} = 0, \forall i \neq j\Big) 
	\end{equation}
	that was constructed by Drinfeld in \cite{Dr}.
	
	\medskip
	
	\subsection{Square roots}
	\label{sub:square roots}
	
	The current Subsection can be skipped on a first reading, as it merely explains a rather technical detail about zeta functions and shuffle algebras. If one wanted to fix the monomial discrepancy between the two sides of equation \eqref{eqn:zeta strongly} in a more symmetric way, one could set 
	\begin{equation}
		\label{eqn:zeta strongly square roots}
		\zeta_{ij}(x) = (-1)^{\delta_{i>j}} \frac {x^{-\frac 12}-x^{\frac 12}q^{d_{ij}}}{(x^{-\frac 12}-x^{\frac 12})^{\delta_{ij}}}
	\end{equation}
	where $>$ denotes an arbitrary total order on $I$. In the present Subsection, we argue that the impact of the square roots above on the corresponding shuffle algebras is minimal, as shown by the following result.
	
	\medskip 
	
	\begin{proposition}
		\label{prop:square roots}
		
		Consider arbitrary collections of rational functions
		\[
		\zeta'_{ij}(x) \in \frac {\BK[x^{\pm 1}]}{(1-x)^{\delta_{ij}}} \qquad \text{and} \qquad \zeta''_{ij}(x) = x^{\frac {t_{ij}}2}\zeta'_{ij}(x)
		\]
		and let ${\CS'}^\mp$ and ${\CS''}^\mp$ be the small shuffle algebras \eqref{eqn:small} corresponding to these choices of zeta functions. As long as $t_{ij} \in \BZ$ and $t_{ij} \equiv t_{ji}$ mod $2$ for all $i,j \in I$, we have
		\begin{equation}
			\label{eqn:square roots}
			\CS''_{\mp \bn} = \CS'_{\mp \bn} \cdot \prod_{i \in I} z_i^{\frac {-t_{ii} + \sum_{j \in I} n_j t_{ij}}2}
		\end{equation}
		for all $\bn = (n_i)_{i \in I} \in \nn$.
		
	\end{proposition}
	
	\medskip
	
	\begin{proof} Let us deal with the case $\mp = -$. By definition, $\CS_{-\bn}'$ is the linear span of
		\[
		\mathop{\sum_{\sigma \in \fS_n}}_{i_{\sigma(a)} = i_a, \forall a} z_{\sigma(1)}^{d_1}\cdots z_{\sigma(n)}^{d_n} \prod_{1 \leq a < b \leq n} \zeta'_{i_bi_a} \left(\frac {z_{\sigma(b)}}{z_{\sigma(a)}} \right)
		\]
		for various $i_1,\dots,i_n \in I$ and $d_1,\dots,d_n \in \BZ$ such that $\bs^{i_1}+\cdots+\bs^{i_n} = \bn$ (above, we identify each variable $z_a$ with $z_{i_a\bullet_a}$ where $\bullet_1,\dots, \bullet_n$ are the minimal positive integers such that $\bullet_a < \bullet_b$ if $a<b$ and $i_a = i_b$). Similarly, $\CS_{-\bn}''$ is the linear span of
		\begin{equation*}
			\begin{split}
				\mathop{\sum_{\sigma \in \fS_n}}_{i_{\sigma(a)} = i_a, \forall a} z_{\sigma(1)}^{d_1}\cdots z_{\sigma(n)}^{d_n} \prod_{1 \leq a < b \leq n} \zeta'_{i_bi_a} \left(\frac {z_{\sigma(b)}}{z_{\sigma(a)}} \right) \left( \frac {z_{\sigma(b)}}{z_{\sigma(a)}} \right)^{\frac {t_{i_bi_a}}2} \\
				= \mathop{\sum_{\sigma \in \fS_n}}_{i_{\sigma(a)} = i_a, \forall a}  z_{\sigma(1)}^{\bar{d}_1}\cdots z_{\sigma(n)}^{\bar{d}_n} \prod_{1 \leq a < b \leq n} \zeta'_{i_bi_a} \left(\frac {z_{\sigma(b)}}{z_{\sigma(a)}} \right)
			\end{split}
		\end{equation*}
		where
		\begin{equation*}
			\begin{split}
				\bar{d}_a = d_a - \sum_{b > a} \frac {t_{i_bi_a}}2 + \sum_{b < a} \frac {t_{i_ai_b}}2 &= d_a + \text{integer} + \sum_{b \neq a} \frac {t_{i_ai_b}}2\\
				&= d_a + \text{integer} + \frac {-t_{i_ai_a} + \sum_{j \in I} n_j t_{i_aj}}2
			\end{split}
		\end{equation*}
		As $d_a$ varies over $\BZ$, $\bar{d}_a$ varies over $\BZ + \frac {-t_{i_ai_a} + \sum_{j \in I} n_j t_{i_aj}}2$, which yields \eqref{eqn:square roots}. \end{proof}
	
	\bigskip 
	
	\section{Foldings and coverings}
	\label{sec:folding}
	
	\medskip 
	
	\noindent We introduce the main technical constructions of the present paper: the study of shuffle algebras in the setting of folding quivers with automorphisms and in the setting of coverings of quivers. We emphasize the example of twisted quantum affine algebras, and use our formalism to define twisted quantum toroidal algebras. 
	
	\medskip 
	
	\begin{itemize}[leftmargin=*]
		
		\item In Subsection~\ref{sub:quivers automorphism}, we introduce folded zeta functions for quivers with automorphisms, and in Subsection~\ref{sub:folding shuffle algebras} we study the corresponding shuffle algebras.
		
		\medskip 
		
		\item In Subsections~\ref{sub:auto}-\ref{sub:twisted quantum loop}, we apply the general theory to the Kac-Moody case.
		
		\medskip 
		
		\item In Subsection~\ref{sub:affine}, we prove the equivalence between our presentation and the classical one for twisted quantum affine algebras (finite type $\fg$), while in Subsections~\ref{sub:toroidal}-\ref{sub:toroidal cycle} we define twisted quantum toroidal algebras (affine type $\fg$).
		
		\medskip 
		
		\item In Subsections~\ref{sub:covering quivers}-\ref{sub:covering shuffle algebras}, we present analogous constructions to the ones above, in the setting of coverings of quivers. While there is a non-empty overlap between coverings of quivers and folding quivers with automorphisms (see Remark~\ref{rem:intersection}), the two settings are in general different.
		
	\end{itemize}
	
	\medskip 
	
	\noindent Throughout the present Section, bar notations such as $\oI, \ozeta, \oCS, \oCV, \oUU$ etc will refer to folding shuffle algebras associated to quivers with automorphisms in Subsections~\ref{sub:quivers automorphism}-\ref{sub:toroidal cycle} and \ref{sub:technical}, and to shuffle algebras associated to coverings of quivers in Subsections~\ref{sub:covering quivers}-\ref{sub:covering shuffle algebras}. The reason for this abuse of notation is to emphasize the parallels between the study of shuffle algebras in our two different settings: folding quivers with automorphisms on one hand, and coverings of quivers on the other.
	
	\medskip 
	
	\subsection{Folding quivers with automorphisms}
	\label{sub:quivers automorphism}
	
	Consider a quiver $Q$ endowed with an automorphism $\sigma : Q \rightarrow Q$ of order $m$. Fix a primitive $m$-th root of unity
	\[
	\omega = \sqrt[m]{1} \in \BK^*
	\]
	Assume we have parameters $\{q_{\alpha} \in \BK^*\}_{\alpha \text{ arrow of }Q}$ which are invariant under the action of the automorphism, i.e. $q_{\alpha} = q_{\sigma(\alpha)}$. Thus, given \emph{unfolded zeta functions}
	\begin{equation}
		\label{eqn:unfolded zeta}
		\zeta_{ij}(x) \sim (1-x)^{-\delta_{ij}} \prod_{\alpha : i \rightarrow j} (1-xq_\alpha)
	\end{equation}
	we define the \emph{folded zeta functions} by
	\begin{equation}
		\label{eqn:folded zeta}
		\ozeta_{ij}(x) = \prod_{k=1}^m \zeta_{\sigma^k(i)j}  \left( \omega^k x \right) 
	\end{equation}
	We write $\oI = I/\sigma$ and $\fold : I \rightarrow \oI$ for the standard quotient map. We choose a set of \emph{distinguished preimages} $I \supset J \xrightarrow{\sim} \oI$ for the map $\fold$. Consider the numbers
	\begin{equation}
		\label{eqn:m i}
		m_i = \frac m{\min \{a\geq 1 \ | \ \sigma^a(i) = i \}}
	\end{equation}
	for all $i \in I$. Since the numbers above are equal on the fibers of $\fold$, then it is natural to set $m_{\oi} = m_{i}$ for any $i \in \fold^{-1}(\oi)$. Formula \eqref{eqn:folded zeta} gives us
	\begin{equation*}
		\begin{split}
			\ozeta_{ij}(x) &\sim \prod_{k=1}^m \left[\left(1-x\omega^k \right)^{-\delta_{\sigma^k(i)j}}  \prod_{\alpha : \sigma^k(i) \rightarrow j} \left( 1-x\omega^k q_\alpha \right) \right]  \\  &\sim \prod_{k=1}^m \left[\left(1-x\omega^{-k} \right)^{-\delta_{i\sigma^{k}(j)}}  \prod_{\alpha : i \rightarrow \sigma^{k}(j)} \left( 1-x\omega^{-k} q_\alpha \right) \right] 
		\end{split}
	\end{equation*}
	For suitable choices of the Laurent monomials which determine the proportionality denoted by $\sim$ in the formula above, $\ozeta_{ij}(x)$ can be made invariant under $x \mapsto x \omega^{\frac m{m_i}}$ and $x \mapsto x \omega^{\frac m{m_j}}$, which implies that $\ozeta_{ij}(x)$ is actually a function of $x^{\text{lcm}(m_i,m_j)}$. This makes it natural to write the formulas above in terms of $\oi, \oj \in \oI$, namely
	\begin{equation}
		\label{eqn:folded zeta 3}
		\begin{split}
			\ozeta_{\oi \oj}(x) &\sim  \left(1-x^{m_{\oi}} \right)^{-\delta_{\oi \oj}} \prod_{k=1}^{\frac m{m_{\oi}}} \prod_{\alpha : \sigma^k(i_0) \rightarrow j_0} \left( 1-x^{m_{\oi}} \omega^{km_{\oi}} q_\alpha^{m_{\oi}} \right)  \\ &\sim \left(1-x^{m_{\oj}} \right)^{-\delta_{\oi \oj}} \prod_{k=1}^{\frac m{m_{\oj}}} \prod_{\alpha : i_0 \rightarrow \sigma^{k}(j_0)} \left( 1-x^{m_{\oj}} \omega^{-km_{\oj}} q_\alpha^{m_{\oj}} \right) 
		\end{split}
	\end{equation}
	where $i_0,j_0 \in J$ denote the distinguished preimages of $\oi,\oj \in \oI$. 
	
	\medskip 
	
	\begin{example}
		\label{ex:a2}
		
		Consider the following quiver with parameters for any $q \in \BC^* \backslash \sqrt[\BN]{1}$ and consider the order $2$ automorphism that swaps the vertices $1$ and $2$.
		
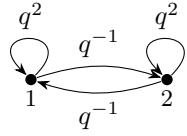
\[
		\begin{tikzpicture}[
			>=Stealth,
			vertex/.style={circle, fill=black, inner sep=1.5pt},
			lab/.style={below=3pt, font=\small}
			]
			
			\node[vertex, label={[lab]$1$}] (1) at (0,0) {};
			\node[vertex, label={[lab]$2$}] (2) at (1.8,0) {};
			
			\draw[->, bend left=18] (1) to node[above] {$q^{-1}$} (2);
			\draw[->, bend left=18] (2) to node[below] {$q^{-1}$} (1);
			\draw[->] (1) to[out=45, in=135, looseness=20] node[above] {$q^2$} (1);
			\draw[->] (2) to[out=45, in=135, looseness=20] node[above] {$q^2$} (2);
		\end{tikzpicture}
		\]
		
		\noindent The unfolded zeta functions as defined in the present subsection are given by
		\[
		\zeta_{ij}(x) \sim \begin{cases} \frac {1-xq^2}{1-x} &\text{if }i=j \\
			1-xq^{-1} &\text{if }i\neq j \end{cases}
		\]
		Since $\sigma$ has order $2$, we have $\omega = -1$. The folded zeta function (with the distinguished preimage of $\bar{1}$ being $1$) is given by
		\begin{equation}
			\label{eqn:example A2}
			\ozeta_{\bar{1} \bar{1}}(x) \sim \frac{(1-xq^2)(1+xq^{-1})}{1-x}
		\end{equation}
		
	\end{example}
	
	\medskip
	
	\begin{example} 
		\label{ex:a3}
		
		Consider the following quiver with parameters for any $q \in \BC^* \backslash \sqrt[\BN]{1}$ and consider the order $2$ automorphism that permutes the vertices $1$ and $3$, all the while fixing vertex $2$.
		
		
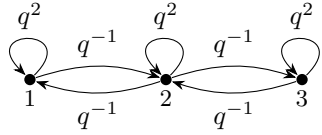
\[
		\begin{tikzpicture}[
			>=Stealth,
			vertex/.style={circle, fill=black, inner sep=1.5pt},
			lab/.style={below=3pt, font=\small}
			]
			
			\node[vertex, label={[lab]$1$}] (1) at (0,0) {};
			\node[vertex, label={[lab]$2$}] (2) at (1.8,0) {};
			\node[vertex, label={[lab]$3$}] (3) at (3.6,0) {};
			
			\draw[->, bend left=18] (1) to node[above] {$q^{-1}$} (2);
			\draw[->, bend left=18] (2) to node[below] {$q^{-1}$} (1);
			\draw[->, bend left=18] (2) to node[above] {$q^{-1}$} (3);
			\draw[->, bend left=18] (3) to node[below] {$q^{-1}$} (2);
			\draw[->] (1) to[out=45, in=135, looseness=20] node[above] {$q^2$} (1);
			\draw[->] (2) to[out=45, in=135, looseness=20] node[above] {$q^2$} (2);
			\draw[->] (3) to[out=45, in=135, looseness=20] node[above] {$q^2$} (3);
		\end{tikzpicture}
		\]
		
		\noindent The unfolded zeta functions are given by
		\[
		\zeta_{ij}(x) \sim \begin{cases} \frac {1-xq^2}{1-x} &\text{if }i=j \\
			1-xq^{-1} &\text{if } (i,j) \in \{(1,2),(2,1),(3,2),(2,3)\} \\ 1 &\text{if } (i,j) \in \{(1,3),(3,1)\} \end{cases}
		\]
		We have $\omega = -1$. The folded zeta functions (with the distinguished preimages of $\bar{1}$ and $\bar{2}$ being $1$ and $2$, respectively) are given by
		\[
		\ozeta_{\oi \oj}(x) \sim \begin{cases} \frac {1-xq^2}{1-x} &\text{if }\oi=\oj = \bar{1} \\ \frac {1-x^2q^4}{1-x^2} &\text{if }\oi = \oj = \bar{2} \\ 1 - x^{2}q^{-2} &\text{if } (\oi,\oj) \in \{(\bar{1}, \bar{2}), (\bar{2},\bar{1})\} \end{cases}
		\]
		
	\end{example} 
	
	\medskip 
	
	\subsection{Folded shuffle algebras}
	\label{sub:folding shuffle algebras}
	
	In the setting of the previous Subsection, define
	\begin{equation}
		\label{eqn:v folded}
		\oCV^+ = \oCV^- := \bigoplus_{\obn \in \onn} \BK \left[z_{\oi 1}^{m_{\oi}}, z_{\oi 1}^{- m_{\oi}},\dots,z_{\oi \on_{\oi}}^{m_{\oi}}, z_{\oi \on_{\oi}}^{- m_{\oi}} \right]_{\oi \in \oI}^{\fS_{\obn}}
	\end{equation}
	The following result is a simple exercise that we leave to the reader, using the fact that the folded zeta functions \eqref{eqn:folded zeta 3} are actually functions of $x^{\text{lcm}(m_{\oi},m_{\oj})}$ (this claim requires a suitable condition on the Laurent monomial which measures the discrepancy between the LHS and RHS of \eqref{eqn:folded zeta 3}, which we henceforth assume).
	
	\medskip 
	
	\begin{proposition}
		
		The sets of Laurent polynomials \eqref{eqn:v folded} are closed under shuffle product with respect to the folded zeta functions \eqref{eqn:folded zeta 3}.
		
	\end{proposition}
	
	\medskip
	
	\noindent We will find it convenient to write for all $\oi \in \oI$ and $a \geq 1$
	\begin{equation}
		\label{eqn:formula 1}
		z_{i_0 a} = z_{\oi a}
	\end{equation}
	where $i_0 \in I$ is the distinguished preimage of $\oi \in \oI$, and to extend this notation by
	\begin{equation}
		\label{eqn:formula 2}
		z_{\sigma^k(i_0)a} = z_{i_0 a}\omega^k
	\end{equation}
	for all $k \in \BZ/m\BZ$. The advantage of this notation is that Laurent polynomials in $z_{ia}$ can be understood as elements of \eqref{eqn:v folded}: simply replace all the $z_{ia}$'s with $z_{\oi a}$'s times powers of $\omega$ using formulas \eqref{eqn:formula 1} and \eqref{eqn:formula 2}, and arbitrarily change the $a$ indices to prevent overlaps. Dually, this means that the (half) quantum loop algebras 
	\[
	\oUUp \text{ and } \oUUm
	\]
	(defined as in Subsection~\ref{sub:upsilon}, but using the zeta functions \eqref{eqn:folded zeta 3}) have generators
	\[
	e_{\oi,d} =: e_{i_0,d} \quad \text{and} \quad f_{\oi,d} =: f_{i_0,d}
	\]
	for all $\oi \in \oI$ and $d \in m_{\oi}\BZ$. We will find it convenient to extend the notation of these generators to all $i \in I$ and $d \in \BZ$ using the conditions
	\begin{equation}
		\label{eqn:formula 4}
		e_{\sigma^k(i_0),d} = e_{i_0,d} \omega^{dk} \quad \text{and} \quad f_{\sigma^k(i_0),d} = f_{i_0,d} \omega^{dk}
	\end{equation}
	Note that formula \eqref{eqn:formula 4} implies that $e_{i,d} = f_{i,d} = 0$ unless $d$ is a multiple of $m_i$.
	
	\medskip
	
	\subsection{Folding Kac-Moody Lie algebras with automorphisms}
	\label{sub:auto}
	
	A particularly important case of the construction in Subsection~\ref{sub:quivers automorphism} applies to a symmetrizable Kac-Moody Lie algebra $\fg$, endowed with an automorphism
	\[
	\sigma : I \rightarrow I 
	\]
	such that 
	\[
	d_{i j} = d_{\sigma(i)\sigma(j)}
	\]
	for all $i, j \in I$. Let $m$ denote the order of $\sigma$, and fix a primitive $m$-th root of unity $\omega$ (whenever we discuss the Kac-Moody case, the ground field will be $\BC$ and we will fix a parameter $q \in \BC^* \backslash \sqrt[\BN]{1}$). In this case, the unfolded zeta functions are given by 
	\[
	\zeta_{ij}(x) \sim \frac {1-xq^{d_{ij}}}{(1-x)^{\delta_{ij}}}
	\]
	as in \eqref{eqn:zeta strongly}, and the folded zeta functions are given by
	\begin{equation}
		\label{eqn:zeta strongly twisted}
		\ozeta_{ij}(x) = \prod_{k=1}^m \zeta_{\sigma^k(i)j}(\omega^kx) \sim \prod_{k=1}^m \frac {1 - x\omega^k q^{d_{\sigma^k(i)j}}}{(1-x\omega^k)^{\delta_{\sigma^k(i)j}}}
	\end{equation}
	By analogy with \eqref{eqn:zeta strongly exchange}, the symbol $\sim$ means that the two sides of \eqref{eqn:zeta strongly twisted} differ by a Laurent monomial that must be chosen such that
	\begin{equation}
		\label{eqn:zeta strongly twisted exchange}
		\frac {\ozeta_{ij}(x)}{\ozeta_{ji}(x^{-1})} = \prod_{k=1}^m \frac {(x\omega^k q^{d_{\sigma^k(i)j}}-1)}{(x\omega^k-q^{d_{\sigma^k(i)j}})}
	\end{equation}
	In particular, this choice ensures that the folded zeta functions are balanced, in the sense of \eqref{eqn:limit}. As in \eqref{eqn:zeta strongly square roots}, one rather symmetric choice of said monomial is 
	\[
	\ozeta_{ij}(x) =  \prod_{k=1}^m (-1)^{\sigma^k(i)>j}\frac {x^{-\frac 12} - x^{\frac 12}\omega^k q^{d_{\sigma^k(i)j}}}{(x^{-\frac 12}\omega^{-\frac k2} -x^{\frac 12}\omega^{\frac k2})^{\delta_{\sigma^k(i)j}}},
	\]
	where $>$ denotes an arbitrary total order on $I$, which underlies an arbitrary total order on $\oI$ (i.e. if $\oi > \oj$ in $\oI$, then $i > j$ in $I$). We refer to Proposition~\ref{prop:square roots} for the impact of the square roots above on the shuffle algebra. With all this in mind, we fix the normalization of the folded zeta functions \eqref{eqn:folded zeta} as 
	\[
	\ozeta_{\oi\oj}(x) = \ozeta_{i_0j_0}(x)
	\]
	where $i_0,j_0 \in I$ denote the distinguished preimages of $\oi,\oj \in \oI$ (for example, in the case of the finite type Dynkin diagrams in Subsection~\ref{sub:affine}, the distinguished preimage of any $\oi \in \oI = I/\sigma$ is chosen to be the preimage with the smallest label).
	
	\medskip 
	
	\subsection{Twisted wheel conditions}
	\label{sub:twisted wheel strongly}
	
	The following result provides a full description of the small shuffle algebra associated to folding a strongly symmetrizable Kac-Moody Lie algebra $\fg$ with respect to an automorphism $\sigma$, see Subsection~\ref{sub:auto}.
	
	\medskip 
	
	\begin{proposition}
		\label{prop:twisted wheel strongly}
		
		For any strongly symmetrizable $\fg$, the small shuffle algebra $\oCS^\pm$ is the subset of symmetric Laurent polynomials $F \in \oCV^\pm$ which satisfy the wheel conditions
		\begin{equation}
			\label{eqn:twisted wheel strongly}
			F\Big|_{(z_{i1},z_{i2}, \dots, z_{i,-c_{ij}}, z_{i,1-c_{ij}}) \mapsto (z_{j1} q^{d_{ij}}, z_{j1} q^{d_{ij}+d_{ii}},  \dots, z_{j1} q^{-d_{ij}-d_{ii}}, z_{j1} q^{-d_{ij}})} =  0
		\end{equation}
		$\forall i \neq j$ in $I$. 
		The variables $z_{ia}$ of $F$ are indexed by $i \in I$ in \eqref{eqn:twisted wheel strongly}, see \eqref{eqn:formula 1}-\eqref{eqn:formula 2}. \footnote{Note that the identification $z_{\sigma(i)a} = z_{ia} \omega$ implicit in formulas \eqref{eqn:formula 1}-\eqref{eqn:formula 2} means that the wheel conditions \eqref{eqn:twisted wheel strongly} impose quite a lot of constraints on elements of $\oCV^\pm$. For example, if $d_{i'j'} = 0$ for all elements $i',j'$ in the orbits of given $i,j \in I$, then conditions \eqref{eqn:twisted wheel strongly} require $F$ to be divisible by $z_{i1}^g-z_{j1}^g$, where $g = \text{lcm}(m_i,m_j)$.}
	\end{proposition}
	
	\medskip 
	
	\noindent The Proposition implies that there exists an isomorphism analogous to \eqref{eqn:upsilon}
	\begin{equation}
		\label{eqn:upsilon twisted strongly}
		\oUpsilon^\pm : \oUUpm \xrightarrow{\sim} \oCS^\pm
	\end{equation}
	
	\medskip
	
	\begin{proof} We will sketch the usual argument, as laid out in \cite{N Shuffle, N Wheel, N Arbitrary, N New, NSS}, and provide references for the details. First of all, the fact that $\oCS^\pm$ is closed under the shuffle product is proved as in \cite[Proposition 2.1]{N Shuffle}; the key feature it uses is that
		\[
		\ozeta_{ji}(q^{-d_{ij}}) = \ozeta_{ii}(q^{-d_{ii}}) = \ozeta_{ij}(q^{-d_{ij}}) = 0
		\]
		As such, $\oCS^\pm$ contains the image of the homomorphism
		\begin{equation}
			\label{eqn:bar tupsilon}
			\bar{\tUpsilon}^\pm : \bar{\tUU}^{\pm} \longrightarrow \oCV^\pm, \qquad e_{i,d}, f_{i,d} \mapsto z_{i 1}^d, \quad \forall d \in m_i\BZ
		\end{equation}
		defined as in Subsection~\ref{sub:upsilon}, but with respect to the zeta functions \eqref{eqn:zeta strongly twisted}. This allows us to define the pairings which are non-degenerate in the $\oCV^\pm$ factor
		\begin{align}
			&\bar{\tUU}^+ \otimes \oCV^- \xrightarrow{\langle \cdot, \cdot \rangle} \BC \label{eqn:pair twisted} \\
			&\oCV^+ \otimes \bar{\tUU}^- \xrightarrow{\langle \cdot, \cdot \rangle} \BC \label{eqn:pair opposite twisted}
		\end{align}
		as in Definition~\ref{def:pair}, but with $d_a \in m_{i_a} \BZ$ instead of $d_a \in \BZ$ in formulas \eqref{eqn:pair formula}-\eqref{eqn:pair formula opposite}. 
		
		\medskip 
		
		\begin{claim}
			\label{claim:key}
			
			The pairings \eqref{eqn:pair twisted} and \eqref{eqn:pair opposite twisted} have the property that
			\begin{equation}
				\label{eqn:key}
				\Big \langle \emph{Ker }\bar{\tUpsilon}^+, \oCS^- \Big \rangle = \Big \langle \oCS^+, \emph{Ker }\bar{\tUpsilon}^- \Big \rangle = 0
			\end{equation}
			
		\end{claim}
		
		\medskip
		
		\noindent We will prove Claim~\ref{claim:key} after showing how it allows us to complete the proof of Proposition~\ref{prop:twisted wheel strongly}. By \eqref{eqn:key}, the pairings \eqref{eqn:pair twisted}-\eqref{eqn:pair opposite twisted} descend to pairings
		\begin{align}
			& \Big(\text{Im }\bar{\tUpsilon}^+ \Big) \otimes \oCS^- \xrightarrow{\langle \cdot, \cdot \rangle} \BC \label{eqn:pair descended twisted} \\
			&\oCS^+ \otimes \Big( \text{Im }\bar{\tUpsilon}^- \Big) \xrightarrow{\langle \cdot, \cdot \rangle} \BC \label{eqn:pair opposite descended twisted}
		\end{align}
		which are non-degenerate in the $\oCS^\pm$ factor. Since $\oCS^\pm \supseteq \text{Im }\bar{\tUpsilon}^\pm$, if all the above were finite-dimensional $\BC$-vector spaces, it would immediately follow that
		\begin{equation}
			\label{eqn:equality}
			\oCS^\pm = \text{Im }\bar{\tUpsilon}^\pm
		\end{equation}
		and Proposition~\ref{prop:twisted wheel strongly} would be proved. In the current infinite-dimensional setup, one proves \eqref{eqn:equality} by using the combinatorics of words argument of \cite[Theorem 2.13]{N Wheel} and \cite[Theorem 2.11]{N Arbitrary}. As a consequence of \eqref{eqn:equality}, \eqref{eqn:pair descended twisted}-\eqref{eqn:pair opposite descended twisted} yields two pairings 
		\[
		\oCS^+ \otimes \oCS^- \xrightarrow{\langle \cdot, \cdot \rangle} \BC 
		\] 
		which coincide due to a calculation analogous to \cite[Remark 3.16]{N Wheel}. As we have 
		\[
		\bar{\UU}^\pm = \bar{\tUU}^{\pm} \Big/ \Big( \text{Ker }\bar{\tUpsilon}^{\pm} \Big)
		\]
		by definition, the first isomorphism theorem implies that \eqref{eqn:bar tupsilon} descends to \eqref{eqn:upsilon twisted strongly}. \end{proof}
	
	\medskip 
	
	\begin{proof} \emph{of Claim~\ref{claim:key}:} We sketch the usual argument given in \cite[Proposition 3.3]{N Wheel}, see also \cite[Proposition 2.22]{N Symmetric} and \cite[Lemma 3.23]{N New}. Assume $|q|>1$ for ease of explanation, although the entire argument is algebraic and will not depend on this restriction. The principle is to start from formula \eqref{eqn:pair formula} and gradually move the contours of integration from $|z_1| \gg \cdots \gg |z_n|$ to $|z_1| = \cdots = |z_n|$. Suppose for the purpose of illustration that we performed the first $n-1$ steps of this algorithm, in other words we have already moved the variables $z_1,\dots,z_{n-1}$ toward one and the same contour. Along the way, we inevitably pick up poles at
		\begin{equation}
			\label{eqn:poles}
			z_{a_s} = z_{a_{s-1}} q^{-d_{ii}}  = z_{a_{s-2}} q^{-2d_{ii}}  = \cdots = z_{a_1} q^{(1-s)d_{ii}} 
		\end{equation}
		for various $n-1 \geq a_s > a_{s-1} > \cdots > a_1 \geq 1$ with $i_{a_s} = \cdots = i_{a_1} =: i$. Then at the $n$-th step of the algorithm (see \cite[Claim 3.5]{N Wheel} and \cite[Claim 3.24]{N New}), we must ensure that moving the variable $z_n$ toward the other variables does not produce any poles of a different nature from \eqref{eqn:poles}. Indeed, the poles all come from the zeroes of
		\[
		\ozeta_{i_ni_{a_s}} \left(\frac {z_n}{z_{a_s}} \right), \ \dots, \ \ozeta_{i_ni_{a_1}} \left(\frac {z_n}{z_{a_1}} \right)
		\]
		In other words, there might be simple poles at the following values
		\[
		z_n = z_{a_s} \omega^{-k} q^{-d_{\sigma^k(i_n)i}}, \ \dots, \ z_n = z_{a_1} \omega^{-k} q^{-d_{\sigma^k(i_n)i}}
		\]
		If we let $j = \sigma^k(i_n)$ and interpret $z_n$ as a variable of color $j$ instead of color $i_n$ (see formulas \eqref{eqn:formula 1}-\eqref{eqn:formula 2}), then the above poles take the form
		\begin{equation}
			\label{eqn:poles 3}
			z_n = z_{a_s} q^{-d_{ji}}, \ \dots, \ z_n = z_{a_1} q^{-d_{ji}}
		\end{equation} 
		However, as explained in the proof of \cite[Lemma 3.23]{N New}, we only need to move the contour of $z_n$ from $0$ to the midpoint of the geometric progression \eqref{eqn:poles}, i.e. to
		\begin{equation}
			\label{eqn:where it should}
			|z_n| = \left|z_{a_s} q^{\frac {(1-s)d_{ii}}2}\right|
		\end{equation}
		Therefore, let us restrict attention to the following pole among \eqref{eqn:poles 3}
		\begin{equation}
			\label{eqn:fake pole}
			z_n = z_{a_t} q^{-d_{ji}}
		\end{equation}
		with $t \geq \frac {s+1}2$. The fact that $F$ vanishes at the specialization \eqref{eqn:twisted wheel strongly} implies that we have a zero due to the presence of the triangle
		
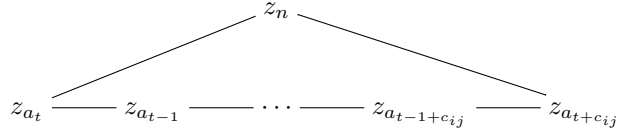
\[
		\xymatrix{& & z_n \ar@{-}[rrd] & & \\
			z_{a_t} \ar@{-}[rru] &  z_{a_{t-1}} \ar@{-}[l] & \cdots \ar@{-}[l] &  z_{a_{t-1+c_{ij}}} \ar@{-}[l] & z_{a_{t+c_{ij}}} \ar@{-}[l]}
		\]
		with all the variables in the bottom row among the geometric progression \eqref{eqn:poles}. Thus, \eqref{eqn:fake pole} are not real poles, and thus we can move the contour of $z_n$ all the way to \eqref{eqn:where it should} without picking up any residues. This completes the proof of Claim~\ref{claim:key}. \end{proof} 
	
	\medskip 
	
	\subsection{Twisted quantum loop algebras}
	\label{sub:twisted quantum loop}
	
	In the generality of Subsection~\ref{sub:quivers automorphism}, we may sometimes require the unfolded and folded zeta functions to satisfy the relation
	\begin{equation}
		\label{eqn:limit folded}
		\frac {\ozeta_{\oi \oj}(x)}{\ozeta_{\oj \oi}(x^{-1})} = \prod_{k=1}^m \frac {\zeta_{\sigma^\ell(i_0)\sigma^k(j_0)} (x\omega^{\ell-k})}{\zeta_{\sigma^k(j_0) \sigma^{\ell}(i_0)} (x^{-1}\omega^{k-\ell})}
	\end{equation}
	for all $\oi, \oj \in \oI$ and $\ell \in \BZ/m\BZ$, where $i_0,j_0 \in I$ denote the distinguished preimages of $\oi, \oj$. This can be achieved by appropriately fixing the monomials that measure the discrepancy between the LHS and RHS of \eqref{eqn:unfolded zeta}, and we note that it holds in the particular case of symmetrizable Kac-Moody Lie algebras due to formula \eqref{eqn:zeta strongly twisted exchange}. Whenever formula \eqref{eqn:limit folded} holds, we define the double quantum loop algebras
	\[
	\UU \text{ and } \oUU
	\]
	as in Subsection~\ref{sub:double}. We may now apply the discussion in Subsection~\ref{sub:wheel conditions} to the twisted zeta functions \eqref{eqn:zeta strongly twisted}, and recall the principle that the wheel conditions \eqref{eqn:twisted wheel strongly} are dual to, and hence determine, all relations between $e$'s and $f$'s in the corresponding quantum loop algebra.
	
	\medskip 
	
	\begin{definition}
		\label{def:twisted clunky}
		
		Suppose $\fg$ is a strongly symmetrizable Kac-Moody Lie algebra, endowed with an automorphism $\sigma$ of order $m$. The corresponding twisted quantum loop algebra is defined as
		\[
		U_q(L\fg^{\sigma}) = \BC \Big \langle e_{i,d}, f_{i,d}, \ph^\pm_{i,d'} \Big \rangle_{i \in I, d \in \BZ, d' \geq 0} \Big/ \Big(\text{relations \eqref{eqn:rel quantum 0 twisted strongly}-\eqref{eqn:rel quantum 9 twisted strongly}} \Big) 
		\]
		where 
		\begin{equation}
			\label{eqn:rel quantum 0 twisted strongly}
			e_{\bigcirc} = f_{\bigcirc} = 0 \text{ as }\bigcirc \text{ runs over the wheels \eqref{eqn:twisted wheel strongly}}
		\end{equation} 
		\begin{equation}
			\label{eqn:rel quantum 1 twisted strongly}
			e_{\sigma(i),d} = e_{i,d} \omega^d, f_{\sigma(i),d} = f_{i,d} \omega^d, \ph^\pm_{\sigma(i),d} = \ph^\pm_{i,d} \omega^{\pm d}
		\end{equation} 
		\begin{equation}
			\label{eqn:rel quantum 2 twisted strongly}
			e_i(x) e_j(y) \prod_{k=1}^m \left( x\omega^k - y q^{d_{\sigma^k(i)j}} \right) = e_j(y) e_i(x) \prod_{k=1}^m \left( x\omega^k q^{d_{\sigma^k(i)j}} - y \right)
		\end{equation}
		\begin{equation}
			\label{eqn:rel quantum 3 twisted strongly}
			f_j(y) f_i(x) \prod_{k=1}^m \left( x\omega^k - y q^{d_{\sigma^k(i)j}} \right) = f_i(x) f_j(y) \prod_{k=1}^m \left( x\omega^k q^{d_{\sigma^k(i)j}} - y \right)
		\end{equation}
		\begin{equation}
			\label{eqn:rel quantum 4 twisted strongly}
			\left[\ph^+_{i}(x), \ph^+_{j}(y)\right] = \left[\ph^+_{i}(x), \ph^-_{j}(y)\right] = \left[\ph^-_{i}(x), \ph^-_{j}(y)\right] = 0, \ \ph^+_{i,0}\ph^-_{i,0} = 1
		\end{equation} 
		\begin{align}
			&\ph^\pm_i(x) e_j(y) = e_j(y) \ph^\pm_i(x) \prod_{k=1}^m \frac {x\omega^k q^{d_{\sigma^k(i)j}}-y}{x\omega^k-y q^{d_{\sigma^k(i)j}}} \label{eqn:rel quantum 5 twisted strongly} \\
			&f_j(y)\ph^\pm_i(x)  = \ph^\pm_i(x)  f_j(y) \prod_{k=1}^m \frac {x\omega^k q^{d_{\sigma^k(i)j}}-y}{x\omega^k-y q^{d_{\sigma^k(i)j}}} \label{eqn:rel quantum 7 twisted strongly} 
		\end{align}
		\begin{equation}
			\label{eqn:rel quantum 9 twisted strongly}
			\Big[e_{i}(x), f_{j}(y)\Big] = \left( \sum_{k=1}^m \frac {\delta_{\sigma^k(i)j}\delta(\frac{x\omega^k}{y})}{m_i} \right) \frac {\ph^+_i(x) - \ph^-_j(y)}{q_i-q_i^{-1}}
		\end{equation}
		for all $i,j \in I$. 
		
	\end{definition}
	
	\medskip 
	
	\noindent We will now relate Definition~\ref{def:twisted clunky} to the existing definition of twisted quantum affine algebras ($\fg$ of finite type), and provide explicit formulas for twisted quantum toroidal algebras ($\fg$ of affine type). To make the aforementioned formulas more presentable, we will henceforth replace the zeta functions \eqref{eqn:zeta strongly} by
	\begin{equation}
		\label{eqn:zeta prime strongly}
		\zeta'_{ij}(x) \sim \begin{cases} \frac {1-xq^{d_{ii}}}{1-x} &\text{if }i = j \\ 1-xq^{d_{ij}} &\text{if }i \sim j \\ 1 &\text{otherwise} \end{cases}
	\end{equation}
	where $i \sim j$ means that $d_{ij} < 0$ (i.e. there is an edge between $i$ and $j$ in the Dynkin diagram of $\fg$). Therefore, the twisted zeta functions \eqref{eqn:zeta strongly twisted} should be replaced by 
	\begin{equation}
		\label{eqn:zeta prime strongly twisted}
		\ozeta'_{ij}(x) = \prod_{k=1}^m \zeta'_{\sigma^k(i)j}(\omega^kx) 
	\end{equation}
	(see Examples~\ref{ex:a2} and \ref{ex:a3} for the cases of type $A_2$ and $A_3$, respectively). We are allowed to make the modification $\zeta \leadsto \zeta'$ because the prime zeta functions only differ from their non-primed counterparts by certain powers of $1-x$ \footnote{It is a general principle (that is quite easy to prove) that if we have collections of zeta functions
		\[
		\{\zeta_{ij}(x)\}_{i,j \in I} \quad \text{and} \quad \{\zeta'_{ij}(x)\}_{i,j \in I} \quad \text{such that} \quad \frac {\zeta'_{ij}(x)}{\zeta_{ij}(x)} = \frac {\zeta'_{ji}(x^{-1})}{\zeta_{ji}(x^{-1})}
		\]
		the corresponding algebras $\CS, \UU$ and $\CS',\UU'$ of Subsections~\ref{sub:zeta}-\ref{sub:double} are isomorphic to each other.}. With respect to the conventions \eqref{eqn:zeta prime strongly}-\eqref{eqn:zeta prime strongly twisted}, relations \eqref{eqn:rel quantum 2 twisted strongly}-\eqref{eqn:rel quantum 3 twisted strongly} can be strengthened to
	\begin{align} 
		\primeeqtag{eqn:rel quantum 2 twisted strongly}
		\label{eqn:rel prime quantum 2 twisted strongly}
		e_i(x) e_j(y) \prod_{1 \leq k \leq m}^{d_{\sigma^k(i)j} 
			\neq 0} \left( x\omega^k - y q^{d_{\sigma^k(i)j}} \right) = e_j(y) e_i(x) \prod_{1 \leq k \leq m}^{d_{\sigma^k(i)j} 
			\neq 0} \left( x\omega^k q^{d_{\sigma^k(i)j}} - y \right) \\
		\primeeqtag{eqn:rel quantum 3 twisted strongly}
		\label{eqn:rel prime quantum 3 twisted strongly}
		f_j(y) f_i(x) \prod_{1 \leq k \leq m}^{d_{\sigma^k(i)j} 
			\neq 0} \left( x\omega^k - y q^{d_{\sigma^k(i)j}} \right) = f_i(x) f_j(y) \prod_{1 \leq k \leq m}^{d_{\sigma^k(i)j} \neq 0} \left( x\omega^k q^{d_{\sigma^k(i)j}} - y \right)
	\end{align}
	However, this comes at the cost of removing relations \eqref{eqn:rel quantum 0 twisted strongly} for various wheels $\bigcirc$ that disappear due to $\zeta'_{ij}(x)$, $\ozeta'_{ij}(x)$ having fewer linear factors than $\zeta_{ij}(x)$, $\ozeta_{ij}(x)$. Thus, the primed zeta functions will make the subsequent formulas easier to follow.
	
	\medskip 
	
	\subsection{Twisted quantum affine algebras}
	\label{sub:affine}
	
	While Definition~\ref{def:twisted clunky} provides a general construction of twisted quantum loop algebras, the experienced reader might observe that relations \eqref{eqn:rel quantum 0 twisted strongly} do not match the Drinfeld-Serre relations usually found in the literature (\cite{D inj, Dr}) for finite type $\fg$. This happens for the same reason that made formulas \eqref{eqn:km 1} and \eqref{eqn:km 2} provide different, but equivalent, definitions of untwisted quantum loop algebras: there is simply nothing canonical about the relations $e_{\bigcirc} = f_{\bigcirc} = 0$, and we can replace them by any other products of $e$'s and $f$'s which are dual to the wheel conditions \eqref{eqn:twisted wheel strongly}. In particular, we will show that the usual Drinfeld-Serre relations provide such products in the case of finite type $\fg$. To this end, let us recall that the following pictures represent the only non-trivial automorphisms of a finite type Dynkin diagram.
	
	\[
	\text{Type }A_{2n}
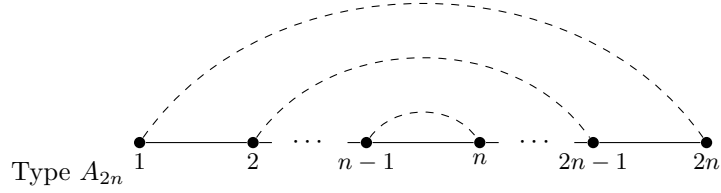
	\begin{tikzpicture}[
		>=Stealth,
		vertex/.style={circle, fill=black, inner sep=1.5pt},
		lab/.style={below=3pt, font=\small}
		]
		\node[vertex, label={[lab]above:$1$}] (1) at (0,0) {};
		\node[vertex, label={[lab]above:$2$}] (2) at (1.5,0) {};
		\node[vertex, label={[lab]above:$n-1$}] (n-1) at (3,0) {};
		\node[vertex, label={[lab]above:$n$}] (n) at (4.5,0) {};
		\node[vertex, label={[lab]above:$2n-1$}] (2n-1) at (6,0) {};
		\node[vertex, label={[lab]above:$2n$}] (2n) at (7.5,0) {};
		
		\draw[-] (1) -- (2);
		\draw (2) -- (1.75,0);
		\node at (2.25,0) {$\cdots$};
		\draw (2.75,0) -- (n-1);
		\draw[-] (n-1) -- (n);
		\draw (n) -- (4.75,0);
		\node at (5.25,0) {$\cdots$};
		\draw (5.75,0) -- (2n-1);
		\draw[-] (2n-1) -- (2n);
		
		\draw[dashed, bend left=55] (1) to (2n);
		\draw[dashed, bend left=55] (2) to (2n-1);
		\draw[dashed, bend left=55] (n-1) to (n);
	\end{tikzpicture}
	\]
	
	\[
	\text{Type }A_{2n+1}
	
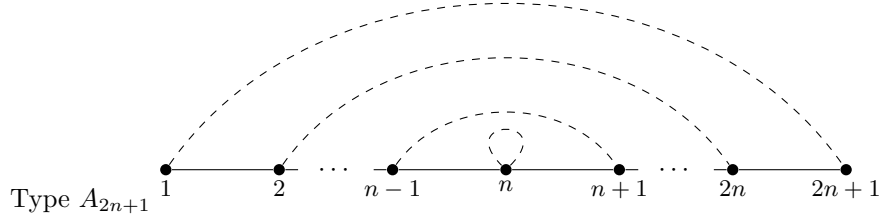
\begin{tikzpicture}[
		>=Stealth,
		vertex/.style={circle, fill=black, inner sep=1.5pt},
		lab/.style={below=3pt, font=\small}
		]
		\node[vertex, label={[lab]above:$1$}] (1) at (0,0) {};
		\node[vertex, label={[lab]above:$2$}] (2) at (1.5,0) {};
		\node[vertex, label={[lab]above:$n-1$}] (n-1) at (3,0) {};
		\node[vertex, label={[lab]above:$n$}] (n) at (4.5,0) {};
		\node[vertex, label={[lab]above:$n+1$}] (n+1) at (6,0) {};
		\node[vertex, label={[lab]above:$2n$}] (2n) at (7.5,0) {};
		\node[vertex, label={[lab]above:$2n+1$}] (2n+1) at (9,0) {};
		
		\draw[-] (1) -- (2);
		\draw (2) -- (1.75,0);
		\node at (2.25,0) {$\cdots$};
		\draw (2.75,0) -- (n-1);
		\draw[-] (n-1) -- (n);
		\draw[-] (n) -- (n+1);
		\draw (n+1) -- (6.25,0);
		\node at (6.75,0) {$\cdots$};
		\draw (7.25,0) -- (2n);
		\draw[-] (2n) -- (2n+1);
		
		\draw[dashed, bend left=55] (1) to (2n+1);
		\draw[dashed, bend left=55] (2) to (2n);
		\draw[dashed, bend left=55] (n-1) to (n+1);
		\draw[dashed, -] (n) to[out=45, in=135, looseness=20] (n);
	\end{tikzpicture}
	\]
	
	\[
	\text{Type }D_{n}
	
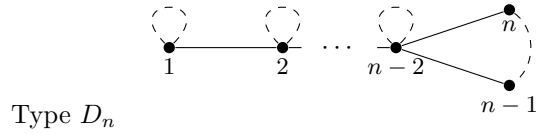
\begin{tikzpicture}[
		>=Stealth,
		vertex/.style={circle, fill=black, inner sep=1.5pt},
		lab/.style={below=3pt, font=\small}
		]
		
		\node[vertex, label={[lab]above:$1$}] (1) at (0,0) {};
		\node[vertex, label={[lab]above:$2$}] (2) at (1.5,0) {};
		\node[vertex, label={[lab]above:$n-2$}] (n-2) at (3,0) {};
		\node[vertex, label={[lab]above:$n-1$}] (n-1) at (4.5,-0.5) {};
		\node[vertex, label={[lab]above:$n$}] (n) at (4.5,0.5) {};
		
		\draw[-] (1) -- (2);
		\draw (2) -- (1.75,0);
		\node at (2.25,0) {$\cdots$};
		\draw (2.75,0) -- (n-2);
		\draw[-] (n-2) -- (n-1);
		\draw[-] (n-2) -- (n);
		
		\draw[dashed, bend left=55] (n) to (n-1);
		\draw[dashed, -] (1) to[out=45, in=135, looseness=20] (1);
		\draw[dashed, -] (2) to[out=45, in=135, looseness=20] (2);
		\draw[dashed, -] (n-2) to[out=45, in=135, looseness=20] (n-2);
	\end{tikzpicture}
	\]
	
	\[
	\text{Type }E_{6}
	
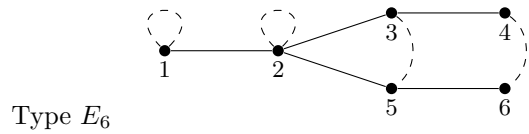
\begin{tikzpicture}[
		>=Stealth,
		vertex/.style={circle, fill=black, inner sep=1.5pt},
		lab/.style={below=3pt, font=\small}
		]
		
		\node[vertex, label={[lab]above:$1$}] (1) at (0,0) {};
		\node[vertex, label={[lab]above:$2$}] (2) at (1.5,0) {};
		\node[vertex, label={[lab]above:$3$}] (3) at (3,0.5) {};
		\node[vertex, label={[lab]above:$4$}] (4) at (4.5,0.5) {};
		\node[vertex, label={[lab]above:$5$}] (5) at (3,-0.5) {};
		\node[vertex, label={[lab]above:$6$}] (6) at (4.5,-0.5) {};

		\draw[-] (1) -- (2);
		\draw[-] (2) -- (3);
		\draw[-] (3) -- (4);
		\draw[-] (2) -- (5);
		\draw[-] (5) -- (6);
		
		\draw[dashed, bend left=55] (3) to (5);
		\draw[dashed, bend left=55] (4) to (6);
		\draw[dashed, -] (1) to[out=45, in=135, looseness=20] (1);
		\draw[dashed, -] (2) to[out=45, in=135, looseness=20] (2);
	\end{tikzpicture}
	\]
	
	\[
	\text{Type }D_{4}
	
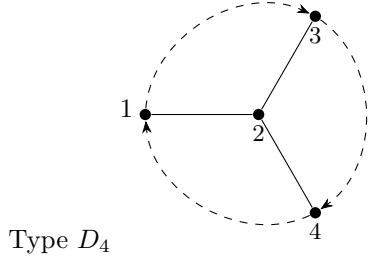
\begin{tikzpicture}[
		>=Stealth,
		vertex/.style={circle, fill=black, inner sep=1.5pt},
		lab/.style={below=3pt, font=\small}
		]
		
		\node[vertex, label={[left=1pt]:$1$}] (1) at (0,0) {};
		\node[vertex, label={[lab]above:$2$}] (2) at (1.5,0) {};
		\node[vertex, label={[lab]above:$3$}] (3) at (2.25,1.3) {};
		\node[vertex, label={[lab]above:$4$}] (4) at (2.25,-1.3) {};

		\draw[-] (1) -- (2);
		\draw[-] (2) -- (3);
		\draw[-] (2) -- (4);
		
		\draw[dashed, ->, bend left=55] (1) to (3);
		\draw[dashed, ->, bend left=55] (3) to (4);
		\draw[dashed, ->, bend left=55] (4) to (1);
	\end{tikzpicture}
	\]
	
	\medskip
	
	\begin{proposition}
		\label{prop:affine}
		
		For any finite type Lie algebra $\fg$ endowed with an automorphism $\sigma$ of order $m$, the usual correspondence of generators yields an isomorphism between the twisted quantum loop algebra $\twU$ of Definition~\ref{def:twisted clunky} and the so-called twisted quantum affine algebra 
		\[
		\BC\Big\langle e_{i,d}, f_{i,d}, \varphi_{i,d'}^\pm \Big\rangle_{i \in I, d \in \BZ, d' \in \BN} \Big / \Big(\text{\eqref{eqn:rel quantum 1 twisted strongly}, \eqref{eqn:rel prime quantum 2 twisted strongly}, \eqref{eqn:rel prime quantum 3 twisted strongly}, \eqref{eqn:rel quantum 4 twisted strongly}-\eqref{eqn:rel quantum 9 twisted strongly} and \eqref{eqn:cubic serre 1}-\eqref{eqn:cubic serre 3 in A2n}}\Big)
		\]
		where the latter relations depend on the following cases:
		
		\medskip 
		
		Case 1: $i \sim j$ and there are no other vertices in the orbit of $i$  connected to $j$. Assuming that $i$ is not one of the middle two nodes in type $A_{2n}$, we have: 
		\begin{equation}
			\label{eqn:cubic serre 1}
			\mathrm{Sym}  \Big[e_j(y) e_i(x) e_i(x') - (q^{m_i} + q^{-m_i}) e_i(x) e_j(y) e_i(x') +  e_i(x) e_i(x') e_j(y) \Big] = 0
		\end{equation}
		
		\medskip 
		
		Case 2: $i\sim j$ and there are exactly $m$ vertices in the orbit of $i$ connected to $j$:
		\begin{equation}
			\label{eqn:cubic serre 3}
			\begin{split}
				\mathrm{Sym} \Big[ \frac{x^mq^{2m} - {x'}^m}{xq^{2} - x'} \Big( e_j(y) e_i(x) e_i(x')& - (q^{m} + q^{-m}) e_i(x) e_j(y) e_i(x') \\
				& +  e_i(x) e_i(x') e_j(y) \Big) \Big] = 0
			\end{split}
		\end{equation}
		
		\medskip
		
		Case 3: at the middle node $n$ in type $A_{2n}$:
		\begin{equation}
			\label{eqn:cubic serre 1 in A2n}
			\mathrm{Sym} \Big[ \Big( q^{-3/2}x - (q^{1/2} + q^{-1/2}){x'} + q^{3/2}{x''} \Big) e_n(x)e_n(x')e_n(x'')\Big] =0
		\end{equation}
		
		Case 4: at the middle node $n$ in type $A_{2n}$: \footnote{In this formula we correct a typo from \cite{H KR} and \cite{W QQ}. The formula matches the notation of \cite{Dr}.}
		\begin{equation}
			\label{eqn:cubic serre 3 in A2n}
			\begin{split}
				\mathrm{Sym} \Big[ (xq + x') \Big( e_{n-1}(y) e_n(x) e_n(x')& - (q + q^{-1}) e_n(x) e_{n-1}(y) e_n(x') \\
				& +  e_n(x) e_n(x') e_{n-1}(y) \Big) \Big] = 0
			\end{split}
		\end{equation}
		
		\noindent Above, ``\emph{Sym}" denotes symmetrization with respect to the variables $x$,$x'$ and $x''$. Though we will not write them down explicitly, it is implied that one also imposes the opposite relations of the ones above with $e$'s replaced by $f$'s (i.e. whenever we have $\cdots ee'\cdots = 0$ we also impose $\cdots f'f\cdots = 0$).
		
	\end{proposition}
	
	\medskip 
	
	\noindent Proposition~\ref{prop:affine} involves a case-by-case check of the various finite type Dynkin diagram automorphisms, and we will establish it in Appendix A. In all cases, the gist of the Proposition is that
	\[
	\Big \langle \text{LHS of \eqref{eqn:cubic serre 1}-\eqref{eqn:cubic serre 3 in A2n}}, F \Big \rangle = 0 \ \Leftrightarrow \ \Big(F \text{ satisfies \eqref{eqn:twisted wheel strongly}} \Big) \ \Leftrightarrow \ \Big \langle e_{\bigcirc \text{ as in \eqref{eqn:twisted wheel strongly}}}, F \Big \rangle = 0 
	\]
	which establishes the equivalence of the generators-and-relations presentations of Definition~\ref{def:twisted clunky} and Proposition~\ref{prop:affine}. Each of these presentations has its own advantages and disadvantages: the one of Definition~\ref{def:twisted clunky} is unified and the relations have a shuffle algebra explanation, but they are rather clunky. On the other hand, the presentation of Proposition~\ref{prop:affine} features relations that are easier to work with, but their specific form differs on a case-by-case basis.
	
	\medskip 
	
	\subsection{Twisted quantum toroidal algebras}
	\label{sub:toroidal}
	
	Another important class of strongly symmetrizable Kac--Moody $\fg$ are the affine Lie algebras of type other than $A_1$ (the interested reader may consult Appendix~\ref{appendix2} for a list of all non-trivial automorphisms of affine Dynkin diagrams in types other than $A_{n-1}^{(1)}$, cf. \cite[Section 14]{Lu}). If $\fg$ is an affine Lie algebra endowed with a diagram automorphism, then the twisted quantum loop algebras of Definition~\ref{def:twisted clunky} will be called twisted quantum toroidal algebras, and the following result gives an alternate (i.e. Drinfeld-Serre) presentation for them.  
	
	\medskip
	
	\begin{proposition}
		\label{prop:toroidal}
		
		For any affine Lie algebra $\fg$ (of type other than $A_{n-1}^{(1)}$ and $D_3^{(2)}$) endowed with an automorphism $\sigma$ of order $m$, the usual correspondence of generators yields an isomorphism between the twisted quantum loop algebra $\twU$ of Definition~\ref{def:twisted clunky} and the twisted quantum toroidal algebra defined below
		\[
		\BC \Big\langle e_{i,d}, f_{i,d}, \varphi_{i,d'}^\pm \Big\rangle_{i \in I, d \in \BZ, d' \in \BN} \Big / \Big(\text{\eqref{eqn:rel quantum 1 twisted strongly}, \eqref{eqn:rel prime quantum 2 twisted strongly}, \eqref{eqn:rel prime quantum 3 twisted strongly}, \eqref{eqn:rel quantum 4 twisted strongly}-\eqref{eqn:rel quantum 9 twisted strongly} and \eqref{eqn:rel toroidal serre 1}-\eqref{eqn:rel toroidal serre 4}}\Big)
		\]
		where the latter relations depend on the following cases:
		
		\medskip
		
		Case 1: $i \sim j$ and there are no other vertices in the orbit of $i$  connected to $j$. Assuming that $\sigma(i)$ is not connected to $i$, we have: 
		
		\begin{equation}
			\label{eqn:rel toroidal serre 1}
			\mathrm{Sym} \Big[\sum_{k=0}^{1-c_{ij}} (-1)^k \binom{1-c_{ij}}{k}_{q_i^{m_i}}  e_i(x_1) \cdots e_i(x_k) e_j(y) e_i(x_{k+1}) \cdots e_i(x_{1-c_{ij}}) \Big] = 0
		\end{equation}
		
		\medskip
		
		Case 2: $i\sim j$ and there are exactly $r > 1$ vertices in the orbit of $i$ connected to $j$ (note that in this case we have $r=m$ and $m_i =1$, unless we are in type $D_{2n+1}^{(1)}$ with an order 4 automorphism, in which case we can have $r=2$). This can happen only when $c_{ij} = -1$:
		
		\begin{equation}
			\label{eqn:rel toroidal serre 2}
			\begin{split}
				\mathrm{Sym} \Big[ \frac{x^{rm_i}q_i^{2rm_i} - {x'}^{rm_i}}{x^{m_i}q_i^{2m_i} - {x'}^{m_i}} \Big( e_j(y) e_i(x) e_i(x')& - (q_i^{rm_i} + q_i^{-rm_i}) e_i(x) e_j(y) e_i(x') \\
				& +  e_i(x) e_i(x') e_j(y) \Big) \Big] = 0
			\end{split}
		\end{equation}
		
		\medskip
		
		Case 3: $i \sim \sigma(i)$, which can happen only at the middle node $n$ in types $C_n^{(1)}$ and $D_{2n+1}^{(k)}$ for $k = 1,2$ (note that in this case, $m_n =1$ if $m =2$ and $m_n = 2$ if $m = 4$):
		
		\begin{equation}
			\label{eqn:rel toroidal serre 3}
			\mathrm{Sym} \Big[ \Big( q_n^{-\frac 32 m_n}x^{m_n} - (q_n^{\frac 12 m_n} + q_n^{-\frac 12 m_n}){x'}^{m_n} + q_n^{\frac 32 m_n}{x''}^{m_n} \Big) e_n(x)e_n(x')e_n(x'')\Big] = 0
		\end{equation}
		
		\medskip
		Case 4: $i \sim \sigma(i)$, which can happen only at the middle node $n$ in types $C_n^{(1)}$ and $D_{2n+1}^{(k)}$ for $k=1,2$: 
		\begin{equation}
			\label{eqn:rel toroidal serre 4}
			\begin{split}
				\mathrm{Sym} \Big[ (x^{m_n}q_n^{m_n} + {x'}^{m_n}) \Big( e_{n-1}(y) e_n(x) e_n(x')& - (q_n^{m_n} + q_n^{-{m_n}}) e_n(x) e_{n-1}(y) e_n(x') \\
				& +  e_n(x) e_n(x') e_{n-1}(y) \Big) \Big] = 0
			\end{split}
		\end{equation}
		Above, ``\emph{Sym}" denotes symmetrization with respect to the $x$ variables.
		
	\end{proposition}
	
	\medskip
	
	\noindent Proposition~\ref{prop:toroidal} can be checked following the same method as in twisted affine types of Subsection~\ref{sub:affine} (see Appendix~\ref{appendix1}), namely by showing that
	\[
	\Big \langle \text{LHS of \eqref{eqn:rel toroidal serre 1}-\eqref{eqn:rel toroidal serre 4}}, F \Big \rangle = 0 \ \Leftrightarrow \ \Big(F \text{ satisfies \eqref{eqn:twisted wheel strongly}} \Big) \ \Leftrightarrow \ \Big \langle e_{\bigcirc \text{ as in \eqref{eqn:twisted wheel strongly}}}, F \Big \rangle = 0 
	\]
	which establishes the equivalence of the generators-and-relations presentations of Definition~\ref{def:twisted clunky} and Proposition~\ref{prop:toroidal}, respectively.
	
	\medskip 
	
	\subsection{The cyclic quiver} 
	\label{sub:toroidal cycle} 
	
	We conclude with the case when $\fg$ is of type $A_{n-1}^{(1)}$, whose Dynkin diagram is a cycle with vertices labeled by $I = \BZ/n \BZ = \{1,\dots,n\}$. Consider the automorphism 
	\begin{equation}
		\label{eqn:sigma cycle}
		\sigma : \BZ/n\BZ \rightarrow \BZ/n\BZ, \qquad  \sigma(i) = i+k 
	\end{equation}
	for some fixed $k \in \BZ/n\BZ$. It corresponds to a rotation by $\frac {2\pi k}n$, see below for $k=1$:
	
	
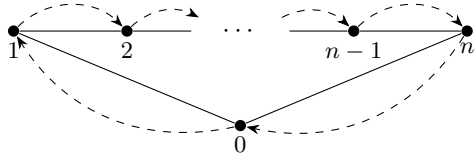
\[
	\begin{gathered}
		\begin{tikzpicture}[
			>=Stealth,
			vertex/.style={circle, fill=black, inner sep=1.5pt},
			lab/.style={below=3pt, font=\small}
			]
			\node[vertex, label={[lab]above:$1$}] (1) at (0,0) {};
			\node[vertex, label={[lab]above:$2$}] (2) at (1.5,0) {};
			\node[vertex, label={[lab]above:$n-1$}] (n-1) at (4.5,0) {};
			\node[vertex, label={[lab]above:$n$}] (n) at (6,0) {};
			\node[vertex, label={[lab]above:$0$}] (0) at (3,-1.25) {};
			
			\draw[-] (1) -- (2);
			\draw (2) -- (2.35,0);
			\node at (3,0) {$\cdots$};
			\draw (3.65,0) -- (n-1);
			\draw[-] (n-1) -- (n);
			\draw[-] (1) -- (0);
			\draw[-] (n) -- (0);
			
			\draw[dashed, ->, bend left=45] (1) to (2);
			\draw[dashed, ->, bend left=35] (2) to (2.45,0.18);
			\draw[dashed, ->, bend left=35] (3.55,0.18) to (n-1);
			\draw[dashed, ->, bend left=45] (n-1) to (n);
			\draw[dashed, ->, bend left=35] (n) to (0);
			\draw[dashed, ->, bend left=35] (0) to (1);
			
		\end{tikzpicture}
		\\[4pt]
	\end{gathered}
	\]
	
	\noindent In general, the order of $\sigma$ of \eqref{eqn:sigma cycle} is $m = \frac n{\gcd(n,k)}$. We have
	\[
	\zeta'_{ij}(x) = \frac {(1-xq^2)^{\delta_{ij}} (q^{-1} - x^{-1})^{\delta_{i+1,j}}(1-xq^{-1})^{\delta_{i-1,j}}}{(1-x)^{\delta_{ij}}}
	\]
	for all $i,j \in \BZ/n\BZ$ (the Kronecker delta is taken modulo $n$). Therefore, \eqref{eqn:folded zeta} reads
	\[
	\ozeta'_{\oi\oj}(x) = \frac {(1-xq^2)^{\delta_{\oi\oj}} (q^{-1} - x^{-1}\omega^{\frac {n\ell}m\delta_{\oj 1}})^{\delta_{\oi+1,\oj}}(1-x\omega^{\frac {n\ell}m \delta_{\oi 1}}q^{-1})^{\delta_{\oi-1,\oj}}}{(1-x)^{\delta_{\oi\oj}}}
	\]
	for all $\oi,\oj \in \BZ/\frac nm\BZ = \{1,\dots,\frac nm\}$ (the Kronecker delta is taken modulo $\frac nm$), where 
	\[
	\frac {mk}n \cdot \ell \equiv 1 \text{ mod }m
	\]
	uniquely determines $\ell \in \BZ/m\BZ$. Thus, we recognize that the twisted quantum toroidal algebra associated to the automorphism \eqref{eqn:sigma cycle} is
	\[
	U_q(L\fg^\sigma) \cong U_q(\mathfrak{sl}_{\frac nm, tor}) \Big|_{C = 1, D = \omega^{\ell}}
	\] 
	where the right-hand side refers to the deformed quantum toroidal algebra introduced in \cite[Section 5]{GKV} (it provides an interesting algebra even if $m=n$, in which case the cubic analogues of the Drinfeld-Serre relations go back to \cite{M}).
	
	\medskip 
	
	\subsection{Covering quivers}
	\label{sub:covering quivers}
	
	We will now consider a variant of the construction of folding quivers, which actually generalizes the setup of Subsection~\ref{sub:quivers automorphism} if the automorphism $\sigma$ acts freely on $I$. The term \emph{covering} will refer to a map of quivers 
	\[
	\fold : Q \rightarrow \oQ
	\]
	and we use the same notation for the induced function on vertex sets
	\begin{equation}
		\label{eqn:cover vertices}
		\fold : I \rightarrow \oI
	\end{equation}
	which we assume is surjective. We further assume that above every arrow $\oalpha$ of $\oQ$, there is a unique arrow of $Q$ ending at any preimage of the target of $\oalpha$, i.e. 
	\begin{equation}
		\label{eqn:assumption}
		\begin{split}
			&\forall \oalpha : \oi \rightarrow \oj \text{ and } j \in \fold^{-1}(\oj), \ \exists! i \in \fold^{-1}(\oi) \text{ and } \alpha : i \rightarrow j \text{ s.t. } \fold(\alpha) = \oalpha \\  
		\end{split}
	\end{equation}
	We fix the following data:
	
	\medskip 
	
	\begin{itemize}
		\item parameters $\{q_{\oalpha} \in \BK^*\}_{\oalpha \text{ arrow of }\oQ}$
		
		\medskip 
		
		\item a choice of distinguished preimages $I \supseteq J \xrightarrow{\sim} \oI$ for the function \eqref{eqn:cover vertices} 
		
		\medskip 
		
		\item a choice of shifts $\omega_i \in \BK^*$ for all $i \in I$ such that $\omega_i = 1$ if $i \in J$.
		
		\medskip 
		
	\end{itemize}
	
	\noindent Having made the choices above, we define
	\begin{equation}
		\label{eqn:matching parameters}
		q_{\alpha} = q_{\oalpha} \cdot \frac {\omega_j}{\omega_i}
	\end{equation} 
	for every arrow $\alpha : i\rightarrow j$ in $Q$ which maps onto an arrow $\oalpha : \oi \rightarrow \oj$ in $\oQ$. Then we define the \emph{upper zeta functions} as
	\begin{equation}
		\label{eqn:zeta above}
		\zeta_{ij}(x) \sim (1-x)^{-\delta_{ij}} \prod_{\alpha : i \rightarrow j} (1-xq_{\alpha})
	\end{equation} 
	and the \emph{lower zeta functions} as
	\begin{equation}
		\label{eqn:zeta below}
		\ozeta_{\oi \oj}(x) \sim (1-x)^{-\delta_{\oi \oj}} \prod_{\oalpha : \oi \rightarrow \oj} (1-xq_{\oalpha})
	\end{equation} 
	As before, the notation $\sim$ means that the two sides of the equation are equal up to a Laurent monomial, which we posit must be chosen so as to ensure the equation 
	\[
	\ozeta_{\oi \oj}(x) = \prod_{i' \in \fold^{-1}(\oi)} \zeta_{i'j_0} \left( \omega_{i'} x \right) 
	\]
	for all $\oi, \oj \in \oI$ (above, $j_0 \in J$ denotes the distinguished preimage of $\oj$). If we define 
	\[
	\ozeta_{ij}(x) = \prod_{i' \in \fold^{-1}(\fold(i))} \zeta_{i'j}  \left( \frac {\omega_{i'}}{\omega_i} x \right) 
	\]
	for all $i,j \in I$, then we have the relation
	\[
	\ozeta_{ij}(x) = \ozeta_{i'j}\left(\frac {\omega_{i'}}{\omega_i} x \right) 
	\]
	whenever $\fold(i)=\fold(i')$. This will allow us the freedom of indexing the lower zeta functions $\ozeta$ with the vertices in the upper quiver $Q$, which will be quite convenient. When discussing double quantum loop algebras as \eqref{eqn:quantum}, we will always assume that
	\begin{equation}
		\label{eqn:limit covering}
		\frac {\ozeta_{\oi \oj}(x)}{\ozeta_{\oj \oi}(x^{-1})} = \prod_{j \in \fold^{-1}(\oj)} \frac {\zeta_{ij} \left( \frac {x \omega_i}{\omega_j} \right)}{\zeta_{ji} \left( \frac {\omega_j}{x \omega_i} \right)}
	\end{equation}
	holds for all $\oi, \oj \in \oI$ and $i \in \fold^{-1}(\oi)$. The formula above implies that whenever the upper zeta functions are balanced in the sense of \eqref{eqn:limit}, so are the lower zeta functions. This can usually be ensured by appropriately fixing the Laurent monomials that control the discrepancy between the LHS and RHS of \eqref{eqn:zeta above}.
	
	\medskip 
	
	\begin{example}
		\label{ex:a3 covering}
		
		Let $Q$ be the quiver in Example~\ref{ex:a3}, but with the arrow parameters
		
		
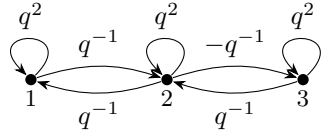
\[
		\begin{tikzpicture}[
			>=Stealth,
			vertex/.style={circle, fill=black, inner sep=1.5pt},
			lab/.style={below=3pt, font=\small}
			]
			
			\node[vertex, label={[lab]$1$}] (1) at (0,0) {};
			\node[vertex, label={[lab]$2$}] (2) at (1.8,0) {};
			\node[vertex, label={[lab]$3$}] (3) at (3.6,0) {};
			
			\draw[->, bend left=18] (1) to node[above] {$q^{-1}$} (2);
			\draw[->, bend left=18] (2) to node[below] {$q^{-1}$} (1);
			\draw[->, bend left=18] (2) to node[above] {$-q^{-1}$} (3);
			\draw[->, bend left=18] (3) to node[below] {$q^{-1}$} (2);
			\draw[->] (1) to[out=45, in=135, looseness=20] node[above] {$q^2$} (1);
			\draw[->] (2) to[out=45, in=135, looseness=20] node[above] {$q^2$} (2);
			\draw[->] (3) to[out=45, in=135, looseness=20] node[above] {$q^2$} (3);
		\end{tikzpicture}
		\]
		
		\noindent It covers the quiver $\oQ$ with parameters in the following picture
		
		
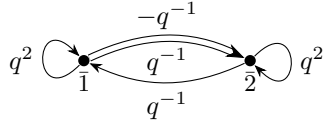
\[
		\begin{tikzpicture}[
			>=Stealth,
			vertex/.style={circle, fill=black, inner sep=1.5pt},
			nodelab/.style={below=3pt, font=\small}
			]
			
			\node[vertex, label={[nodelab]$\bar{1}$}] (1) at (0,0) {};
			\node[vertex, label={[nodelab]$\bar{2}$}] (2) at (2.2,0) {};
			
			\draw[->, double, double distance=2pt, bend left=25] (1) to node[above] {$-q^{-1}$} node[below] {$q^{-1}$} (2);
			\draw[->,bend left=25] (2) to node[below] {$q^{-1}$} (1);
			\draw[->] (1) to[out=225, in=135, looseness=20] node[left] {$q^2$} (1);
			\draw[->] (2) to[out=45, in=315, looseness=20] node[right] {$q^2$} (2);
			
		\end{tikzpicture}
		\]
		
		\noindent (the map $\rho$ sends $1,2,3$ to $\bar{1},\bar{2},\bar{1}$, and the shifts are $\omega_1 = \omega_2 = 1$, $\omega_3 = -1$). In this case, the upper zeta function is
		\[
		\zeta_{ij}(x) \sim \begin{cases} \frac {1-xq^2}{1-x} &\text{if }i=j \\
			1-xq^{-1} &\text{if } (i,j) \in \{(1,2),(2,1),(3,2)\}\\
			1+xq^{-1} &\text{if } (i,j) = (2,3)  \\ 1 &\text{if } (i,j) \in \{(1,3),(3,1)\} \end{cases}
		\]
		while the lower zeta function is
		\[
		\ozeta_{\oi \oj}(x) \sim \begin{cases} \frac {1-xq^2}{1-x} &\text{if }\oi=\oj  \\ 1 - x^{2}q^{-2} &\text{if } (\oi,\oj) = (\bar{1}, \bar{2} ) \\ 1 - xq^{-1} &\text{if } (\oi,\oj) = (\bar{2}, \bar{1}  )\end{cases}
		\]
		Compare the formulas above with the unfolded and folded zeta functions of Example~\ref{ex:a3}; the two situations are incompatible mainly because the respective $\ozeta$ functions involve different numbers of linear factors.
		
	\end{example}
	
	\medskip 
	
	\begin{remark}
		\label{rem:intersection}
		
		If $\sigma$ is an automorphism which acts freely on the vertex set of $Q$ (i.e. $m_i = 1, \forall i \in I$ in the notation \eqref{eqn:m i}), the construction of Subsection~\ref{sub:quivers automorphism} is a particular case of that of Subsection~\ref{sub:covering quivers}. To see this, note that if we set 
		\[
		\omega_{\sigma^k(i_0)} = \omega^k
		\]
		for all distinguished $i_0 \in J$ (in the RHS, $\omega$ denotes a primitive $m$-th root of unity), then the lower zeta functions of \eqref{eqn:zeta below} recover the folded zeta functions \eqref{eqn:folded zeta}. For instance, the quiver in Example~\ref{ex:a2} covers the following quiver:
		
		
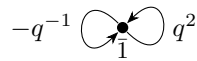
\[
		\begin{tikzpicture}[
			>=Stealth,
			vertex/.style={circle, fill=black, inner sep=1.5pt},
			lab/.style={below=3pt, font=\small}
			]
			
			\node[vertex, label={[lab]$\bar{1}$}] (1) at (0,0) {};
			
			\draw[->] (1) to[out=315, in=45, looseness=20] node[right] {$q^2$} (1);
			\draw[->] (1) to[out=135, in=225, looseness=20] node[left] {$-q^{-1}$} (1);
		\end{tikzpicture}
		\]
		
		\noindent The corresponding lower zeta function is $\ozeta_{\bar{1}\bar{1}}(x) \sim \frac{(1-xq^2)(1+xq^{-1})}{1-x}$, matching \eqref{eqn:example A2}.
		
	\end{remark}
	
	\medskip
	
	\subsection{Covering shuffle algebras}
	\label{sub:covering shuffle algebras}
	
	Just as the big and small shuffle algebras $\CV^\pm$ and $\CS^\pm$ were defined in Subsections~\ref{sub:def shuf}-\ref{sub:upsilon} with respect to the upper zeta functions of \eqref{eqn:zeta above}, we define the big and small shuffle algebras $\oCV^\pm$ and $\oCS^\pm$ with respect to the lower zeta functions of \eqref{eqn:zeta below}. In particular, for any $\bn \in \nn$ we may define
	\[
	\fold(\bn) = \obn \in \onn
	\]
	to be the vector of non-negative integers $\on_{\oi} = \sum_{i \in \fold^{-1}(\oi)} n_i$. Explicitly, we have
	\begin{equation}
		\label{eqn:v above}
		\CV^+ = \CV^- = \bigoplus_{\bn \in \nn} \BK \left[z_{i1},z_{i1}^{- 1},\dots,z_{in_i}, z_{in_i}^{-1} \right]_{i \in I}^{\fS_{\bn}}
	\end{equation}
	and 
	\begin{equation}
		\label{eqn:v below}
		\oCV^+ = \oCV^- = \bigoplus_{\obn \in \onn} \BK \left[z_{\oi 1},z_{\oi 1}^{-1},\dots,z_{\oi \on_{\oi}},z_{\oi \on_{\oi}}^{-1} \right]_{\oi \in \oI}^{\fS_{\obn}}
	\end{equation}
	Note that the invariants which appear in \eqref{eqn:v above} are taken with respect to a smaller permutation group than the invariants which appear in \eqref{eqn:v below}. We then obtain a linear map (which does not induce a shuffle algebra homomorphism, due to the difference between the functions $\zeta$ and $\ozeta$) for all $\fold(\bn) = \obn$:
	\begin{equation}
		\label{eqn:iota}
		\iota_{\obn,\bn} : \oCV_{\pm \obn} \hookrightarrow \CV_{\pm \bn}
	\end{equation}
	by mapping symmetric Laurent polynomials according to the rule
	\begin{equation}
		\label{eqn:variable shift}
		F(z_{\oi 1},\dots,z_{\oi n_{\oi}})_{\oi \in \oI} \mapsto F\left(\frac {z_{i 1}}{\omega_i} ,\dots, \frac {z_{i n_{i}}}{\omega_i}\right)_{i \in I}
	\end{equation}

	\medskip 
	
	\begin{proposition}
		\label{prop:covering shuffle}
		
		The map \eqref{eqn:iota} respects the small shuffle algebras, i.e. maps
		\[
		\oCS_{\pm \obn} \hookrightarrow \CS_{\pm \bn}
		\]
		for all $\rho(\bn) = \obn$. 
		
	\end{proposition}
	
	\medskip 
	
	\begin{proof} Without loss of generality, we will prove the case $\pm = -$. By definition, any element of $\CS_{-\bn}$ is a linear combination of 
		\begin{equation}
			\label{eqn:spherical}
			f_{i_1,d_1} * \cdots * f_{i_n,d_n} = \text{Sym}_{\bn}\left[z_{i_1\bullet_1}^{d_1} \cdots z_{i_n\bullet_n}^{d_n} \prod_{1\leq a < b \leq n} \zeta_{i_bi_a} \left(\frac {z_{i_b \bullet_b}}{z_{i_a \bullet_a}} \right) \right] 
		\end{equation}
		as $i_1,\dots,i_n \in I$, $d_1,\dots,d_n \in \BZ$ (in the formula above, $\bullet_1,\dots,\bullet_n$ denote the minimal positive integers with the property that $\bullet_a < \bullet_b$ for all $a < b$ such that $i_a = i_b$). Similarly, any element of $\oCS_{-\obn}$ is a linear combination of 
		\begin{equation}
			\label{eqn:spherical 2}
			f_{\oi_1,d_1} * \cdots * f_{\oi_n,d_n} = \text{Sym}_{\obn} \left[z_{\oi_1\bullet_1}^{d_1} \cdots z_{\oi_n\bullet_n}^{d_n} \prod_{1\leq a < b \leq n} \ozeta_{\oi_b\oi_a} \left(\frac {z_{\oi_b \bullet_b}}{z_{\oi_a \bullet_a}} \right) \right] 
		\end{equation}
		Because 
		\[
		\zeta_{ij}\left( x \right)  \text{ divides }\ozeta_{\oi\oj} \left(\frac {x \omega_j}{\omega_i}\right)
		\]
		and the symmetrization in \eqref{eqn:spherical 2} is defined with respect to a bigger group than in \eqref{eqn:spherical}, we conclude that any element of $\oCS_{-\obn}$ is also in $\CS_{-\bn}$, upon performing the variable shift \eqref{eqn:variable shift}. \end{proof}
	
	\medskip
	
	\subsection{A technical statement}
	\label{sub:technical}
	
	We will need an analogue of Proposition~\ref{prop:covering shuffle} in the setup of Subsections~\ref{sub:quivers automorphism} and \ref{sub:folding shuffle algebras}, whose proof we leave to the reader. In what follows, we denote by $i_0$ the distinguished preimage of any $\oi \in \oI$, and write $i_k = \sigma^k(i_0), \forall k$.
	
	\medskip 
	
	\begin{proposition}
		\label{prop:folding shuffle}
		
		In the generality of folding quivers with automorphisms (see Subsection~\ref{sub:folding shuffle algebras}), suppose $\rho(\bn) = \obn$ where $\on_{\oi} = n_{i_0}+\dots+n_{i_{\frac m{m_{\oi}}-1}}$. The map
		\begin{equation}
			\label{eqn:iota folding}
			\iota_{\obn,\bn} : \oCV_{\pm \obn} \hookrightarrow \CV_{\pm \bn}
		\end{equation}
		that sends a symmetric Laurent polynomial $F(z_{\oi 1},\dots,z_{\oi \on_{\oi}})_{\oi \in \oI}$ to
		\[
		F\left(z_{i_0 1},\dots,z_{i_0n_{i_0}}, \frac {z_{i_1 1}}{\omega},\dots, \frac {z_{i_1n_{i_1}}}{\omega}, \dots, \frac {z_{i_{\frac m{m_{\oi}}-1}1}}{\omega^{\frac m{m_{\oi}}-1}},\dots, \frac {z_{i_{\frac m{m_{\oi}}-1}n_{i_{\frac m{m_{\oi}}-1}}}}{\omega^{\frac m{m_{\oi}}-1}} \right)_{i_0 \in J}
		\]
		respects the small shuffle algebras, i.e. maps
		\[
		\oCS_{\pm \obn} \hookrightarrow \CS_{\pm \bn}
		\]
		
	\end{proposition}
	
	\medskip 
	
	\noindent By the very definition \eqref{eqn:v folded}, while elements of $\oCV^\pm$ are Laurent polynomials in $z_{\oi a}^{m_{\oi}}$, elements of $\CV^\pm$ are in general merely Laurent polynomials in $z_{ia}$. Thus, another way to think about the injection \eqref{eqn:iota folding} is to consider the finite group
	\begin{equation}
		\label{eqn:finite group}
		G_{\bm,\bn} = \prod_{i \in I} (\BZ/m_i \BZ)^{\times n_i}
	\end{equation}
	that acts on the variables $\{z_{i1},\dots,z_{in_i}\}_{i \in I}$ by independently multiplying them with $m_i$-th roots of unity. It is easy to see that
	\[
	\oCV_{\pm \obn} \cong \BK\left[z_{i1},z_{i1}^{-1},\dots,z_{in_i},z_{in_i}^{-1}\right]_{i \in I}^{G_{\bm,\bn}\rtimes \fS_{\obn}}
	\]
	thus realizing the map \eqref{eqn:iota folding} as the embedding of the set of Laurent polynomials which are invariant with respect to the bigger group $G_{\bm,\bn}\rtimes \fS_{\obn}$ into the set of Laurent polynomials which are invariant with respect to the smaller group $\fS_{\bn}$. 
	
	\bigskip
	
	\section{Simple modules}
	\label{sec:modules}
	
	\medskip
	
	\noindent We will now recall the simple modules $L(\bpsi)$ defined in \cite{HN, N Cat} for arbitrary quantum loop algebras, generalizing the classical constructions of \cite{CP, CP tw, H Shifted, HJ}. We provide the natural generalization of this construction to the setting of folded quivers with automorphisms and covering quivers, thus proving Theorems~\ref{thm:main folding intro} and \ref{thm:main covering intro}.
	
	\medskip 
	
	\begin{itemize}[leftmargin=*]
		
		\item In Subsections~\ref{sub:simple modules}-\ref{sub:q-characters}, we recall the general theory of simple modules, $q$-characters and their shuffle algebra realizations.
		
		\medskip 
		
		\item In Subsections~\ref{sub:simple folding}-\ref{sub:q-characters folding}, we adapt the above discussion to describe $q$-characters in the setting of folding quivers with automorphisms, and we state our main Theorem~\ref{thm:main folding intro} in Subsection~\ref{sub:folding q-characters} (labeled as Theorem~\ref{thm:main folding} below).
		
		\medskip 
		
		\item In Subsection~\ref{sub:covering q-characters}, we describe $q$-characters in the setting of coverings of quivers, and we state and prove Theorem~\ref{thm:main covering intro} (labeled as Theorem~\ref{thm:main covering} below).
		
		\medskip 
		
		\item In Subsection~\ref{sub:analogies}, we modify the proof of Theorem~\ref{thm:main covering intro} to obtain Theorem~\ref{thm:main folding intro}.
		
	\end{itemize}
	
	\medskip 
	
	\noindent To better navigate the inherent abuse of notation, we note that bar notations such as $\oI, \ozeta, \oCS, \oCV, \oUU$ refer to the setting of folding quivers with automorphisms in Subsections~\ref{sub:simple folding}-\ref{sub:folding q-characters} and \ref{sub:analogies}, and to the setting of coverings of quivers in Subsection~\ref{sub:covering q-characters}. Our approach of using the same notation for two different settings will be vindicated in Subsection~\ref{sub:analogies}, when we will adapt much of the proof of Theorem~\ref{thm:main covering} to obtain a proof of Theorem~\ref{thm:main folding} (this analogy would be much more difficult to follow if the two theorems used different notation). 
	
	\medskip 
	
	\subsection{Loop weights and simple modules}
	\label{sub:simple modules}
	
	Let us consider the data $(I,\BK,\zeta_{ij})$ as in Subsection~\ref{sub:zeta}, and assume that the zeta functions are balanced in the sense of \eqref{eqn:limit}. A \emph{loop weight} will refer to an $I$-tuple of power series
	\begin{equation}
		\label{eqn:loop weight}
		\bpsi = \left(\psi_i(z) = \sum_{d=0}^{\infty} \frac {\psi_{i,d}}{z^d} \in \BK[[z^{-1}]]^\times \right)_{i \in I}
	\end{equation}
	that we henceforth assume are expansions of completely factored rational functions:
	\begin{equation}
		\label{eqn:factored}
		\psi_i(z) = \gamma_i \frac {\left(1 - \frac {a_{i1}}z \right) \cdots \left(1 - \frac {a_{ip_i}}z \right)}{\left(1 - \frac {b_{i1}}z \right) \cdots \left(1 - \frac {b_{ir_i}}z \right)}
	\end{equation}
	for all $i \in I$, and various $a_{i1},\dots,a_{ip_i},b_{i1},\dots,b_{ir_i},\gamma_i \in \BK^*$. Let $\bord \bpsi = (p_i-r_i)_{i \in I} \in \zz$ encode the orders of the poles of the above rational functions at $z = 0$.
	
	\medskip 
	
	\begin{definition}
		\label{def:shifted}
		
		Given $\emph{\textbf{r}} = (r_i)_{i \in I} \in \zz$, the shifted quantum loop algebra is
		\[
		\UU^{\emph{\textbf{r}}} = \BK \Big \langle e_{i,d}, f_{i,d}, \ph^\pm_{i,d'} \Big \rangle_{i \in I, d \in \BZ, d' \geq 0} \Big/ \Big( \text{\eqref{eqn:rel quantum 1}-\eqref{eqn:rel quantum 7} and \eqref{eqn:rel quantum 9 shifted}} \Big)
		\]
		in which we replace relation \eqref{eqn:rel quantum 9} by the following shifted version:
		\begin{equation}
			\label{eqn:rel quantum 9 shifted}
			\Big[e_{i}(x), f_{j}(y)\Big] = \delta_{ij} \delta \left(\frac xy \right) \Big(y^{-r_j} \ph^-_j(y) - \ph^+_i(x)\Big)
		\end{equation} 
		While we still assume that $\ph_{i,0}^+$ and $\ph_{i,0}^-$ are invertible in $\UU^{\emph{\textbf{r}}}$, we remove the assumption that their product is equal to 1 from \eqref{eqn:rel quantum 4} (see \cite[Section~3.1]{H Shifted}).
		
	\end{definition}
	
	\medskip 
	
	\noindent With the zeta functions associated to a finite type Lie algebra (see Subsection~\ref{sub:Kac-Moody}), the construction above is known as a shifted quantum affine algebra.
	
	\medskip
	
	\begin{theorem}
		\label{thm:simple}
		
		(\cite{HN, N Cat}) For any loop weight $\bpsi$, there exists a simple module
		\begin{equation}
			\label{eqn:simple shift}
			\UU^{\emph{\bord} \bpsi} \curvearrowright L(\bpsi) 
		\end{equation}
		generated by a single vector $\vac$ modulo the following relations
		\begin{align*}
			&\ph_i^+(z) \cdot \vac = \psi_i(z) \vac \qquad \text{ expanded near } z \sim \infty \\
			&\ph_i^-(z) \cdot \vac = \psi_i(z) \vac \qquad \text{ expanded near } z \sim 0 \\
			&e_{i}(z) \cdot \vac = 0
		\end{align*}
		for all $i \in I$. As the shift of $\UU$ is dictated by $\bpsi$, we henceforth abbreviate \eqref{eqn:simple shift} by 
		\[
		\UU^{\esh} \curvearrowright L(\bpsi)
		\]
		Note that when each constituent rational function $\psi_i(z)$ has finite limits at 0 and $\infty$ which multiply to 1, a $\UU^{\esh}$-module is the same as a $\UU$-module. 
		
	\end{theorem}

	\medskip 
	
	\noindent Explicitly, the simple module $L(\bpsi)$ is constructed via the following formula, which is analogous to the classic construction of irreducible representations of simple Lie algebras as quotients of Verma modules by the kernel of Shapovalov forms:
	\[
	L(\bpsi) = \bigoplus_{\bn \in \nn} L(\bpsi)_{\bn}, \qquad L(\bpsi)_{\bn} = \CS_{-\bn} \Big / J(\bpsi)_{\bn}
	\]
	where 
	\begin{equation}
		\label{eqn:j n}
		\begin{split}
			J(\bpsi)_{\bn} = \Big \{F \in &\CS_{-\bn} \Big|  \forall \bn = \bs^{i_1} + \cdots + \bs^{i_n}, \forall d_1,\dots,d_n \in \BZ, \\
			&\underset{z_n \in \BK^*}{\text{Res}} \cdots \underset{z_1 \in \BK^*}{\text{Res}} \frac {F (z_1,\dots,z_n)z_1^{d_1}\cdots z_n^{d_n}}{\prod_{1\leq a < b \leq n} \zeta_{i_bi_a} \left(\frac {z_b}{z_a} \right)} \prod_{a=1}^n \psi_{i_a}(z_a) = 0  \Big\}
		\end{split}
	\end{equation}
	The notation $\underset{z \in \BK^*}{\text{Res}}$ refers to the sum of residues of a rational function in $z$ over all non-zero field elements, and $F(z_1,\dots,z_n)$ is defined via the substitution \eqref{eqn:relabeling}. 
	
	\medskip 
	
	\subsection{\texorpdfstring{$q$}{q}-characters} 
	\label{sub:q-characters}
	
	As shown in \cite{HN, N Cat}, the generalized eigenspace decomposition of $L(\bpsi)$ with respect to the commuting operators $\ph_{i,0}^+,\ph_{i,1}^+,\ph_{i,2}^+,\dots$ is given by
	\begin{equation}
		\label{eqn:l x}
		L(\bpsi)_{\bn} = \bigoplus_{\bx \in (\BK^*)^{\bn}} L(\bpsi)_{\bx}, \qquad L(\bpsi)_{\bx} = \CS_{-\bn} \Big / J(\bpsi)_{\bx}
	\end{equation}
	for any $\bx = (x_{i1},\dots,x_{in_i})_{i \in I} \in (\BK^*)^{\bn} = \prod_{i \in I} (\BK^*)^{n_i}/\fS_{n_i}$. Above, we set 
	\begin{equation}
		\label{eqn:j x}
		\begin{split}
			J(\bpsi)_{\bx} =& \Big \{F \in \CS_{-\bn} \Big|  \forall x_1,\dots,x_n \text{ ordering of }\bx, \forall \bn = \bs^{i_1} + \cdots + \bs^{i_n},\\
			&\underset{z_n = x_n}{\text{Res}} \cdots \underset{z_1 = x_1}{\text{Res}} \frac {F (z_1,\dots,z_n)(\text{any monomial})}{\prod_{1\leq a < b \leq n} \zeta_{i_bi_a} \left(\frac {z_b}{z_a} \right)} \prod_{a=1}^n \psi_{i_a}(z_a) = 0  \Big\}
		\end{split}
	\end{equation}
	where a sequence $(x_1,\dots,x_n) \subset \BK^*$ is said to be an ordering of $\bx$ if we have an equality of multisets $\{x_a| i_a = i\} = \{x_{i1},\dots,x_{in_i}\}$ for all $i$ (with respect to a given, but arbitrary, sequence $i_1,\dots,i_n \in I$ such that $\bn = \bs^{i_1}+\cdots+\bs^{i_n}$). 
	
	\medskip 
	
	\begin{lemma}
		\label{lem:residues}
		
		For any loop weight $\bpsi$ as in \eqref{eqn:factored} and any $\bx$ as above, we have
		\[
		\dim_{\BK} L(\bpsi)_{\bx} < \infty
		\]
		Moreover, for any $\bn \in \nn$ only finitely many $\bx \in (\BK^*)^{\bn}$ have the property that 
		\[
		L(\bpsi)_{\bx} \neq 0
		\]
		and so $L(\bpsi)_{\bn}$ is also finite-dimensional.
		
	\end{lemma}
	
	\medskip
	
	\begin{proof} As a module over the ring $\CV_{-\bn} = \BK[z_{i1},z_{i1}^{- 1},\dots,z_{in_i}, z_{in_i}^{- 1}]^{\fS_{\bn}}_{i \in I}$, the quotient $\CV_{-\bn} / J(\bpsi)_{\bx}$ is supported at the point
		\[
		\bx \in  (\BK^*)^{\bn} = \text{Spec}(\CV_{-\bn})
		\]
		The aforementioned quotient is a finite-dimensional module, because if $F$ is divisible by $(z_a-x_a)^N$ for large enough $N$ (compared with the number of factors in the numerators of $\zeta$ and denominators of $\psi$), then the residue that appears in \eqref{eqn:j x} automatically vanishes. For the second claim of Lemma~\ref{lem:residues}, it suffices to note that we can have a non-trivial residue in \eqref{eqn:j x} only if for all $b \in \{1,\dots,n\}$ either
		
		\medskip 
		
		\begin{itemize}
			
			\item $x_b$ is one of the finitely many poles of $\psi_{i_b}(z)$, or 
			
			\medskip 
			
			\item $x_b = x_a q^{-1}_{\alpha}$ for some $a < b$ and some arrow $\alpha : i_b \rightarrow i_a$ of the quiver
			
		\end{itemize}
		
		\medskip 
		
		\noindent Since there are finitely many choices for $x_b$ for all $b \in \{1,\dots,n\}$, we conclude that $L(\bpsi)_{\bx} \neq 0$ only for finitely many $\bx \in (\BK^*)^{\bn}$. \end{proof}
	
	\medskip 
	
	\noindent Recall that the $q$-character of any module $\UU^{\sh} \curvearrowright V$ is defined as (\cite{FR, FM, H Shifted})
	\begin{equation}
		\label{eqn:q char}
		\chi_q(V) = \sum_{\bpsi} [\bpsi] \dim_{\BK}(V_{\bpsi})
	\end{equation}
	where $V_{\bpsi}$ denotes the generalized eigenspace of $V$ on which $\ph^+_i(z)$ acts by $\psi_i(z)$, for all $i \in I$. The notion \eqref{eqn:q char} makes sense in those situations when the generalized eigenspaces are finite-dimensional, and Lemma~\ref{lem:residues} implies that this is the case for $V = L(\bpsi)$. Moreover, we have the following explicit description of the $q$-character:
	\[
	\chi_q(L(\bpsi)) = [\bpsi] \sum_{\bn \in \nn} \sum_{\bx \in (\BK^*)^{\bn}} \dim_{\BK} \left( L(\bpsi)_{\bx} \right) \prod_{i \in I} A_{i,x_{i1}}^{-1} \cdots A_{i,x_{in_i}}^{-1}
	\]
	where for all $i \in I$ and $x \in \BK^*$ we define the loop weight
	\begin{equation}
		\label{eqn:fm}
		A_{i,x}^{-1} = \left[\frac {\zeta_{ij} \left(\frac xz \right)}{\zeta_{ji} \left(\frac zx \right)} \right]_{j \in I}
	\end{equation}
	and we multiply the formal symbols $[\bpsi]$ for loop weights $\bpsi$ component-wise in $I$. 
	
	\medskip
	
	\subsection{Simple modules and folding}
	\label{sub:simple folding}
	
	Let us now consider the setting of Subsection~\ref{sub:quivers automorphism} with the notation therein. We consider the folded zeta functions \eqref{eqn:folded zeta}, and assume that \eqref{eqn:limit folded} holds. By applying the construction in Subsections~\ref{sub:zeta}-\ref{sub:double} to the folded zeta functions, we may define the \emph{folded quantum loop algebra}
	\[
	\oUU = \oUUp \otimes \BK[\ph_{\oi,d}^\pm]_{\oi \in I, d \in m_{\oi}\BN} / (\ph^+_{\oi,0}\ph^-_{\oi,0}-1) \otimes \oUUm
	\]
	By definition, $\oUUpm$ are isomorphic to the images of the homomorphisms
	\[
	\bar{\tUpsilon}^\pm : \bar{\tUU}^\pm \longrightarrow \oCV^\pm, \qquad e_{\oi,d}, f_{\oi,d} \mapsto z_{\oi 1}^d, \quad \forall d \in m_{\oi}\BZ
	\]
	where $\oCV^\pm$ are defined in \eqref{eqn:v folded}. Explicitly, $\oCS^- = \text{Im }\bar{\tUpsilon}^-$ is the linear span of
	\[
	f_{\oi_1,d_1} * \cdots * f_{\oi_n,d_n} = \text{Sym}_{\obn} \left[z_{\oi_1\bullet_1}^{d_1} \cdots z_{\oi_n\bullet_n}^{d_n} \prod_{1\leq a < b \leq n} \ozeta_{\oi_b\oi_a} \left(\frac {z_{\oi_b \bullet_b}}{z_{\oi_a \bullet_a}} \right) \right]
	\]
	for various $\oi_1,\dots,\oi_n \in \oI$, $d_1 \in m_{\oi_1}\BZ,\dots,d_n \in m_{\oi_n}\BZ$, and similarly for $\oCS^+$. 
	
	\medskip 
	
	\begin{definition}
		
		For any $\bar{\mathbf{r}} = (\bar{r}_{\oi})_{\oi \in \oI} \in \ozz$, we define the \emph{shifted folded quantum loop algebra} $\oUU^{\bar{\mathbf{r}}}$ by replacing the $e$-$f$ commutation relation \eqref{eqn:rel quantum 9 twisted strongly} by its shifted version
		\begin{equation}
			\label{eqn:shifted folded quantum loop algebra}
			\Big[e_{i}(x), f_{j}(y)\Big] = \left( \sum_{k=1}^m \frac {\delta_{\sigma^k(i)j}\delta(\frac{x}{y}\omega^k)}{m_i} \right) \Big( y^{-\bar{r}_{\oj}m_{\oj}}\ph^-_j(y) - \ph^+_i(x)\Big)
		\end{equation}
		(in particular, for $\fg$ of finite type, relation \eqref{eqn:shifted folded quantum loop algebra} coincides with that of \cite{LLL}). 
		
	\end{definition}
	
	\medskip 
	
	\noindent We call
	\begin{equation}
		\label{eqn:folded loop weight}
		\obpsi = \left(\opsi_{\oi}(z) = \sum_{d \in m_{\oi}\BN} \frac {\opsi_{\oi,d}}{z^d} \in \BK[[z^{-m_{\oi}}]]^\times \right)_{\oi \in \oI}
	\end{equation}
	a \emph{folded loop weight}, and we assume that it consists of expansions of completely factored rational functions of the form
	\[
	\opsi_{\oi}(z) = \gamma_{\oi} \frac {\left(1 - \left( \frac {a_{\oi 1}}{z} \right)^{m_{\oi}}\right)  \cdots \left(1 - \left( \frac {a_{\oi p_{\oi}}}{z} \right)^{m_{\oi}}\right)}{\left(1 - \left( \frac {b_{\oi 1}}{z} \right)^{m_{\oi}}\right)  \cdots \left(1 - \left( \frac {b_{\oi r_{\oi}}}{z} \right)^{m_{\oi}}\right)}
	\]
	for all $\oi \in \oI$, and various $a_{\oi 1},\dots,a_{\oi p_{\oi}},b_{\oi 1},\dots,b_{\oi r_{\oi}},\gamma_{\oi} \in \BK^*$. By analogy with Theorem~\ref{thm:simple}, we claim that for any folded loop weight $\obpsi$ there exists a simple module
	\begin{equation}
		\label{eqn:shifted simple module}
		\oUU^{\sh} \curvearrowright L(\obpsi) 
	\end{equation}
	explicitly given by
	\[
	L(\obpsi) = \bigoplus_{\obn \in \onn} L(\obpsi)_{\obn}, \qquad L(\obpsi)_{\obn} = \oCS_{-\obn} \Big/ J(\obpsi)_{\obn}
	\]
	where 
	\begin{equation}
		\label{eqn:j n folded}
		\begin{split}
			J(\obpsi)_{\obn} = \Big \{F \in &\CS_{-\obn} \Big|  \forall \obn = \bs^{\oi_1} + \cdots + \bs^{\oi_n}, \forall d_1 \in m_{\oi_1}\BZ,\dots,d_n \in m_{\oi_n}\BZ \\
			&\underset{z_n \in \BK^*}{\text{Res}} \cdots \underset{z_1 \in \BK^*}{\text{Res}} \frac {F (z_1,\dots,z_n)z_1^{d_1}\cdots z_n^{d_n}}{\prod_{1\leq a < b \leq n} \ozeta_{\oi_b\oi_a} \left(\frac {z_b}{z_a} \right)} \prod_{a=1}^n \opsi_{\oi_a}(z_a) = 0  \Big\}
		\end{split}
	\end{equation}
	Since the integrand in \eqref{eqn:j n folded} is a rational function in $z_a^{m_{i_a}}$, we note that the vanishing of the residue at any point $\{z_a = x_{a}\}_{a \in \{1,\dots,n\}}$ implies the vanishing of the analogous residues when the $x_a$'s are shifted independently by $m_{i_a}$-th roots of unity. Because of this, nothing changes if we replace the condition $d_a \in m_{\oi_a} \BZ$ in \eqref{eqn:j n folded} by the a priori stronger condition $d_a \in \BZ$. 
	
	\medskip 
	
	\begin{example} 
		\label{ex:trivial sigma}
		
		Assume $\sigma = \emph{Id}$, but $m > 1$ (while $m$ was defined as the order of $\sigma$ in Subsection~\ref{sub:quivers automorphism}, our construction only used the fact that $\sigma^m = \emph{Id}$). In this case, if the unfolded zeta functions are
		\[
		\zeta_{ij}(x) \sim (1-x)^{-1} \prod_{\alpha : i \rightarrow j} (1-xq_\alpha)
		\]
		then the folded zeta functions are
		\[
		\ozeta_{ij}(x) \sim (1-x^m)^{-1} \prod_{\alpha : i \rightarrow j} (1-x^m q^m_\alpha)
		\]
		for all $i,j \in \oI = I$. Moreover, to any unfolded loop weight 
		\[
		\bpsi = \left( \frac {\prod_s \left( 1-\frac {a_s}z \right)}{\prod_t \left( 1-\frac {b_t}z \right)} \right)_{i \in I}
		\]
		we may associate in obvious fashion the folded loop weight
		\[
		\obpsi = \left( \frac {\prod_s \left( 1- \left( \frac {a_s}z \right)^m \right)}{\prod_t \left( 1- \left(\frac {b_t}z\right)^m \right)} \right)_{i \in I} 
		\]
		Thus, if we write $y_{ia} = z_{ia}^m$, it is clear that the folded shuffle algebra $\oCS^-$ in the $z$ variables is identical to the unfolded shuffle algebra $\CS^-$ in the $y$ variables. Moreover, condition \eqref{eqn:j n} is the same as \eqref{eqn:j n folded} upon the change of variables $z_{ia}^m = y_{ia}$, because if a rational function in $z_{ia}^m$ has no residue at some point then it also has no residue at shifts of said point by arbitrary $m$-th roots of unity. We conclude that $L(\bpsi)$ is naturally identified with $L(\obpsi)$, as expected in the $\sigma = \emph{Id}$ case.
		
	\end{example}
	
	\medskip 
	
	\subsection{\texorpdfstring{$q$}{q}-characters and folding}
	\label{sub:q-characters folding}
	
	Define $q$-characters of modules $\oUU^{\sh} \curvearrowright V$ as
	\[
	\chi_q(V) = \sum_{\obpsi} [\obpsi] \dim_{\BK} \left(V_{\obpsi}\right) 
	\]
	where $V_{\obpsi}$ denotes the generalized eigenspace of $V$ on which 
	\[
	\ph^+_{\oi}(z) = \sum_{d \in m_{\oi}\BN} \frac {\ph_{\oi,d}^+}{z^d} \quad \text{acts by} \quad \opsi_{\oi}(z)
	\]
	(of course, we assume that all the aforementioned generalized eigenspaces are finite-dimensional). The explicit description of simple $\oUU^{\sh}$ modules from Subsection~\ref{sub:simple folding} allows us to give a shuffle algebra interpretation of the $q$-character of $L(\obpsi)$. Specifically, by analogy with \eqref{eqn:l x}, the generalized eigenspace decomposition is
	\[
	L(\obpsi)_{\obn} = \bigoplus_{\obx \in (\BK^*)^{\obn}/G_{\bm,\bn}} L(\obpsi)_{\obx}, \qquad L(\obpsi)_{\obx} = \CS_{-\obn} \Big / J(\obpsi)_{\obx}
	\]
	where $G_{\bm,\bn}$ is the finite abelian group \eqref{eqn:finite group} that rescales the variables $z_{\oi a}$ by $m_{\oi}$-th roots of unity. Thus, for any $\obx = (\ox_{\oi 1},\dots,\ox_{\oi \on_{\oi}})_{\oi \in \oI} \in (\BK^*)^{\obn}/G_{\bm,\bn}$, we define 
	\begin{equation}
		\label{eqn:j x folded}
		\begin{split}
			J(\obpsi)_{\obx} =& \Big\{ F \in \CS_{-\obn} \Big|\forall \ox_1,\dots,\ox_n \text{ ordering of }\obx, \obn = \bs^{\oi_1} + \cdots + \bs^{\oi_n}, \\
			&\underset{z_n = x_n}{\text{Res}} \cdots \underset{z_1 = x_1}{\text{Res}} \frac {F (z_1,\dots,z_n)(\text{any monomial})}{\prod_{1\leq a < b \leq n} \ozeta_{\oi_b\oi_a} \left(\frac {z_b}{z_a} \right)} \prod_{a=1}^n \opsi_{\oi_a}(z_a) = 0  \Big\}
		\end{split}
	\end{equation}
	As explained right before Example~\ref{ex:trivial sigma}, the reason why condition \eqref{eqn:j x folded} is invariant under the $G_{\bm,\bn}$ action on $\obx$ is the fact that the integrand is a function of $z_a^{m_{\oi_a}}$. The natural folded analogue of Lemma~\ref{lem:residues} holds. Moreover, we have that
	\[
	\chi_q(L(\obpsi)) = [\obpsi] \sum_{\obn \in \onn} \sum_{\obx \in (\BK^*)^{\obn}/G_{\bm,\bn}} \dim_{\BK} \left( L(\obpsi)_{\obx} \right) \prod_{\oi \in \oI} A_{\oi,\ox_{\oi 1}}^{-1} \cdots A_{\oi,\ox_{\oi n_{\oi}}}^{-1}
	\]
	where for all $\oi \in \oI$ and $\ox \in \BK^*/(\BZ/m_{\oi}\BZ)$ we define the folded loop weight
	\begin{equation}
		\label{eqn:fm fold}
		A_{\oi,\ox}^{-1} = \left[\frac {\ozeta_{\oi \oj} \left(\frac {\ox}z \right)}{\ozeta_{\oj \oi} \left(\frac z{\ox} \right)} \right]_{\oj \in \oI}
	\end{equation}
	and we multiply the formal symbols $[\obpsi]$ for loop weights $\obpsi$ component-wise in $\oI$.
	
	\medskip 
	
	\subsection{The main theorem}
	\label{sub:folding q-characters}
	
	Assuming \eqref{eqn:limit folded}, we will compare the simple modules
	\[
	\UU^{\sh} \curvearrowright L(\bpsi) \quad \text{and} \quad \oUU^{\sh} \curvearrowright L(\obpsi)
	\]
	associated to unfolded loop weights $\bpsi$ of \eqref{eqn:loop weight} and folded loop weights $\obpsi$ of \eqref{eqn:folded loop weight}. We assume that these loop weights are related to each other by the folding map:
	\[
	\fold(\bpsi) = \obpsi
	\]
	where the $\oi$-th component of $\obpsi$ is given in terms of the components of $\bpsi$ by
	\[
	\opsi_{\oi}(z) = \prod_{k = 1}^m \psi_{\sigma^k(i_0)} \left(z \omega^k\right)
	\]
	where $i_0 \in I$ is the distinguished preimage of $\oi \in \oI$. Note that $\opsi_{\oi}(z)$ is actually a function of $z^{m_{\oi}}$, and that the map $\fold$ is multiplicative with respect to the component-wise multiplication of loop weights.  Due to property \eqref{eqn:limit folded}, the unfolded loop weights \eqref{eqn:fm} and the folded loop weights \eqref{eqn:fm fold} are related by the equation 
	\[
	\fold(A_{\sigma^{\ell}(i_0),x\omega^{\ell}}^{-1}) = A_{\oi,x}^{-1}
	\]
	for all distinguished $J \ni i_0 \mapsto \oi \in \oI$ and all $\ell \in \BZ/m\BZ$. Assume we are given subsets
	\begin{equation}
		\label{eqn:partition folding}
		P_i \subseteq \BK^*
	\end{equation}
	for all $i \in I$, such that 
	\begin{equation}
		\label{eqn:condition folding}
		P_j  \subseteq P_i q_{\alpha}
	\end{equation}
	for all $\alpha : i \rightarrow j$, and
	\begin{equation}
		\label{eqn:disjoint folding}
		P_i \cap  P_{\sigma^k(i)} \omega^{-k} = \varnothing
	\end{equation}
	for all $i \in I$ and $k \in \{1,\dots,m-1\}$.
	
	\medskip
	
	\begin{definition}
		\label{def:acceptable folding}
		
		Given sets \eqref{eqn:partition folding}, we call an unfolded loop weight $\bpsi$ \emph{acceptable} if
		\begin{equation}
			\label{eqn:acceptable folding}
			b_{i1},\dots,b_{ir_i} \in P_i 
		\end{equation}
		for all $i \in I$, with the notation as in \eqref{eqn:factored}.
		
	\end{definition}
	
	\medskip
	
	\noindent The following result is our main theorem; it is inspired by the conjectures on folding $q$-characters of twisted quantum affine algebras, of \cite{H KR}, \cite{W QQ}, \cite{FQ}. 
	
	\medskip 
	
	\begin{theorem}
		\label{thm:main folding} 
		
		Consider the shifted quantum loop algebra modules $\UU^{\esh} \curvearrowright L(\bpsi)$ and $\oUU^{\esh} \curvearrowright L(\rho(\bpsi))$ defined in Subsections~\ref{sub:simple modules} and \ref{sub:simple folding} with respect to the zeta functions \eqref{eqn:unfolded zeta} and \eqref{eqn:folded zeta}, respectively (we assume that \eqref{eqn:limit folded} holds). If $\bpsi$ is acceptable in the sense of Definition \ref{def:acceptable folding}, then we have
		\[
		\fold(\chi_q(L(\bpsi))) = \chi_q(L(\fold(\bpsi)))
		\]
		
	\end{theorem}
	
	\medskip 
	
	\begin{corollary}
		\label{cor:main}
		
		For twisted quantum loop algebras of any Kac-Moody type (as in Subsection~\ref{sub:twisted quantum loop}), the assumptions above hold for
		\[
		P_{i} = q^{\BZ}, \ \forall i \in I
		\]
		More generally, the $P_{i}$ could be defined to be any subset of $\BC^*$ which is invariant with respect to multiplication by $q$, and is disjoint from all its shifts by non-trivial $m$-th roots of unity. In this case, Theorem~\ref{thm:main folding} yields Theorem~\ref{thm:main}.
		
	\end{corollary}
	
	\medskip
	
	\subsection{Covering quivers and \texorpdfstring{$q$}{q}-characters}
	\label{sub:covering q-characters}
	
	Before we prove Theorem~\ref{thm:main folding}, we will find it more convenient to state and prove its analogue in the setting of Subsection~\ref{sub:covering quivers}. With the notation therein, consider upper zeta functions \eqref{eqn:zeta above} and lower zeta functions \eqref{eqn:zeta below}, and assume that they are both balanced and that \eqref{eqn:limit covering} holds. This allows us to define the upper and lower double quantum loop algebras, and denote them by $\UU$ and $\oUU$, respectively. We will call the loop weights
	\[
	\bpsi = (\psi_i(z))_{i\in I} \quad \text{and} \quad \obpsi = (\opsi_{\oi}(z))_{\oi \in \oI}
	\]
	upper and lower, respectively. To such loop weights, we associate simple modules
	\[
	\UU^{\sh} \curvearrowright L(\bpsi) \quad \text{and} \quad \oUU^{\sh} \curvearrowright L(\obpsi)
	\]
	as in Subsection~\ref{sub:simple modules} (we do not need to invoke the treatment of Subsection~\ref{sub:simple folding} because, as opposed from the setting of folding quivers with automorphisms, the setting of coverings of quivers does not require raising any variables of shuffle algebras to $m_i$-th powers), and our goal now is to compare these modules. There is a natural folding map from upper to lower loop weights, given by 
	\[
	\fold(\bpsi) = \obpsi
	\]
	where the $\oi$-th component of $\obpsi$ is given in terms of the components of $\bpsi$ by
	\begin{equation}
		\label{eqn:loop weight covering}
		\opsi_{\oi}(z) = \prod_{i \in \fold^{-1}(\oi)} \psi_i\left(\omega_i  z\right)
	\end{equation}
	It is clear that the map $\fold$ is multiplicative with respect to the component-wise multiplication of loop weights. Due to property \eqref{eqn:limit covering}, the upper and lower versions of the loop weights \eqref{eqn:fm} are related by the equation 
	\begin{equation}
		\label{eqn:compatible covering}
		\fold(A_{i,x}^{-1}) = A_{\oi,\frac x{\omega_i}}^{-1}
	\end{equation}
	for all $i \in I$ and $x \in \BK^*$. We assume that there exist subsets
	\begin{equation}
		\label{eqn:partition covering}
		P_{i} \subseteq \BK^*
	\end{equation}
	for all $i \in I$, such that
	\begin{equation}
		\label{eqn:condition covering}
		P_{j}  \subseteq P_{i} q_{\alpha}
	\end{equation}
	for all $\alpha : i \rightarrow j$, and 
	\begin{equation}
		\label{eqn:disjoint covering}
		\frac {P_i}{\omega_i} \cap \frac {P_{i'}}{\omega_{i'}} = \varnothing
	\end{equation}
	for all $i \neq i'$ with $\fold(i) = \fold(i')$.
	
	\medskip 
	
	\begin{definition}
		\label{def:acceptable covering}
		
		Given sets \eqref{eqn:partition covering}, we call an upper loop weight $\bpsi$ \emph{acceptable} if
		\begin{equation}
			\label{eqn:acceptable covering}
			b_{i1},\dots,b_{ir_i} \in P_i
		\end{equation}
		for all $i \in I$, with the notation as in \eqref{eqn:factored}.
		
	\end{definition}
	
	\medskip
	
	\begin{theorem}
		\label{thm:main covering}
		
		Consider the shifted quantum loop algebra modules $\UU^{\esh} \curvearrowright L(\bpsi)$ and $\oUU^{\esh} \curvearrowright L(\rho(\bpsi))$ defined in Subsection~\ref{sub:simple modules} with respect to the zeta functions \eqref{eqn:zeta above} and \eqref{eqn:zeta below}, respectively (we assume that \eqref{eqn:limit covering} holds). If $\bpsi$ is acceptable in the sense of Definition \ref{def:acceptable covering}, then we have
		\begin{equation}
			\label{eqn:main covering}
			\fold(\chi_q(L(\bpsi))) = \chi_q(L(\fold(\bpsi)))
		\end{equation}
		
	\end{theorem}
	
	\medskip 
	
	\begin{proof} We henceforth assume that $\BK$ is algebraically closed (for instance, by extending all scalars in the formulas below to the algebraic closure of $\BK$). Consider any $\bn = (n_i \geq 0)_{i \in I} \in \nn$ and $\fold(\bn) = \obn = (\on_{\oi} \geq 0)_{\oi \in \oI} \in \onn$ and recall the natural inclusion of $\BK$-algebras \eqref{eqn:iota}
		\[
		\iota_{\obn,\bn}: \oCV_{-\obn} = \BK\left[z_{\oi 1}, z_{\oi 1}^{-1},\dots,z_{\oi \on_{\oi}}, z_{\oi \on_{\oi}}^{-1}\right]^{\fS_{\obn}}_{\oi \in \oI} \hookrightarrow \CV_{-\bn} = \BK\left[z_{i 1}, z_{i 1}^{- 1},\dots,z_{i n_{i}}, z_{i n_{i}}^{- 1}\right]^{\fS_{\bn}}_{i \in I}
		\]
		Note that the polynomial rings in the two sides of the equation are identified by \eqref{eqn:variable shift}, but the LHS features invariants with respect to a bigger group than the RHS. Recall that the LHS of \eqref{eqn:main covering} counts the dimension of the fibers of the $\CV_{-\bn}$ module
		\[
		\CS_{-\bn} \Big/ J(\bpsi)_{\bn}
		\]
		over various $\bx \in (\BK^*)^{\bn} = \text{Spec}(\CV_{-\bn})$. Similarly, the RHS of \eqref{eqn:main covering} counts the dimension of the fibers of the $\oCV_{-\obn}$ module
		\[
		\oCS_{-\obn} \Big/ J(\obpsi)_{\obn}
		\]
		over various $\obx \in (\BK^*)^{\obn} = \text{Spec}(\CV_{-\obn})$. Specifically, the aforementioned fibers are
		\[
		\CS_{-\bn} \Big/ J(\bpsi)_{\bx} \quad \text{and} \quad \oCS_{-\obn} \Big/ J(\obpsi)_{\obx}
		\]
		respectively. We remind the reader that $J(\bpsi)_{\bx}$ is the $\CV_{-\bn}$ submodule of $\CS_{-\bn}$ consisting of Laurent polynomials $F(z_{i1},\dots,z_{in_i})_{i \in I}$ such that
		\[
		\underset{z_n = x_n}{\text{Res}} \cdots \underset{z_1 = x_1}{\text{Res}} \frac {F (z_1,\dots,z_n)(\text{any monomial})}{\prod_{1\leq a < b \leq n} \zeta_{i_bi_a} \left(\frac {z_b}{z_a} \right)} \prod_{a=1}^n \psi_{i_a}(z_a) = 0  
		\]
		for any ordering $x_1,\dots,x_n$ of $\bx$ compatible with any ordering $i_1,\dots,i_n$ of $\bn$, and similarly for $J(\obpsi)_{\obx}$. With this in mind, formula \eqref{eqn:main covering} follows via the compatibility \eqref{eqn:compatible covering} from the equality of the dimensions of the two sides in equation \eqref{eqn:claim iso} below. 
		
		\medskip 
		
		\begin{claim} 
			\label{claim:iso}
			
			For any $\obx \in (\BK^*)^{\obn}$, the natural map
			\begin{equation}
				\label{eqn:claim iso}
				\oCS_{-\obn} \Big/ J(\obpsi)_{\obx} \longrightarrow \mathop{\bigoplus_{\bn \in \fold^{-1}(\obn)}}_{\bx \in (\BK^*)^{\bn} \cap \fold^{-1}(\obx)} \CS_{-\bn} \Big/ J(\bpsi)_{\bx}
			\end{equation}
			induced by $\iota_{\obn,\bn} : \oCV_{-\obn} \hookrightarrow \CV_{-\bn}$ is an isomorphism. In the formula above, $\fold(\bn) \in \onn$ is given by summing over the fibers of $\fold : I \rightarrow \oI$, while 
			\[
			\fold(\text{tuple }\bx \text{ with entries }x \text{ on }i\text{-th spot}) = \left( \text{tuple }\obx \text{ with entries } \frac {x}{\omega_i} \text{ on }\oi\text{-th spot} \right)
			\]
			
		\end{claim}
		
		\medskip
		
		\noindent We will now prove Claim~\ref{claim:iso}. To prove that the map \eqref{eqn:claim iso} is simultaneously well-defined and injective, we must show that any $F \in \oCS_{-\obn}$ has the property that
		\[
		F \in J(\obpsi)_{\obx} \quad \Leftrightarrow \quad F\in J(\bpsi)_{\bx}, \forall \bx \in \fold^{-1}(\obx)
		\]
		In other words, we have
		\begin{equation}
			\label{eqn:folded residue}
			\underset{z_n = \ox_n}{\text{Res}} \cdots \underset{z_1 = \ox_1}{\text{Res}} \frac {F (z_1,\dots,z_n)(\text{any monomial})}{\prod_{1\leq a < b \leq n} \ozeta_{\oi_b \oi_a} \left(\frac {z_b}{z_a} \right)} \prod_{a=1}^n \opsi_{\oi_a}(z_a) = 0  
		\end{equation}
		for all orderings $\ox_1,\dots,\ox_n$ of $\obx$ (with compatible ordering $\oi_1,\dots,\oi_n$ of $\obn$) if and only if 
		\begin{equation}
			\label{eqn:unfolded residue}
			\underset{z_n = x_n}{\text{Res}} \cdots \underset{z_1 = x_1}{\text{Res}} \frac {F (z_1,\dots,z_n)(\text{any monomial})}{\prod_{1\leq a < b \leq n} \zeta_{i_b i_a} \left(\frac {z_b}{z_a} \right)} \prod_{a=1}^n \psi_{i_a}(z_a) = 0  
		\end{equation}
		for all orderings $x_1,\dots,x_n$ of any $\bx \in \fold^{-1}(\obx)$ (with compatible ordering $i_1,\dots,i_n$ of some $\bn \in \fold^{-1}(\obn)$). The aforementioned ``if and only if" claim follows from the observation that we can have a non-trivial residue in \eqref{eqn:folded residue} only if for all $b \in \{1,\dots,n\}$, the scalar $\ox_b$ is either
		
		\medskip 
		
		\begin{itemize}[leftmargin=*]
			
			\item a pole of $\opsi_{\oi_b}(z_b)$, so by \eqref{eqn:acceptable covering} we have 
			\[
			\ox_b = \frac {x_b}{\omega_{i_b}}
			\]
			with $x_b \in P_{i_b}$ for a unique (due to \eqref{eqn:disjoint covering}) vertex $i_b \in \fold^{-1}(\oi_b)$, or
			
			\medskip 
			
			\item a zero of $\ozeta_{\oi_b\oi_a} \left(\frac {z_b}{z_a} \right)$ for some $a < b$, which must be of the form
			\[
			\ox_b = \ox_a q^{-1}_{\oalpha}
			\]
			for some $\oalpha : \oi_b \rightarrow \oi_a$. By \eqref{eqn:assumption}, there is a unique lift $\alpha : i_b \rightarrow i_a$ of $\oalpha$, for a unique $i_b \in \fold^{-1}(\oi_b)$. Plugging formula \eqref{eqn:matching parameters} into the display above, we have
			\[
			\ox_b = \ox_a q^{-1}_\alpha \frac {\omega_{i_a}}{\omega_{i_b}}
			\]
			By induction, we have $\ox_a = \frac {x_a}{\omega_{i_a}}$ with $x_a \in P_{i_a}$. If we define
			\[
			\ox_b = \frac {x_b}{\omega_{i_b}}
			\]
			then \eqref{eqn:condition covering} implies that $x_b \in P_{i_b}$.
			
		\end{itemize}
		
		\noindent Thus, any residue \eqref{eqn:folded residue} corresponds to a unique residue \eqref{eqn:unfolded residue}. The fact that the vanishing of these two residues is equivalent boils down to the fact that
		\[
		\frac {\ozeta_{\oi_b \oi_a}(x)}{\zeta_{i_bi_a} \left(\frac {\omega_{i_b}}{\omega_{i_a}} x \right)} = \prod_{i_b \neq j \in \rho^{-1}(\oi_b)} \zeta_{ji_a}\left(\frac {\omega_j}{\omega_{i_a}}x\right) \quad \text{and} \quad \frac {\psi_{i_b}(\omega_{i_b} y)}{\opsi_{\oi_b}(y)} = \prod_{i_b \neq j \in \rho^{-1}(\oi_b)} \frac 1{\psi_j(\omega_j y)}
		\]
		do not vanish at $x \in \frac {\omega_{i_a}P_{i_b}}{\omega_{i_b}P_{i_a}}$ and $y \in \frac {P_{i_b}}{\omega_{i_b}}$, respectively. This is because of \eqref{eqn:condition covering} and
		\[
		\frac {P_j}{\omega_j} \cap \frac {P_{i_b}}{\omega_{i_b}} = \varnothing
		\]
		for all $i_b \neq j \in \rho^{-1}(\oi_b)$. The discussion above establishes that the map \eqref{eqn:claim iso} is well-defined and injective. Moreover, it shows that for any fixed $\obx$ there is a unique
		\[
		\bn \in \rho^{-1}(\obn) \quad \text{and} \quad \bx \in \fold^{-1}(\obx) \subset (\BK^*)^{\bn} 
		\]
		which produces a non-zero quotient in the codomain of the map \eqref{eqn:claim iso}. Thus, it remains to show that 
		\begin{equation}
			\label{eqn:remains to show}
			\CS_{-\bn} = \iota_{\obn,\bn} \left(\oCS_{-\obn} \right) + (\CS_{-\bn}  \cap J(\bpsi)_{\bx})
		\end{equation}
		for any $\bx \in (\BK^*)^{\bn}$, in order to conclude the surjectivity of the map \eqref{eqn:claim iso}.
		
		\medskip 
		
		\begin{claim}
			\label{claim:commutative algebra}
			
			Consider rings $R = \BK[z_1^{\pm 1},\dots,z_n^{\pm 1}]^{\fS} \supset \BK[z_1^{\pm 1},\dots,z_n^{\pm 1}]^{\ofS} = \oR$, defined with respect to parabolic subgroups
			\[
			\fS \subset \ofS \subset \fS_n
			\]
			and any ideal $J \subset R$ that is supported at a single point $x \in \emph{Spec}(R)$ such that $g(x) \neq x$ for all $g \in \ofS \backslash \fS$. Then for any ideal $I \subset R$ such that $\Sigma(I) \subset I$, we have 
			\begin{equation}
				\label{eqn:commutative algebra}
				I = \Sigma(I) + (I \cap J)
			\end{equation}
			where 
			\[
			\Sigma : R \rightarrow \oR \subset R, f(z_1,\dots,z_n) \mapsto \sum_{g \in \ofS/\fS} g(f(z_1,\dots,z_n)) = \frac{1}{|\fS|} \sum_{g \in \ofS} g(f(z_1,\dots,z_n))
			\] 
			denotes the map of averaging with respect to $\ofS$. 
			
		\end{claim}
		
		\medskip 
		
		\noindent We will prove Claim~\ref{claim:commutative algebra} after we show how to use it in order to obtain \eqref{eqn:remains to show}. Indeed, let us invoke the Claim for  $R = \CV_{-\bn}$, $\oR = \oCV_{-\obn}$, $J = J(\bpsi)_{\bx}$ and
		\[
		I \subset \CV_{-\bn}
		\]
		defined to be the linear span of elements
		\begin{equation}
			\label{eqn:spherical 4}
			\text{Sym}_{\bn} \left[z_{\oi_1\bullet_1}^{d_1} \cdots z_{\oi_n\bullet_n}^{d_n} \prod_{1\leq a < b \leq n} \ozeta_{\oi_b\oi_a} \left(\frac {z_{\oi_b \bullet_b}}{z_{\oi_a \bullet_a}} \right) \right] 
		\end{equation}
		(compare with \eqref{eqn:spherical}-\eqref{eqn:spherical 2}). Formula \eqref{eqn:commutative algebra} states that
		\begin{equation}
			\label{eqn:1}
			I = \iota_{\obn,\bn} \left(\oCS_{-\obn} \right) + (I \cap J(\bpsi)_{\bx})
		\end{equation}
		While it is clear that $I \subseteq \CS_{-\bn}$ because $\zeta_{ij}\left(\frac {\omega_i}{\omega_j} x\right)$ divides $\ozeta_{\oi\oj}(x)$, we also claim that
		\begin{equation}
			\label{eqn:2}
			\CS_{-\bn} = I + (\CS_{-\bn} \cap J(\bpsi)_{\bx})
		\end{equation}
		Indeed, the equation above merely says that
		\begin{equation}
			\label{eqn:3}
			I \equiv \CS_{-\bn} \text{ modulo }J(\bpsi)_{\bx}
		\end{equation}
		or equivalently, that $\left[ I \right] = \left[\CS_{-\bn} \right]$ in the quotient ring
		\begin{equation}
			\label{eqn:quotient ring}
			\CV_{-\bn}/J(\bpsi)_{\bx}
		\end{equation}
		Formula \eqref{eqn:3} is a consequence of the fact that the Laurent polynomial
		\[
		\frac {\ozeta_{\oi_b \oi_a}(x)}{\zeta_{i_bi_a} \left(\frac {\omega_{i_b}}{\omega_{i_a}}x\right)} = \prod_{i_b \neq j \in \rho^{-1}(\oi_b)} \zeta_{ji_a} \left( \frac {\omega_j}{\omega_{i_a}} x\right) 
		\]
		does not vanish at any $x \in \frac {\omega_{i_a} P_{i_b}}{\omega_{i_b} P_{i_a}}$ (by \eqref{eqn:condition covering} and \eqref{eqn:disjoint covering}), and is thus an invertible element of the quotient ring \eqref{eqn:quotient ring}. Therefore, the right-hand sides of \eqref{eqn:spherical 2} and \eqref{eqn:spherical 4} span the same linear subspace of the quotient ring \eqref{eqn:quotient ring}. This establishes \eqref{eqn:3}, which is equivalent to \eqref{eqn:2}. Formulas \eqref{eqn:1} and \eqref{eqn:2} imply formula \eqref{eqn:remains to show}, and with it the proof of Theorem~\ref{thm:main covering}. \end{proof}
	
	\medskip
	
	\begin{proof} \emph{of Claim~\ref{claim:commutative algebra}:}
		Let $\mathfrak{m}_{x}$ be the maximal ideal of $R$ supported at a single point $x \in \text{Spec}(R)$. It suffices to show that 
		\[
		I = \Sigma(I) + (I \cap \mathfrak{m}_x^N)
		\]
		for all $N \in \mathbb{N}^*$, or equivalently, that the composition of linear maps
		\begin{equation}
			\label{eqn:composition}
			\Sigma(I) \hookrightarrow I \twoheadrightarrow I/(I \cap \mathfrak{m}_x^N)
		\end{equation}
		is surjective. Let $\tilde{x} \in (\BK^*)^n$ be lift of $x$, i.e. an ordering of its coordinates. Denote by $\mathfrak{n}_{\tilde{x}} \subset \BK[z_1^{\pm 1},\dots,z_n^{\pm 1}]$ the maximal ideal supported at $\tilde{x}$. Then $\mathfrak{m}_x = \mathfrak{n}_{\tilde{x}} \cap R$.
		
		\medskip
		
		Since any $f \in R$ is invariant under $\fS$ by definition, we have
		\begin{equation}
			\label{eqn: map Sigma is surjective}
			\Sigma(f) -f = \frac{1}{|\fS|} \sum_{g \in \ofS} g(f) - f =\frac{1}{|\fS|} \sum_{g \in \ofS \setminus \fS} g(f)
		\end{equation}
		For fixed $N$ as in \eqref{eqn:composition}, we claim that
		\begin{equation}
			\label{eqn:uno}
			(\mathfrak{n}_{\tilde{x}}^{K} \cap R) \subset (\mathfrak{n}_{\tilde{x}} \cap R)^N = \mathfrak{m}_x^N
		\end{equation}
		for $K = Nn$. We will explain why this is so in the case $\fS = \fS_n$, as the case when $\fS$ is an arbitrary parabolic subgroup of $\fS_n$ is analogous (but more notationally heavy): any element of $R$ is a polynomial expression in $y_i = e_i(z_1,\dots,z_n) - e_i(\tilde{x}_1,\dots \tilde{x}_n)$, which must have degree $\geq \frac Kn = N$ in the $y_i$'s if such an element is to lie in $\mathfrak{n}_{\tilde{x}}^{K}$.
		
		\medskip
		
		Let $\{x_1,\dots,x_r\} \subset \mathrm{Spec}(R)$ be the set of distinct $\fS$-orbits in $\ofS x$. In particular, let $x_1 =x$. Because the ideals $\fm_{x_i}$ are supported at distinct points, the Chinese remainder theorem implies that there exists an element
		\begin{equation}
			\label{eqn:due}
			e_x \in \left(1+\fm_x^N\right) \cap \bigcap_{i=2}^r \fm_{x_i}^K \subset R
		\end{equation}
		By assumption, for any $g \in \ofS$ we have $g(\tilde{x}) \in \fS\tilde{x}$ if and only if $g \in \fS$. Under the embedding $R \subset \BK[z_1^{\pm 1},\dots,z_n^{\pm 1}]$, formula \eqref{eqn:due} implies that we have $e_x \in \mathfrak{n}_y^K$ for all $y \in \ofS\tilde{x} \setminus \fS\tilde{x}$. Thus, for any $g \in \ofS \setminus \fS$, we have 
		\[
		g(e_x) \in \mathfrak{n}_{\tilde{x}}^K
		\]
		To complete the proof of Claim~\ref{claim:commutative algebra}, let $a \in I$ be arbitrary and so $\Sigma(e_xa) \in \Sigma(I) \subset I$. By the above argument, we have $g(e_xa) = g(e_x)g(a) \in \mathfrak{n}_{\tilde{x}}^K$ for all $g \in \ofS \setminus \fS$. This suggests that when $f=e_xa$, each term in the sum \eqref{eqn: map Sigma is surjective} lies in $\mathfrak{n}_{\tilde{x}}^K$. Therefore, 
		\[
		\Sigma(e_xa) - e_xa \in \mathfrak{n}_{\tilde{x}}^K \cap R \subset \mathfrak{m}_x^N
		\]
		by \eqref{eqn:uno}. Together with the fact that $e_xa-a = (e_x-1)a \in \mathfrak{m}_x^N$ (see \eqref{eqn:due}), we have $\Sigma(e_xa) -a \in \mathfrak{m}_x^N$. Therefore, the map $\Sigma(I) \to I/(I \cap \mathfrak{m}_x^N)$ of \eqref{eqn:composition} is surjective. \end{proof}
	
	\medskip 
	
	\subsection{Analogies}
	\label{sub:analogies}
	
	Having proven Theorem~\ref{thm:main covering}, we will now turn to the proof of our main Theorem~\ref{thm:main folding}. The two results are quite analogous to each other, and rather than repeat the entire proof, we will point out the ways in which the argument of Theorem~\ref{thm:main covering} must be adapted in order to yield Theorem~\ref{thm:main folding}. Thus, in what follows, the notation $\zeta,\psi,\ozeta,\opsi$ will refer to folded and unfolded zeta functions, as in Subsection~\ref{sub:folding q-characters}. The analysis which follows the statement of Claim~\ref{claim:iso} must be adapted as follows: we can have a non-trivial residue in \eqref{eqn:folded residue} only if for all $b \in \{1,\dots,n\}$, the variable $z_b$ is set equal to a scalar $\ox_b$ which arises either from
	
	\medskip 
	
	\begin{itemize}[leftmargin=*]
		
		\item a pole of $\opsi_{\oi_b}(z_b)$, so by \eqref{eqn:acceptable folding} we have 
		\[
		\ox_b \in P_{\sigma^{\ell_b}(i_{b,0})}\omega^{-\ell_b}
		\]
		where the value of $\ell_b \in \BZ/m\BZ$ is uniquely determined due to \eqref{eqn:disjoint folding} (above, $i_{b,0}$ denotes the distinguished preimage of $\oi_b$). Thus,
		\[
		i_b = \sigma^{\ell_b}(i_{b,0}) \in I
		\]
		is also uniquely determined, and we associate both $\ell_b$ and $i_b$ to the variable $z_b$;
		
		\medskip 
		
		\item a zero of $\ozeta_{\oi_b\oi_a} \left(\frac {z_b}{z_a} \right)$ for some $a < b$, so of the form
		\[
		\ox_b = \ox_a \omega^{-k} q^{-1}_{\alpha}
		\]
		for some arrow $\alpha : \sigma^k(i_{b,0}) \rightarrow i_{a,0}$. By induction, we have $\ox_a \in P_{i_a} \omega^{-\ell_a}$ where $i_a = \sigma^{\ell_a}(i_{a,0})$ has already been associated to $z_a$. Therefore, if we let $\ell_b = k+\ell_a \in \BZ/m\BZ$ and $i_b = \sigma^{\ell_b}(i_{b,0}) \in I$, then \eqref{eqn:condition folding} implies that
		\[
		\ox_b \in P_{i_b} \omega^{-\ell_b}
		\]
		We associate $\ell_b$ and $i_b$ to $z_b$, and notice that $\alpha: i_b \rightarrow i_a$.
		
	\end{itemize}
	
	\medskip
	
	\noindent The bullets above show that any residue of the form \eqref{eqn:folded residue} corresponds to a unique residue of the form \eqref{eqn:unfolded residue}, where in the latter, we consider the scalars $x_b = \ox_b \omega^{\ell_b} \in \BK^*$ and the vertices $i_b \in I$ which were defined in the bulleted list above. With this in mind, the argument of Theorem~\ref{thm:main covering} carries on. The only significant difference is that in \eqref{eqn:remains to show}, the map $\iota_{\bn,\obn}$ of \eqref{eqn:iota} must be replaced by the map
	\[
	\iota_{\bn,\obn} : \oCV_{-\obn} = \BK\left[z_{\oi 1}^{m_{\oi}}, z_{\oi 1}^{-m_{\oi}}, \dots z_{\oi n_{\oi}}^{m_{\oi}}, z_{\oi n_{\oi}}^{-m_{\oi}} \right]_{\oi \in \oI}^{\fS_{\obn}} \hookrightarrow \BK\left[z_{i 1}, z_{i 1}^{-1}, \dots z_{in_i}, z_{in_i}^{-1} \right]_{i \in I}^{\fS_{\bn}} = \CV_{-\bn}
	\]
	of \eqref{eqn:iota folding}, which we showed Subsection~\ref{sub:technical} to be none other than the embedding
	\[
	\BK\left[z_{i 1}, z_{i 1}^{-1}, \dots z_{in_i}, z_{in_i}^{-1} \right]_{i \in I}^{G_{\bm,\bn} \rtimes \fS_{\obn}} \hookrightarrow \BK\left[z_{i 1}, z_{i 1}^{-1}, \dots z_{in_i}, z_{in_i}^{-1} \right]_{i \in I}^{\fS_{\bn}}
	\]
	With this in mind, one needs to modify Claim~\ref{claim:commutative algebra} by replacing the parabolic subgroup $\ofS$ with the semidirect product $G \rtimes \ofS$, where $G$ is a finite abelian group that multiplies the variables $x_1,\dots,x_n$ independently by various roots of unity. The reason why the proof of Claim~\ref{claim:commutative algebra} carries on under this more restrictive situation is that assumption \eqref{eqn:disjoint folding} still guarantees the property
	\[
	g(\tilde{x}) \in \fS\tilde{x} \quad \Leftrightarrow \quad g \in \fS
	\]
	for any $g \in G \rtimes \ofS$ and any $\tilde{x}$ which can arise from ordering the coordinates of the support point of the ideal $J(\obpsi)_{\obx}$.
	However, once we make this modification, the proof of the aforementioned claim is essentially unchanged. 
	
	\bigskip
	
	\appendix
	
	\section{Proof of twisted Drinfeld-Serre relations}
	\label{appendix1}
	
	\medskip
	
	We will now prove Proposition~\ref{prop:affine}, following the arguments laid out in \cite{N Reduced, N Arbitrary, NSS}. This will involve a case-by-case analysis of the possible interactions between the various orbits of the automorphism $\sigma$ of finite and affine Dynkin diagrams.
	
	\medskip 
	
	\begin{proof} \emph{of Proposition~\ref{prop:affine}:} As explained in the paragraph immediately after the statement of the Proposition, our strategy is to prove that
		\begin{equation}
			\label{eqn:strategy}
			\Big \langle \text{LHS of \eqref{eqn:cubic serre 1}-\eqref{eqn:cubic serre 3 in A2n}}, F \Big \rangle = 0 \quad \Leftrightarrow \quad \Big(F \text{ satisfies \eqref{eqn:twisted wheel strongly}} \Big) 
		\end{equation}
		Since the wheel conditions \eqref{eqn:twisted wheel strongly} depend on two vertices $i \neq j \in I$, we will prove the above equivalence on a case-by-case basis depending on the $\sigma$-orbits of $i$ and $j$. Note that in order to keep our formulas manageable, we will work with the primed zeta functions \eqref{eqn:zeta prime strongly} and \eqref{eqn:zeta prime strongly twisted}, in both the twisted and untwisted cases. 
		
		\medskip
		
		\subsection*{\emph{Case 1}: \texorpdfstring{$i \sim j$}{i sim j} and there are no other vertices in the orbit of \texorpdfstring{$i$}{i} connected to \texorpdfstring{$j$}{j}. Assume that \texorpdfstring{$i$}{i} is not one of the middle two nodes in type \texorpdfstring{$A_{2n}$}{A2n}, which is treated in Cases 3 and 4. }

		\subsubsection*{\emph{Case 1.1}}
		Let's start with the case $\sigma(i) \neq i$. We have 
		\[
		\ozeta'_{ij}(x) \sim \ozeta'_{ji}(x) \sim 1 - xq^{-1}, \quad \ozeta'_{ii}(x) \sim \frac{1 - xq^2}{1-x}
		\]
		In particular, the zeta functions and wheel conditions are exactly the same as in untwisted type $A_2$, for which the equivalence \eqref{eqn:strategy} was proved in \cite{NSS}. For completeness, we now provide an alternative proof. Consider the cubic expression
		\begin{equation*}
			\begin{split}
				C(x_1,x_2,y) & = (x_1 - yq^{-1})e_i(x_1)e_i(x_2)e_j(y)\\ 
				& + (x_2q^{-1}-x_1q)e_i(x_2)e_j(y)e_i(x_1) \\
				&+ (yq-x_2)e_j(y)e_i(x_1)e_i(x_2)
			\end{split}
		\end{equation*}
		and the quadratic expressions
		\[
		Q(x,y) = e_i(x)e_j(y)(x-yq^{-1}) - e_j(y)e_i(x)(xq^{-1}-y)
		\]
		\[
		Q'(x_1,x_2) = e_i(x_1)e_i(x_2)(x_1q^{-2}-x_2) - e_i(x_2)e_i(x_1)(x_1-x_2q^{-2})
		\]
		We have $Q(x,y) =0$ and $Q'(x_1,x_2) =0$ on both sides of the isomorphism in Proposition~\ref{prop:affine}, because of relation \eqref{eqn:rel prime quantum 2 twisted strongly}. Moreover, the equation \footnote{Specifically, we replace the factor $\frac{x_1q_2\cdots q_{\ell}}{x_{\ell}}$ in $P_{\ell}$ in the expression $e_{\bigcirc}$ (see \eqref{eqn:clunky}) with $\frac{x_{\ell+1}}{x_2q_3\cdots q_{\ell+1}}$. These two expressions are equivalent because we are simply using a different expansion of delta functions here, and both are dual to the wheel condition.} $e_{\bigcirc} = 0$ with $\bigcirc = \{z_{i1} = z_{i2}q^2, z_{i2}=z_{j1}q^{-1} \}$ takes the particular form $C(x_1,x_2,y) = 0$, 
		and
		\begin{equation}
			\label{eqn:cf}
			\Big \langle C(x_1,x_2,y), F \Big \rangle = 0 \quad \Leftrightarrow \quad F \Big |_{z_{i1} \mapsto z_{j1}q, z_{i2} \mapsto z_{j1}q^{-1}} = 0
		\end{equation}
		was proved in \cite{NSS} (the fact that the right-hand side does not depend on $x_1$ and $x_2$ is due to the fact that $(x_2-yq^{-1})C(x_1,x_2,y) = (x_1-yq)C(x_1,x_2,y) = 0$ modulo the quadratic expressions $Q(x_1,y),Q(x_2,y),Q'(x_1,x_2)$). We need to show that under the condition $Q(x_1,y) = Q(x_2,y) = Q'(x_1,x_2) =0$, the relation $C(x_1,x_2,y) = 0$ is equivalent to the Drinfeld-Serre relation \eqref{eqn:cubic serre 1}. To this end, consider the expression
		\begin{equation}
			\label{eqn:big s}
			\begin{split}
				S &= e_i(x_2)Q(x_1,y)x_2 + Q(x_2,y)e_i(x_1)x_2q - C(x_1,x_2,y)x_2q^2\\
				&+ e_i(x_1)Q(x_2,y)x_1 + Q(x_1,y)e_i(x_2)x_1q - C(x_2,x_1,y)x_1q^2\\
				&+Q'(x_1,x_2)e_j(y)yq + e_j(y)Q'(x_1,x_2)(-yq^3 +x_1q^2 + x_2q^2)
			\end{split}
		\end{equation}
		A direct computation shows that 
		\[
		S = \Big(\text{LHS of \eqref{eqn:cubic serre 1}}\Big) \times x_1x_2(1-q^2)
		\] 
		Therefore, the Drinfeld-Serre relation \eqref{eqn:cubic serre 1} holds if and only if $S = 0$. We conclude
		\[
		C(x_1,x_2,y) = C(x_2,x_1,y) = 0 \quad \Rightarrow \quad  \eqref{eqn:cubic serre 1} \text{ holds} 
		\]
		On the other hand, if \eqref{eqn:cubic serre 1} holds, then $S = 0$ and so \eqref{eqn:big s} implies that
		\[
		C(x_1,x_2,y)x_2 + C(x_2,x_1,y)x_1 = 0
		\]
		By a direct calculation using \eqref{eqn:pair formula} that is akin to the proof of \eqref{eqn:cf}, we have
		\begin{equation*}
			\begin{split}
				\Big \langle C(x_1,&x_2,y)x_2  + C(x_2,x_1,y)x_1, F \Big \rangle = 0 \quad \Leftrightarrow \\
				&\left[z_{i2}\delta \left(\frac{z_{i1}}{z_{i2}q^2} \right) \delta\left(\frac{z_{i2}q}{z_{j1}} \right) + z_{i1}\delta \left(\frac{z_{i2}}{z_{i1}q^2}\right) \delta \left(\frac{z_{i1}q}{z_{j1}} \right)\right] F(z_{i1},z_{i2},z_{j1}) = 0
			\end{split}
		\end{equation*}
		Because the two products of delta functions are supported on disjoint loci, we have 
		\begin{equation*}
			\begin{split}
				\Big \langle C(x_1,x_2,y)x_2  &+ C(x_2,x_1,y)x_1  , F \Big \rangle = 0  \quad \Leftrightarrow \\
				&F \Big |_{z_{i1} \mapsto z_{i2} q^{2}, z_{j1} \mapsto z_{i2} q} =  F \Big |_{z_{i2} \mapsto z_{i1} q^{2}, z_{j1} \mapsto z_{i1} q} = 0
			\end{split}
		\end{equation*}
		and so $S=0$ implies $C(x_1,x_2,y) = C(x_2,x_1,y) = 0$ (again using \eqref{eqn:cf}. \footnote{Alternatively, one can derive $C(x_1,x_2,y)$ from $S$ using the formula
			\begin{equation*}
				\begin{split}
					C(x_1,x_2,y) = &e_i(x_1)Q(x_2,y)(q^2+1) + e_i(x_2)Q(x_1,y)q^2 -Q(x_2,y)e_i(x_1)q \\ 
					&+Q'(x_1,x_2)e_j(y)q^2 + S(x_1,x_2,y)\frac{qy-x_2}{(1-q^2)x_1x_2}.
				\end{split}
			\end{equation*}
			(note the importance of $q$ not being a root of unity).}
		
		\medskip
		
		\subsubsection*{\emph{Case 1.2}}
		Consider the case $\sigma(i) = i$. We have
		\[
		\ozeta'_{ij}(x) \sim \ozeta'_{ji}(x) \sim 1-q^{-m}x^m, \quad \ozeta'_{ii}(x) \sim \frac{1 - q^{2m}x^m}{1-x^m}
		\]
		while the wheel condition states that
		\[
		F \Big |_{z_{i1} \mapsto z_{j1} q, z_{i2} \mapsto z_{j1} q^{-1}} = 0
		\]
		However, as $F$ is in fact a Laurent polynomial of $z_{i1}^m$ and $z_{i2}^m$, the condition above is equivalent to 
		\[
		F \Big |_{z_{i1}^m \mapsto z_{j1}^m q^{-m}, z_{i2}^m \mapsto z_{j1}^m q^m} = 0
		\]
		The argument is then identical to the previous case, with all variables being replaced by their $m$-th powers.
		
		\medskip 
		
		\subsection*{\emph{Case 2}: if \texorpdfstring{$i \sim j$}{i sim j} and there are exactly \texorpdfstring{$m$}{m} vertices in the orbit of \texorpdfstring{$i$}{i} connected to \texorpdfstring{$j$}{j}.}
		
		\subsubsection*{\emph{Case 2.1:} \texorpdfstring{$m=2$}{m=2}}
		We have the zeta functions
		\[
		\ozeta'_{ij}(x) \sim \ozeta'_{ji}(x) \sim 1-q^{-2}x^2, \quad \ozeta'_{ii}(x) \sim \frac{1 - q^{2}x}{1-x}
		\]
		Since shuffle elements $F$ are actually Laurent polynomials in $z_{j1}^2$, we have the equivalence
		\[
		F \Big |_{z_{i1} \mapsto z_{i2}q^2, z_{j1} \mapsto z_{i2}q} = 0 \quad \Leftrightarrow \quad
		F \Big |_{z_{i1} \mapsto z_{i2} q^2, z_{j1}^2 \mapsto z_{i2}^2q^2} = 0
		\]
		Following \cite[Proposition~3.5]{NSS}, we have the equality of formal series
		\begin{equation}
			\label{eqn:delta 2.1}
			\begin{split}
				\delta \left(\frac{x_1}{x_2q^2} \right) \delta\left(\frac{x_2^2q^2}{y^2}\right)& = \text{ev}_{|x_1| \gg |x_2| \gg |y|} \left[ \frac{1}{(1-\frac{x_2q^2}{x_1})(1-\frac{y^2}{x_2^2q^2})} \right] \\
				& + \text{ev}_{|x_2| \gg |y| \gg |x_1|} \left[ \frac{\frac{x_1^2}{y^2q^2}(1+\frac{x_2q^2}{x_1})}{(1-\frac{y^2}{x_2^2q^2})(1-\frac{x_1^2}{y^2q^2})} \right] \\
				& + \text{ev}_{|y| \gg |x_1| \gg |x_2|} \left[ \frac{\frac{x_2^2q^2}{y^2}}{(1-\frac{x_2q^2}{x_1})(1-\frac{x_1^2}{y^2q^2})} \right]
			\end{split}
		\end{equation}
		where ``ev" refers to the power series expansion of the rational function in question, determined by the appropriate order of the variables. Consider the cubic expression
		\begin{equation*}
			\begin{split}
				C(x_1,x_2,y) & = (x_1^2 - y^2q^{-2})e_i(x_1)e_i(x_2)e_j(y)\\ 
				& + (x_1q^{-1} + x_2q)(x_2q^{-1}-x_1q)e_i(x_2)e_j(y)e_i(x_1) \\
				&+ (y^2q^2-x_2^2)e_j(y)e_i(x_1)e_i(x_2)
			\end{split}
		\end{equation*}
		and the quadratic expressions
		\[
		Q(x,y) = e_i(x)e_j(y)(x^2-y^2q^{-2}) - e_j(y)e_i(x)(x^2q^{-2}-y^2)
		\]
		\[
		Q'(x_1,x_2) = e_i(x_1)e_i(x_2)(x_1q^{-2}-x_2) - e_i(x_2)e_i(x_1)(x_1-x_2q^{-2})
		\]
		Following \cite[Proposition~3.5]{NSS}, formula \eqref{eqn:delta 2.1} implies
		\begin{equation}
			\label{eqn:cf 2}
			\Big \langle C(x_1,x_2,y), F \Big \rangle = 0 \qquad \Leftrightarrow \quad F \Big |_{z_{i1} \mapsto z_{i2} q^2, z_{j1}^2 \mapsto z_{j2}^2q^2} = 0
		\end{equation}
		As before, we need to show that under the condition that $Q(x_1,y) = Q(x_2,y) = Q'(x_1,x_2) =0$, the relation $C(x_1,x_2,y) = 0$ is equivalent to the Drinfeld-Serre relation \eqref{eqn:cubic serre 3}. To this end, we consider the expression
		\begin{equation*}
			\begin{split}
				S & =  e_i(x_2)Q(x_1,y)x_2 +  Q(x_2,y)e_i(x_1)x_2 q^2 - C(x_1,x_2,y)x_2q^2 \\
				& + e_i(x_1)Q(x_2,y) x_1 + Q(x_1,y)e_i(x_2)x_1q^2 - C(x_2,x_1,y)x_1q^2 \\
				& +Q'(x_1,x_2)e_j(y) \left(y^2+\frac{2x_1x_2}{1+q^{-4}} \right) \\
				& + e_j(y)Q'(x_1,x_2) \left(-y^2q^4+x_1^2q^2+x_2^2q^2 + \frac{2x_1x_2}{1+q^{-4}}\right)
			\end{split}
		\end{equation*}
		Then 
		\[
		S = \Big(\text{LHS of \eqref{eqn:cubic serre 3}}\Big) \times x_1x_2 \frac{q^{-2}-q^{2}}{q^{2}+q^{-2}}
		\]
		and we conclude that the Drinfeld-Serre relation \eqref{eqn:cubic serre 3} holds if and only if $S = 0$. Assuming the quadratic relations $Q(x_1,y) = Q(x_2,y) = Q'(x_1,x_2) =0$, we have 
		\[
		S = 0 \quad \Leftrightarrow \quad
		C(x_1,x_2,y)x_2  + C(x_2,x_1,y)x_1 = 0
		\]
		By a direct calculation using \eqref{eqn:pair formula} and \eqref{eqn:delta 2.1}, we have
		\begin{equation*}
			\begin{split}
				\Big \langle C(x_1,&x_2,y)x_2  + C(x_2,x_1,y)x_1  , F \Big \rangle = 0 \quad \Leftrightarrow \\
				&\left[z_{i2}\delta \left(\frac{z_{i1}}{z_{i2}q^2} \right) \delta\left(\frac{z_{i2}^2q^2}{z_{j1}^2} \right) + z_{i1}\delta \left(\frac{z_{i2}}{z_{i1}q^2}\right) \delta\left(\frac{z_{i1}^2q^2}{z_{j1}^2} \right)\right]F(z_{i1},z_{i2},z_{j1}) = 0
			\end{split}
		\end{equation*}
		Because the two products of delta functions are supported on disjoint loci, we have
		\begin{equation*}
			\begin{split}
				\Big \langle C(x_1,x_2,y)x_2  &+ C(x_2,x_1,y)x_1  , F \Big \rangle = 0  \quad \Leftrightarrow \\
				&F \Big |_{z_{i1} \mapsto z_{i2} q^{2}, z_{j1}^2 \mapsto z_{i2}^2 q^2} =  F \Big |_{z_{i2} \mapsto z_{i1} q^{2}, z_{j1}^2 \mapsto z_{i1}^2 q^2} = 0
			\end{split}
		\end{equation*}
		Since $F$ is symmetric in $z_{i1}$ and $z_{i2}$, the condition above is equivalent to \eqref{eqn:cf 2}, thus establishing the equivalence of $C(x_1,x_2,y) = 0$ and $C(x_1,x_2,y)x_2 + C(x_2,x_1,y)x_1  = 0$, the latter of which is equivalent to $S=0$ and the Drinfeld-Serre relations.
		
		\medskip
		
		\subsubsection*{\emph{Case 2.2}: \texorpdfstring{$m=3$}{m=3}}
		By analogy with the discussion above, we have 
		\begin{equation*}
			\begin{split}
				\delta \left(\frac{x_1}{x_2q^2} \right)\delta \left(\frac{x_2^3q^3}{y^3}\right)& = \text{ev}_{|x_1| \gg |x_2| \gg |y|} \left[ \frac{1}{(1-\frac{x_2q^2}{x_1})(1-\frac{y^3}{x_2^3q^3})} \right] \\
				& + \text{ev}_{|x_2| \gg |y| \gg |x_1|} \left[ \frac{\frac{x_1^3}{y^3q^3}(1+\frac{x_2q^2}{x_1}+\frac{x_2^2q^4}{x_1^2})}{(1-\frac{y^3}{x_2^3q^3})(1-\frac{x_1^3}{y^3q^3})} \right] \\
				& + \text{ev}_{|y| \gg |x_1| \gg |x_2|} \left[ \frac{\frac{x_2^3q^3}{y^3}}{(1-\frac{x_2q^2}{x_1})(1-\frac{x_1^3}{y^3q^3})} \right]
			\end{split}
		\end{equation*}
		and consider correspondingly the cubic expression
		\begin{equation*}
			\begin{split}
				C(x_1,x_2,y) & = (x_1^3 - y^3q^{-3})e_i(x_1)e_i(x_2)e_j(y) \\
				& + (x_2^2q^2 + x_1x_2 + x_1^2q^{-2})(x_2q^{-1}-x_1q) e_i(x_2)e_j(y)e_i(x_1) \\
				& + (y^3q^3 - x_2^3)e_j(y)e_i(x_1)e_i(x_2)
			\end{split}
		\end{equation*}
		and the quadratic expressions
		\[
		Q(x,y) = e_i(x)e_j(y)(x^3 - y^3q^{-3}) - e_j(y)e_i(x)(x^3q^{-3}-y^3)
		\]
		\[
		Q'(x_1,x_2) = e_i(x_1)e_i(x_2)(x_1q^{-2}-x_2) - e_i(x_2)e_i(x_1)(x_1 - x_2q^{-2})
		\]
		If we let
		\begin{equation*}
			\begin{split}
				S & =  e_i(x_2)Q(x_1,y)x_2 +  Q(x_2,y)e_i(x_1)x_2 q^3 - C(x_1,x_2,y)x_2q^2 \\
				& + e_i(x_1)Q(x_2,y) x_1 + Q(x_1,y)e_i(x_2)x_1q^3 - C(x_2,x_1,y)x_1q^2 \\
				& +Q'(x_1,x_2)e_j(y)(q^{-1}y^3 + q^4\frac{x_1^2x_2+x_1x_2^2}{q^4-q^2+1}) \\
				&+ e_j(y)Q'(x_1,x_2)(q^2(x_1^3+x_2^3) - q^5y^3 + q^4\frac{x_1^2x_2+x_1x_2^2}{q^4-q^2+1})
			\end{split}
		\end{equation*}
		then we have
		\[
		S = \Big(\text{LHS of \eqref{eqn:cubic serre 3}}\Big) \times x_1x_2\frac{1-q^4}{1+q^6}
		\]
		The remainder of the argument is identical to the $m=2$ case, so we skip it.
		
		\medskip
		
		\subsection*{\emph{Case 3}: \texorpdfstring{$i=n$}{i=n}, \texorpdfstring{$j=n+1$}{j=n+1} in type \texorpdfstring{$A_{2n}$}{A2n}}
		
		In this case, we have
		\[
		\ozeta'(x) \sim \frac{(1-q^{2}x)(1+q^{-1}x)}{1-x}
		\]
		where above and henceforth we suppress the indices $i,j$ of $\ozeta'$ due to the fact that they have the same image in the folded quiver. The wheel condition reads
		\[
		F \Big |_{z_{1} \mapsto - z_{3} q, z_{2} \mapsto -z_{3} q^{-1}} = 0
		\]
		Consider the cubic expression
		\begin{equation*}
			\begin{split}
				C(x_1,x_2,x_3) & = (x_1 - x_3q^{2})(x_1 + x_3q^{-1})(x_1+x_2q^{-1})(x_2-x_3q^{2})e(x_1)e(x_2)e(x_3)  \\ 
				&- q^{-1}(x_2 - x_1q^{2})(x_2 + x_1q^{-1})(x_2 -x_3q^{2})(x_3-x_1q^{2})e(x_2)e(x_3)e(x_1)  \\ 
				&- q(x_3 - x_2q^{2})(x_3 + x_2q^{-1})(x_3-x_1q^{2})(x_1+x_2q^{-1})e(x_3)e(x_1)e(x_2)
			\end{split}
		\end{equation*}
		and the quadratic expression
		\[
		Q(x,y) = e(x)e(y)(x+yq^{-1})(x-yq^{2}) - e(y)e(x)(xq^{-1}+y)(xq^{2}-y)
		\]
		Consider
		\begin{equation*}
			\begin{split}
				S = &\sum_{\sigma \in \mathfrak{S}_3} C(x_{\sigma(1)},x_{\sigma(2)},x_{\sigma(3)})\frac{1}{q-1+q^{-1}} \\
				&+ \sum_{cyclic} Q(x_1,x_2)e(x_3)(x_1x_2+qx_1x_3+qx_2x_3-q^2(q+1+q^{-1})x_3^2) \\
				&- \sum_{cyclic}e(x_1)Q(x_2,x_3)((q+1+q^{-1})x_1^2-qx_1x_2-qx_1x_3-q^2x_2x_3)
			\end{split}
		\end{equation*}
		It is straightforward to check that
		\[
		S = \Big(\text{LHS of \eqref{eqn:cubic serre 1 in A2n}}\Big) \times (q^{1/2}-q^{-1/2})(q+1)^3x_1x_2x_3
		\]
		Therefore, the Drinfeld-Serre relation \eqref{eqn:cubic serre 1 in A2n} holds if and only if 
		\[
		\sum_{\sigma \in \mathfrak{S}_3} C(x_{\sigma(1)},x_{\sigma(2)},x_{\sigma(3)}) = 0
		\]
		As above, the reason that the above cubic relation is dual to the wheel condition $F |_{z_{\sigma(1)} \mapsto - z_{\sigma(3)} q, z_{\sigma(2)} \mapsto -z_{\sigma(3)} q^{-1}} = 0$ is the delta function identity
		\begin{equation*}
			\begin{split}
				\delta \left(\frac{x_{\sigma(1)}}{x_{\sigma(2)}q^2} \right) \delta\left(\frac{-x_{\sigma(2)}q}{x_{\sigma(3)}}\right)& = \text{ev}_{|x_{\sigma(1)}| \gg |x_{\sigma(2)}| \gg |x_{\sigma(3)}|} \left[ \frac{1}{(1-\frac{x_{\sigma(2)}q^2}{x_{\sigma(1)}})(1+\frac{x_{\sigma(3)}}{x_{\sigma(2)}q})} \right] \\
				& + \text{ev}_{|x_{\sigma(2)}| \gg |x_{\sigma(3)}| \gg |x_{\sigma(1)}|} \left[ \frac{\frac{x_{\sigma(1)}}{-x_{\sigma(3)}q}}{(1+\frac{x_{\sigma(3)}}{x_{\sigma(2)}q})(1+\frac{x_{\sigma(1)}}{x_{\sigma(3)}q})} \right] \\
				& + \text{ev}_{|x_{\sigma(3)}| \gg |x_{\sigma(1)}| \gg |x_{\sigma(2)}|} \left[ \frac{\frac{-x_{\sigma(2)}q}{x_{\sigma(3)}}}{(1-\frac{x_{\sigma(2)}q^2}{x_{\sigma(1)}})(1-\frac{x_{\sigma(1)}}{x_{\sigma(3)}q^2})} \right]
			\end{split}
		\end{equation*}
		Therefore, 
		\begin{equation*}
			\begin{split}
				\Big \langle \sum_{\sigma \in \mathfrak{S}_3} &C(x_{\sigma(1)},x_{\sigma(2)},x_{\sigma(3)}) , F \Big \rangle = 0 \quad \Leftrightarrow \\
				&\sum_{\sigma \in \mathfrak{S}_3} \delta \left(\frac{z_{\sigma(1)}}{z_{\sigma(2)}q^2} \right) \delta\left(\frac{-z_{\sigma(2)}q}{z_{\sigma(3)}}\right)  F(z_{\sigma(1)},z_{\sigma(2)},z_{\sigma(3)}) = 0
			\end{split}
		\end{equation*}
		Because $\delta \left(\frac{z_{\sigma(1)}}{z_{\sigma(2)}q^2} \right) \delta\left(\frac{-z_{\sigma(2)}q}{z_{\sigma(3)}}\right)$ are supported on pairwise disjoint loci, we have 
		\begin{equation*}
			\begin{split}
				\Big \langle \sum_{\sigma \in \mathfrak{S}_3} &C(x_{\sigma(1)},x_{\sigma(2)},x_{\sigma(3)}) , F \Big \rangle = 0 \quad \Leftrightarrow  \\
				&F \Big |_{z_{\sigma(1)} \mapsto - z_{\sigma(3)} q, z_{\sigma(2)} \mapsto -z_{\sigma(3)} q^{-1}} = 0, \forall \sigma \in \mathfrak{S}_3
			\end{split}
		\end{equation*}
		Thus $S = 0$ implies $C(x_1,x_2,x_3) = 0$. The remainder of the argument is identical to the one in the previous case.
		
		\medskip
		
		\subsection*{\emph{Case 4}: \texorpdfstring{$i=n$}{i=n}, \texorpdfstring{$j=n-1$}{j=n-1} in type \texorpdfstring{$A_{2n}$}{A2n}} 
		We have
		\[
		\ozeta'_{ij}(x) \sim \ozeta'_{ji}(x) \sim 1-q^{-1}x, \quad \ozeta'_{ii}(x) \sim \frac{(1-q^{2}x)(1+q^{-1}x)}{1-x}
		\]
		We consider the cubic expression
		\begin{equation*}
			\begin{split}
				C(x_1,x_2,y) &= (x_1 - yq^{-1})(x_1+x_2q^{-1}) e_i(x_1)e_i(x_2)e_j(y) \\
				&+ (x_2q^{-1} - x_1q)(x_2+x_1q^{-1})e_i(x_2)e_j(y)e_i(x_1)  \\ 
				&+ (yq-x_2)(x_1+x_2q^{-1})e_j(y)e_i(x_1)e_i(x_2)
			\end{split}
		\end{equation*}
		and the quadratic expressions
		\[
		Q(x,y) = e_i(x)e_j(y)(x-yq^{-1}) - e_j(y)e_i(x)(xq^{-1}-y)
		\]
		and
		\[
		Q'(x_1,x_2) = e_i(x_1)e_i(x_2)(x_1-x_2q^2)(x_1+x_2q^{-1}) - e_i(x_2)e_i(x_1)(x_1q^2-x_2)(x_1q^{-1}+x_2)
		\]
		We have 
		\[
		\Big(\text{LHS of \eqref{eqn:cubic serre 3 in A2n}}\Big) = S \times x_1x_2(q^{-1}-q)
		\]
		where
		\begin{equation*}
			\begin{split}
				S &= e_i(x_2)Q(x_1,y)(x_1x_2q^{-1}+x_2^2) + Q(x_2,y)e_i(x_1)(x_1x_2+x_2^2q) - C(x_1,x_2,y)x_2q^2\\
				&+ e_i(x_1)Q(x_2,y)(x_1x_2q^{-1}+x_1^2) + Q(x_1,y)e_i(x_2)(x_1x_2+x_1^2q) - C(x_2,x_1,y)x_1q^2\\
				&+Q'(x_1,x_2)e_j(y)yq^{-1} + e_j(y)Q'(x_1,x_2)(-yq +x_1 + x_2)
			\end{split} 
		\end{equation*}
		The remainder of the argument follows the one in Case 1 quite closely, so we omit it. \end{proof}
	
	\section{Affine Dynkin diagrams with automorphisms}
	\label{appendix2}
	We consider all the non-trivial automorphisms of affine Dynkin diagrams (except for affine type $A_{n-1}^{(1)}$), which are represented in the following pictures. 
	
	\medskip
	
	
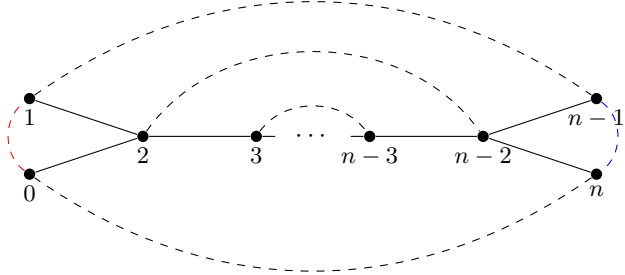
\[
	\begin{gathered}
		\begin{tikzpicture}[
			vertex/.style={circle, fill=black, inner sep=1.5pt},
			lab/.style={below=3pt, font=\small}
			]
			
			\node[vertex, label={[lab]above:$0$}] (0) at (0,-0.5) {};
			\node[vertex, label={[lab]above:$1$}] (1) at (0,0.5) {};
			\node[vertex, label={[lab]above:$2$}] (2) at (1.5,0) {};
			\node[vertex, label={[lab]above:$3$}] (3) at (3,0) {};
			\node[vertex, label={[lab]above:$n-1$}] (n-1) at (4.5,0) {};
			\node[vertex, label={[lab]above:$n$}] (n) at (6,0) {};
			
			\draw (0) -- (2);
			\draw (1) -- (2);
			\draw (2) -- (3);
			\draw (3) -- (3.25,0);
			\node at (3.75,0) {$\cdots$};
			\draw (4.25,0) -- (n-1);
			
			\draw ($(n-1)+(0,0.04)$) -- ($(n)+(0,0.04)$);
			\draw ($(n-1)-(0,0.04)$) -- ($(n)-(0,0.04)$);
			\draw[-] ($(n)-(0.4,0.15)$) -- ($(n)-(0.03,0)$);
			\draw[-] ($(n)-(0.4,-0.15)$) -- ($(n)-(0.03,0)$);
			
			\draw[dashed, -, bend left=55] (0) to (1);
			\draw[dashed, -] (2) to[out=45, in=135, looseness=20] (2);
			\draw[dashed, -] (3) to[out=45, in=135, looseness=20] (3);
			\draw[dashed, -] (n-1) to[out=45, in=135, looseness=20] (n-1);
			\draw[dashed, -] (n) to[out=45, in=135, looseness=20] (n);
		\end{tikzpicture}
		\\[4pt]
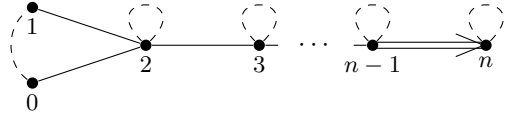
		\text{Type }B_n^{(1)} \, (n \ge 3)
	\end{gathered}
	\]
	
	\[
	\begin{gathered}
		\begin{tikzpicture}[
			vertex/.style={circle, fill=black, inner sep=1.5pt},
			lab/.style={below=3pt, font=\small}
			]
			
			\node[vertex, label={[lab]above:$0$}] (0) at (0,0) {};
			\node[vertex, label={[lab]above:$1$}] (1) at (1.5,0) {};
			\node[vertex, label={[lab]above:$n-1$}] (n-1) at (4.5,0) {};
			\node[vertex, label={[lab]above:$n$}] (n) at (6,0) {};

			\draw (1) -- (2,0);
			\node at (3,0) {$\cdots$};
			\draw (4,0) -- (n-1);
			\draw ($(n-1)+(0,0.04)$) -- ($(n)+(0,0.04)$);
			\draw ($(n-1)-(0,0.04)$) -- ($(n)-(0,0.04)$);
			\draw[-] ($(1)-(0.4,0.15)$) -- ($(1)-(0.03,0)$);
			\draw[-] ($(1)-(0.4,-0.15)$) -- ($(1)-(0.03,0)$);
			\draw ($(0)+(0,0.04)$) -- ($(1)+(0,0.04)$);
			\draw ($(0)-(0,0.04)$) -- ($(1)-(0,0.04)$);
			\draw[-] ($(n-1)+(0.4,0.15)$) -- ($(n-1)+(0.03,0)$);
			\draw[-] ($(n-1)+(0.4,-0.15)$) -- ($(n-1)+(0.03,0)$);
			
			\draw[dashed, -, bend left=55] (0) to (n);
			\draw[dashed, -, bend left=55] (1) to (n-1);
		\end{tikzpicture}
		\\[4pt]
		\text{Type }C_{n}^{(1)} \, (n \ge 2)
	\end{gathered}
	\]
	
	\[
	\begin{gathered}
		\begin{tikzpicture}[
			vertex/.style={circle, fill=black, inner sep=1.5pt},
			lab/.style={below=3pt, font=\small}
			]
			
			\node[vertex, label={[lab]above:$0$}] (0) at (0,-0.5) {};
			\node[vertex, label={[lab]above:$1$}] (1) at (0,0.5) {};
			\node[vertex, label={[lab]above:$2$}] (2) at (1.5,0) {};
			\node[vertex, label={[lab]above:$3$}] (3) at (3,0) {};
			\node[vertex, label={[lab]above:$n-3$}] (n-3) at (4.5,0) {};
			\node[vertex, label={[lab]above:$n-2$}] (n-2) at (6,0) {};
			\node[vertex, label={[lab]above:$n-1$}] (n-1) at (7.5,0.5) {};
			\node[vertex, label={[lab]above:$n$}] (n) at (7.5,-0.5) {};
			
			\draw (0) -- (2);
			\draw (1) -- (2);
			\draw (2) -- (3);
			\draw (3) -- (3.25,0);
			\node at (3.75,0) {$\cdots$};
			\draw (4.25,0) -- (n-3);
			\draw (n-3) -- (n-2);
			\draw (n-2) -- (n-1);
			\draw (n-2) -- (n);
			
			\draw[dashed, red, -, bend left=55] (0) to (1);
			\draw[dashed, blue, -, bend left=55] (n-1) to (n);
			\draw[dashed, -, bend left=35] (n) to (0);
			\draw[dashed, -, bend left=35] (1) to (n-1);
			\draw[dashed, -, bend left=55] (2) to (n-2);
			\draw[dashed, -, bend left=55] (3) to (n-3);
			
		\end{tikzpicture}
		\\[4pt]
		\text{Type }D_n^{(1)} (n \ge 4), \text{the picture above has a dihedral group $D_8$ group symmetry,}
		\\
		\text{generated by the three involutions depicted in black, red and blue above.}
	\end{gathered}
	\]
	
	
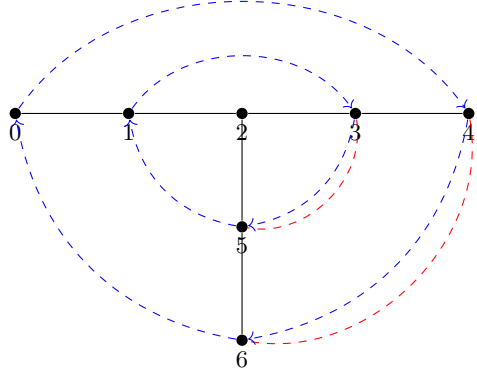
\[
	\begin{gathered}
		\begin{tikzpicture}[
			vertex/.style={circle, fill=black, inner sep=1.5pt},
			lab/.style={below=3pt, font=\small}
			]
			
			\node[vertex, label={[lab]above:$0$}] (0) at (0,0) {};
			\node[vertex, label={[lab]above:$1$}] (1) at (1.5,0) {};
			\node[vertex, label={[lab]above:$2$}] (2) at (3,0) {};
			\node[vertex, label={[lab]above:$3$}] (3) at (4.5,0) {};
			\node[vertex, label={[lab]above:$4$}] (4) at (6,0) {};
			\node[vertex, label={[lab]above:$5$}] (5) at (3,-1.5) {};
			\node[vertex, label={[lab]above:$6$}] (6) at (3,-3) {};
			
			\draw (0) -- (1);
			\draw (1) -- (2);
			\draw (2) -- (3);
			\draw (3) -- (4);
			\draw (2) -- (5);
			\draw (5) -- (6);
			
			\draw[dashed, red, -, bend left=55] (4) to (6);
			\draw[dashed, red, -, bend left=55] (3) to (5);
			
			\draw[dashed, blue, ->, bend left=55] (0) to (4);
			\draw[dashed, blue, ->, bend left=55] (1) to (3);
			\draw[dashed, blue, ->, bend left=35] (3) to (5);
			\draw[dashed, blue, ->, bend left=35] (4) to (6);
			\draw[dashed, blue, ->, bend left=35] (5) to (1);
			\draw[dashed, blue, ->, bend left=35] (6) to (0);
			
		\end{tikzpicture}
		\\[4pt]
		\text{Type }E_6^{(1)}, \text{there are three order $2$ automorphisms (depicted in red in the}
		\\
		\text{picture above) and two order $3$ automorphisms (depicted in blue).}
	\end{gathered}
	\]

	\[
	\begin{gathered}
		\begin{tikzpicture}[
			vertex/.style={circle, fill=black, inner sep=1.5pt},
			lab/.style={below=3pt, font=\small}
			]
			
			\node[vertex, label={[lab]above:$0$}] (0) at (0,0) {};
			\node[vertex, label={[lab]above:$1$}] (1) at (1.5,0) {};
			\node[vertex, label={[lab]above:$2$}] (2) at (3,0) {};
			\node[vertex, label={[lab]above:$3$}] (3) at (4.5,0) {};
			\node[vertex, label={[lab]above:$4$}] (4) at (6,0) {};
			\node[vertex, label={[lab]above:$5$}] (5) at (7.5,0) {};
			\node[vertex, label={[lab]above:$6$}] (6) at (9,0) {};
			\node[vertex, label={[lab]above:$7$}] (7) at (4.5,-1.5) {};
			
			\draw (0) -- (1);
			\draw (1) -- (2);
			\draw (2) -- (3);
			\draw (3) -- (4);
			\draw (4) -- (5);
			\draw (5) -- (6);
			\draw (3) -- (7);
			
			\draw[dashed, -, bend left=55] (0) to (6);
			\draw[dashed, -, bend left=55] (1) to (5);
			\draw[dashed, -, bend left=55] (2) to (4);
			\draw[dashed, -] (3) to[out=45, in=135, looseness=20] (3);
			\draw[dashed, -] (7) to[out=-45, in=45, looseness=20] (7);
			
		\end{tikzpicture}
		\\[4pt]
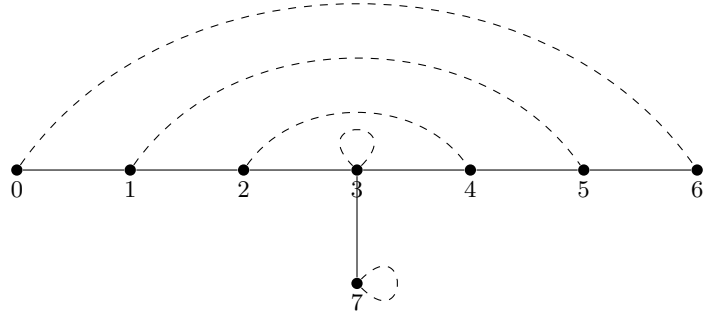
		\text{Type }E_7^{(1)}
	\end{gathered}
	\]
	
	\[
	\begin{gathered}
		\begin{tikzpicture}[
			vertex/.style={circle, fill=black, inner sep=1.5pt},
			lab/.style={below=3pt, font=\small}
			]
			
			\node[vertex, label={[lab]above:$0$}] (0) at (0,-0.5) {};
			\node[vertex, label={[lab]above:$1$}] (1) at (0,0.5) {};
			\node[vertex, label={[lab]above:$2$}] (2) at (1.5,0) {};
			\node[vertex, label={[lab]above:$3$}] (3) at (3,0) {};
			\node[vertex, label={[lab]above:$n-1$}] (n-1) at (4.5,0) {};
			\node[vertex, label={[lab]above:$n$}] (n) at (6,0) {};
			
			\draw (0) -- (2);
			\draw (1) -- (2);
			\draw (2) -- (3);
			\draw (3) -- (3.25,0);
			\node at (3.75,0) {$\cdots$};
			\draw (4.25,0) -- (n-1);
			
			\draw ($(n-1)+(0,0.04)$) -- ($(n)+(0,0.04)$);
			\draw ($(n-1)-(0,0.04)$) -- ($(n)-(0,0.04)$);
			\draw[-] ($(n-1)+(0.4,0.15)$) -- ($(n-1)+(0.03,0)$);
			\draw[-] ($(n-1)+(0.4,-0.15)$) -- ($(n-1)+(0.03,0)$);
			
			\draw[dashed, -, bend left=55] (0) to (1);
			\draw[dashed, -] (2) to[out=45, in=135, looseness=20] (2);
			\draw[dashed, -] (3) to[out=45, in=135, looseness=20] (3);
			\draw[dashed, -] (n-1) to[out=45, in=135, looseness=20] (n-1);
			\draw[dashed, -] (n) to[out=45, in=135, looseness=20] (n);
		\end{tikzpicture}
		\\[4pt]
		\text{Type }A_{2n-1}^{(2)} \, (n \ge 3)
	\end{gathered}
	\]
	
	\[
	\begin{gathered}
		\begin{tikzpicture}[
			vertex/.style={circle, fill=black, inner sep=1.5pt},
			lab/.style={below=3pt, font=\small}
			]
			
			\node[vertex, label={[lab]above:$0$}] (0) at (0,0) {};
			\node[vertex, label={[lab]above:$1$}] (1) at (1.5,0) {};
			\node[vertex, label={[lab]above:$n-1$}] (n-1) at (4.5,0) {};
			\node[vertex, label={[lab]above:$n$}] (n) at (6,0) {};

			\draw (1) -- (2,0);
			\node at (3,0) {$\cdots$};
			\draw (4,0) -- (n-1);
			\draw ($(n-1)+(0,0.04)$) -- ($(n)+(0,0.04)$);
			\draw ($(n-1)-(0,0.04)$) -- ($(n)-(0,0.04)$);
			\draw[-] ($(n)-(0.4,0.15)$) -- ($(n)-(0.03,0)$);
			\draw[-] ($(n)-(0.4,-0.15)$) -- ($(n)-(0.03,0)$);
			\draw ($(0)+(0,0.04)$) -- ($(1)+(0,0.04)$);
			\draw ($(0)-(0,0.04)$) -- ($(1)-(0,0.04)$);
			\draw[-] ($(0)+(0.4,0.15)$) -- ($(0)+(0.03,0)$);
			\draw[-] ($(0)+(0.4,-0.15)$) -- ($(0)+(0.03,0)$);
			
			\draw[dashed, -, bend left=55] (0) to (n);
			\draw[dashed, -, bend left=55] (1) to (n-1);
		\end{tikzpicture}
		\\[4pt]
		\text{Type }D_{n+1}^{(2)} \, (n \ge 2)
	\end{gathered}
	\]

\end{document}